\documentclass{article}
\usepackage{amsmath, amsfonts, amssymb, mathrsfs, amsthm,accents, sectsty,hyphenat, palatino,calc, url}
\usepackage[all]{xy}
\usepackage[usenames]{color}
\def\mf#1{\mathfrak{#1}}
\def\mc#1{\mathcal{#1}}
\def\mb#1{\mathbb{#1}}
\def\tx#1{{\rm #1}}
\def\tb#1{\textbf{#1}}

\def\ts#1{\textsf{#1}}
\def\tr{\tx{tr}\,}
\def\R{\mathbb{R}}

\def\C{\mathbb{C}}
\def\Q{\mathbb{Q}}

\def\Z{\mathbb{Z}}

\def\lmod{\setminus}

\def\ol#1{\overline{#1}}

\def\rk{\tx{rk}}

\def\hat{\widehat}
\def\vol{\tx{vol}}

\def\rw{\rightarrow}
\def\lw{\leftarrow}

\def\lrw{\longrightarrow}

\def\lw{\leftarrow}
\def\sm{\smallsetminus}

\def\<{\langle}
\def\>{\rangle}

\def\mr#1{\mathring{#1}}

\newenvironment{mytitle}
{\begin{center}\large\sc}
{\end{center}}

\newtheorem{thm}{Theorem}[subsection]
\newtheorem{lem}[thm]{Lemma}
\newtheorem{pro}[thm]{Proposition}
\newtheorem{cnd}[thm]{Conditions}
\newtheorem{cor}[thm]{Corollary}

\newtheorem{fct}[thm]{Fact}

\sectionfont{\center\sc\normalsize}
\subsectionfont{\bf\normalsize}

\numberwithin{equation}{subsection}

\def\phi{\varphi}

\newlength{\cplxcorr}

\newlength{\sumcorr}
\def\ssum#1{\setlength{\sumcorr}{(\widthof{$\displaystyle\sum_{#1}$}-\widthof{$\displaystyle\sum$})/2} \hspace{-\sumcorr}\sum_{#1}\hspace{-\sumcorr} }

\begin{document}

\begin{mytitle} Epipelagic $L$-packets and rectifying characters \end{mytitle}
\begin{center} Tasho Kaletha \end{center}
\begin{abstract}We provide an explicit construction of the local Langlands correspondence for general tamely-ramified reductive p-adic groups and a class  of wildly ramified Langlands parameters. Furthermore, we verify that our construction satisfies the expected properties of such a  correspondence. More precisely, we show that each $L$-packet we construct admits a parameterization in terms of the Langlands dual group,  contains a unique generic element for a fixed Whittaker datum, satisfies the formal degree conjecture, is compatible with twists by central and cocentral characters, provides a  stable virtual character, and satisfies the expected endoscopic character identities. Moreover, we show that in the case of $\tx{GL}_n$,  our construction coincides with the one given by Bushnell and Henniart \cite{BH05a}, \cite{BH05b}. Our results suggest a general approach to the construction of the local Langlands correspondence for tamely-ramified groups and regular supercuspidal parameters.
\end{abstract}
\let\thefootnote\relax\footnotetext{This research is supported in part by NSF grant DMS-1161489.}

The purpose of this paper is to construct the local Langlands correspondence for the epipelagic representations of tamely-ramified reductive $p$-adic groups. These are irreducible supercuspidal representations of positive depth $\frac{1}{m}$, where $m$ is a certain integer. When the residual characteristic $p$ does not divide $m$, they appear in the general constructions of supercuspidal representations of Adler \cite{Ad98} and Yu \cite{Yu01}. More recently, their construction was reinterpreted by Reeder and Yu \cite{RY} in terms of Vinberg's theory of graded Lie algebras and this allowed for the restriction $p \nmid m$ to be removed. In the latter paper, these representations were called epipelagic. In terms of the Adler-Yu construction, an epipelagic representation of a tamely-ramified  connected reductive group $G$ arises from a pair $(S,\chi)$, where $S \subset G$ is a maximal torus, and $\chi : S(F) \rw \C^\times$ is a  character, and the pair is subject to certain conditions. These conditions imply in particular that the torus $S$ is tamely-ramified, but not unramified. In fact, the ramification degree of the splitting field of $S$ is equal to $m$, and we are assuming that $p \nmid 2m$.

Our construction of the local Langlands correspondence is a generalization of the approach originally used by Langlands \cite{Lan88} to  construct the discrete series $L$-packets for real groups, and later employed by DeBacker and Reeder \cite{DR09} to construct supercuspidal L-packets for certain tamely-ramified parameters and unramified groups. However, there are two interesting new phenomena occurring in the present situation which were visible neither in the setting of real groups, nor in the setting of tamely-ramified supercuspidal  parameters.

To describe the first phenomenon, we recall that the first step in Langlands' approach is the construction of an elliptic maximal torus  $S \subset G$ and a character $\chi : S(F) \rw \C^\times$ from the given Langlands parameter $\phi$. In fact, obtaining the maximal torus  is quite simple. It is the character $\chi$ that causes the trouble. Both in Langlands' original paper, as well as in the work of  DeBacker and Reeder, $\chi$ was obtained by an ad-hoc construction, which is independent of the particular Langlands-parameter at hand,  in that it is the same construction for all parameters of the same kind. In \cite{Kal11a}, the author gave a reinterpretation of the construction of DeBacker and Reeder using the work of Langlands and Shelstad \cite[\S\S5,6]{LS87} on L-embeddings of maximal tori into the L-group of G. The  upshot is that the Langlands parameter $\phi$ can be factored as the composition of a Langlands parameter $\phi_S : W_F \rw {^LS}$ for  the torus $S$ and an $L$-embedding $^Lj : {^LS} \rw {^LG}$. Moreover, it was shown in loc. cit. that in the setting of tamely-ramified  supercuspidal parameters, there is an essentially unique choice for $^Lj$, and that $\chi$ is the character corresponding to the  parameter $\phi_S$. The same uniqueness is also true in the case of real discrete series $L$-packets and is implicit in the work \cite {She10} of Shelstad. The new phenomenon occurring in the setting of epipelagic $L$-packets is that the $L$-embedding $^Lj$ ceases to be unique. In fact, when the maximal torus $S$ is tamely-ramified, there are always $2^n$ choices of tamely-ramified $L$-embeddings $^Lj$, for a certain number $n$ depending on $S$, and each two choices differ by a sign. The question is then to choose the correct signs. What is interesting is that there does not  appear to be a universal choice -- the signs depend on the arithmetic data encoded in the particular Langlands parameter $\phi$. That  this is the case for $\tx{GL}_n$ is visible in the work of Bushnell and Henniart \cite{BH05b}. Inspired by that work, we give a simple  and explicit construction of the $L$-embedding $^Lj : {^LS} \rw {^LG}$, which relies on the use of Langlands' $\lambda$-constants \cite[Thm. 2.1]{LanArt}. The role of choosing the  correct $L$-embedding in our construction is somewhat parallel to the role of the rectifying character in the work of Bushnell and  Henniart. This parallel is however not entirely direct, because unlike in the case of $\tx{GL}_n$, there is no ``naive'' construction in  our case. Rather, the construction can be performed only after a choice of an $L$-embedding has been made.

We now come to the second new phenomenon occurring in our construction. It concerns the grouping of representations into $L$-packets. Each  constituent of our $L$-packets is constructed, just like in the case of discrete series of real groups and of depth-zero supercuspidal  representations of p-adic groups, from a pair $(S',\chi')$ of a maximal torus $S' \subset G$ and a character $\chi' : S'(F) \rw \C^ \times$.  In the case of discrete series representations of real groups, as well as in the case of depth-zero supercuspidal  representations of unramified $p$-adic groups, an $L$-packet is formed by considering all $G(F)$-conjugacy classes of pairs $(S',\chi')$ in the stable conjugacy class of the pair $(S,\chi)$ that was obtained from the  Langlands parameter $\phi$. While we can clearly perform the same construction in our setting, the result will in general not be an $L$-packet, as one quickly sees by studying the character relations from the theory of endoscopy. The failure of these relations leads to the  definition and study of a new sign invariant associated to a maximal torus $S' \subset G$. This sign invariant takes the form of a function, which assigns to each element in the set of roots of $S'$ the number $+1$ or $-1$. Moreover, this sign invariant appears to be interesting in its own right. For example, it provides a refinement of the Kottwitz sign $e(G)$ associated to the reductive group $G$ \cite{Kot83}. The invariant can be used to construct a character $\epsilon_{S'} : S'(F) \rw \C^\times$. In order to obtain the correct $L$-packet, one has to take the representations corresponding to the pairs $(S',\chi' \cdot \epsilon_{S'})$, for all $G(F)$-conjugacy classes of pairs $(S',\chi')$ in the stable class of $(S,\chi)$. It may be worth pointing out that, contrary to the case  of the $L$-embedding $^Lj : {^LS} \rw {^LG}$, the characters $\epsilon_{S'}$ do not depend on the Langlands-parameter $\phi$ or the  character $\chi$ constructed from it. Rather, they only depend on the particular torus $S'$. We believe that  these characters will play a role in the local Langlands correspondence for other classes of parameters, especially when considering  supercuspidal representations which arise from tamely-ramified, but not unramified, tori.

Both of these phenomena came initially as a surprise to us. In \cite{Kal12}, we studied the $L$-packets of epipelagic representations of depth $\frac{1}{h}$ for split, absolutely simple, simply-connected groups $G$ with Coxeter number $h$. In that situation, neither of the two phenomena appears. We believe that it is the simply-connectedness that is chiefly responsible for this. For example, in the study of the group $\tx{GL}_n$ both phenomena need to be addressed.

We will now describe in more detail the individual results of this paper. After fixing notation in Section \ref{sec:notation} and reviewing some known results in \ref{sec:prelims} that will be used throughout the paper, we set off by defining and studying in Section \ref{sec:torinv} the sign invariant associated to a pair $(G,S)$ of a connected reductive group $G$ defined over a local field and a maximal torus $S$ of $G$. The invariant depends only on the $G(F)$-conjugacy class of $S$, but in general it changes when we pass to a different rational class within the stable class of $S$, and in particular when we transfer $S$ to an inner form of $G$. The initial definition we give of the invariant is straightforward and elementary, yet in order to study it we find it more convenient to provide a cohomological interpretation. With this interpretation at hand, we then give a formula for how the invariant varies within a given stable class, and from this we conclude that one can recover the Kottwitz sign $e(G)$ from the invariant of any elliptic maximal torus of $G$ (Proposition \ref{pro:torinvstab}). The main technical burden of this section is the proof of the vanishing result \ref{pro:torinvvan}, which is crucial for the proofs of stability and endoscopic transfer in the later sections.

Having established the necessary results on the toral invariant, we move in Section \ref{sec:packs} to the first main goal of this paper -- the construction of the local Langlands correspondence for epipelagic representations. In \cite[\S7.1]{RY} Reeder and Yu single out a class of Langlands parameters which they believe should correspond to the class of epipelagic representations of $G$ under the local Langlands correspondence, whenever $G$ is an absolutely simple and simply-connected tamely-ramified $p$-adic group. Their main motivation for this prediction is the formal degree conjecture of Hiraga-Ichino-Ikeda \cite{HII08} and its reformulation given in \cite{GR10}. In the paper at hand, we consider any tamely-ramified reductive $p$-adic group $G$ and a class of Langlands parameters for it which generalizes that for the absolutely-simple simply connected case. The precise conditions we impose on the parameters are Conditions \ref{cnd:parm} in Section \ref{sec:lpackconst}. To a Langlands parameter satisfying these conditions, we explicitly construct in Section \ref{sec:lpackconst} a finite set ($L$-packet) of epipelagic representations of $G$. Here is a brief summary of the construction: Let $F$ be a finite extension of the field $\Q_p$ of $p$-adic numbers, and let $W$ be its Weil group and $\Gamma$ its Galois group. Let $G$ be a connected reductive algebraic group defined over $F$ and split over a tamely-ramified extension of $F$. Let $\hat G$ be the complex Langlands dual group of $G$, and let $^LG = \hat G \rtimes W$ be the Weil-form of the $L$-group of $G$. A Langlands parameter $\phi : W \rw {^LG}$ satisfying Conditions \ref{cnd:parm} normalizes a unique maximal torus $\hat T$ of $\hat G$, and hence provides an action of $W$ on $\hat T$ which one easily shows extends to an action of $\Gamma$. The complex torus $\hat T$ with the new $\Gamma$-action is the complex dual of a torus $S$ defined over $F$. In Section \ref{sec:chispec} we construct an $L$-embedding $^Lj : {^LS} \rw {^LG}$ of the $L$-group of $S$ into the $L$-group of $G$. The construction of ${^Lj}$ is based on Langlands' $\lambda$-constants \cite[Thm. 2.1]{LanArt}. These constants are Gauss-sums formed from additive characters on certain finite fields, and the necessary additive characters are extracted from the parameter $\phi$. The $L$-embedding ${^Lj}$ provides a factorization $\phi = {^Lj} \circ \phi_S$, with $\phi_S : W \rw {^LS}$ a Langlands parameter for $S$. Let $\chi_S : S(F) \rw \C^\times$ be the character corresponding to $\phi_S$ under the local Langlands correspondence for tori. The torus $S$ comes equipped with a stable conjugacy class of embeddings $[j] : S \rw G$. The $L$-packet on $G$ corresponding to $\phi$ is then the set of all epipelagic representations arising from the pairs $j_*(S,\chi_S\cdot \epsilon_j)$, where $j$ runs over the set of $G(F)$-conjugacy classes in the stable conjugacy class $[j]$, and $\epsilon_j : S(F) \rw \C^\times$ is the character constructed in Section \ref{sec:torinvchar} from the toral invariant of $jS$. All the details of this construction are given in Section \ref{sec:lpackconst}, except for the construction of the $L$-embedding $^Lj$, which is given in Section \ref{sec:chispec}. The reason for this separation is that the arguments in \ref{sec:lpackconst} are fairly general and we believe they will apply with little or no modification to much more general classes of Langlands parameters. On the other hand, Section \ref{sec:chispec} is quite specific to the parameters at hand, and while we believe that a similar approach will work for other classes of parameters, the specific formulas will most likely be different.

It would be interesting to compare the material in our Section \ref{sec:lpackconst} with that in \cite[\S7.2,\S7.3]{RY}, where the authors consider the special case of an absolutely simple, simply-connected group $G$ and under the same assumptions on the residual characteristic of $F$ construct a Langlands parameter starting from an epipelagic representation. Their construction is quite different from ours, as it relies on the invariant theory of graded Lie algebras and moreover goes in the opposite direction. It would be interesting to see how our construction in Section \ref{sec:lpackconst}, when specialized to absolutely simple and simply-connected groups, relates to the one in \cite{RY}.

After the $L$-packets have been constructed, the next step, taken up in Section \ref{sec:lpackparm}, is to parameterize their constituents in terms of the centralizer $S_\phi = \tx{Cent}(\phi,\hat G)$. We do this using the language of representations of extended pure inner forms based on Kottwitz's theory of isocrystals with additional structure \cite{Kot85}, \cite{Kot97}. Thus, rather than considering an individual group $G$ and a Langlands parameter for it, we consider all extended pure inner forms $G^b$ of a given fixed quasi-split group $G$ and all representations of $G^b(F)$. The necessary notions from \cite{Kal11a} are recalled in Section \ref{sec:isoinner}. Crucial for the parameterization is the fact that our $L$-packets satisfy Shahidi's tempered packet conjecture \cite{Sha90} -- when the quasi-split group $G$ is endowed with a Whittaker datum, the $L$-packet on $G$ contains a unique generic constituent (Proposition \ref{pro:unigen}). The parameterization itself takes the form of a commutative diagram
\begin{equation} \label{eq:diagint} \xymatrix{
\tx{Irr}(S_\phi)\ar[d]\ar[r]^-\sim&\Pi_\phi\ar[d]\\
X^*(Z(\hat G)^\Gamma)\ar[r]^-\sim&\tb{B}(G)_\tx{bas}
} \end{equation}
The set $\Pi_\phi$ consists of equivalence classes of quadruples $(G^b,\xi,b,\pi)$, where $(\xi,b) : G \rw G^b$ is an extended pure inner twist of the fixed quasi-split group $G$, and $\pi$ is an epipelagic representation of $G^b(F)$. The set $\tx{Irr}(S_\phi)$ consists of the equivalence classes of irreducible algebraic representations of the complex algebraic group $S_\phi$. The left vertical arrow is given by taking central characters, while the right vertical arrow is given by sending $(G^b,\xi,b)$ to $b$. The bottom horizontal arrow is Kottwitz's isomorphism \cite[Prop. 5.6]{Kot85}. The top vertical arrow is the bijection we construct.

The remainder of Section \ref{sec:packs} is devoted to the proof of two conjectural properties for our $L$-packets -- the formal degree conjecture of Hiraga, Ichino, and Ikeda \cite{HII08}, and the compatibility with central and cocentral characters \cite[\S10]{Bo77}. In the course of proving the latter in Section \ref{sec:centchar} we provide a new construction of the character of $G(F)$ associated to a parameter $W \rw Z(\hat G)$, and of the character of $Z_G(F)$ associated to a parameter $W \rw {^LG}$. The original constructions of these objects \cite[\S10]{Bo77}required auxiliary choices. We use the cohomology of crossed modules and the cohomological pairings for complexes of tori of length two from the work of Kottwitz and Shelstad \cite[App. A]{KS99} to provide canonical constructions of these objects which do not rely on auxiliary choices.

The second main goal of this paper is to prove that our $L$-packets satisfy the conjectural character relations of the theory of endoscopy. The first main result in this direction is Theorem \ref{thm:stab}: For any inner twist $\xi : G \rw G'$ of the quasi-split group $G$, put
\[ S\Theta_{\phi,\xi,G'} = e(G')\sum_\pi \Theta_\pi, \]
where the sum runs over the constituents of the $L$-packet on $G'$ for the parameter $\phi$. Then Theorem \ref{thm:stab} asserts that and if $\gamma \in G(F)$ and $\gamma' \in G'(F)$ are related strongly-regular semi-simple elements, then
\[ S\Theta_{\phi,\xi,G'}(\gamma') = S\Theta_{\phi,\tx{id},G}(\gamma).\]
This statement contains as a special case the assertion that the function $S\Theta_{\phi,\xi,G'}$ is stable. The second main result is the endoscopic transfer. Let $(H,s,{^L\eta})$ be an extended endoscopic triple for $G$. In particular $^L\eta$ is an $L$-embedding ${^LH} \rw {^LG}$. While this is not the most general case of endoscopy, one can easily reduce to it by passing to an extension of $G$ whose derived group is simply connected, so it will in fact suffice to treat this case. Fixing a Whittaker datum for $G$, there exists for each extended pure inner twist $(\xi,b) : G \rw G^b$ a canonical normalization $\Delta$ of the Langlands-Shelstad absolute transfer factor \cite[\S2]{Kal11a}. Let $\phi : W \rw {^LG}$ be an epipelagic Langlands parameter which factors through ${^L\eta}$ as $\phi = {^L\eta}\circ{\phi_H}$. We then have the $L$-packet $\Pi_{\phi_H,\tx{id},H}$ on $H(F)$ and define the endoscopic lift to $G^b$ of the stable character $S\Theta_{\phi_H,\tx{id},H}$ by the formula
\[ \tx{Lift}^{G^b}_H S\Theta_{\phi_H,\tx{id},H}(\gamma^b) = \sum_{\gamma^H \in H(F)_\tx{sr}/\tx{st}} \Delta(\gamma^H,\gamma^b)\frac{D^H(\gamma^H)^2}{D^{G^b}(\gamma^b)^2}S\Theta_{\phi^H,\tx{id},H}(\gamma^H). \]
On the other hand, we define the $s$-stable character of the packet $\Pi_{\phi,b}$ on $G^b(F)$ by
\[ \Theta_{\phi,b}^s = e(G^b)\sum_{\substack{\rho \in \tx{Irr}(S_\phi)\\ \rho \mapsto b}} \tx{tr}\rho(s) \Theta_{\pi_\rho}. \]
Here the map $\rho \mapsto \pi_\rho$ is the top horizontal map in the commutative square \eqref{eq:diagint}, and the map $\rho \mapsto b$ is the composition of the left vertical and lower horizontal maps. Then Theorem \ref{thm:endotran} asserts that
\[ \tx{Lift}^{G^b}_H S\Theta_{\phi_H,\tx{id},H}(\gamma^b) = \Theta_{\phi,b}^s(\gamma^b), \]
provided we assume that the residual characteristic of $F$ is not too small.

The proof of these two theorems is based on the recent work of Adler-Spice \cite{AS10} on the characters of supercuspidal representations, the works of Waldspurger \cite{Wal97}, \cite{Wal06}, and Ngo \cite{Ngo10} on endoscopy for $p$-adic Lie\-algebras, as well as a recent result of Kottwitz on the comparison of Weil constants and epsilon factors. The restriction on the residual characteristic comes from the fact that the proof uses a suitable extension of the locally defined logarithm map. It may be possible to remove, or at least significantly weaken, this restriction via the use of the maps $\ts{e}_x$ used by Adler-Spice, since their character formulas, as well as all other ingredients in the proof, are in fact valid without the restrictions we impose. We have not tried to pursue this.

The last main result of this paper is the comparison of our construction with the construction of Bushnell and Henniart for $\tx{GL}_n(F)$ in Section \ref{sec:epipackgln}. This comparison shows (Theorem \ref{thm:recti}) that our construction provides the same result as theirs for the group $\tx{GL}_n(F)$, and, since we have already proved the transfer to inner forms, also for the group $\tx{GL}_n(D)$ for any division algebra $D$. The main part of the comparison is to show that the $L$-embedding we construct in Section \ref{sec:chispec} differs from the ``traditional'' $L$-embedding that may be used in the special case of $\tx{GL}_n$ precisely by the rectifying character of \cite[Thm 2.1]{BH05b}. While the use of Langlands' $\lambda$-constants in Section \ref{sec:chispec} was inspired by the work of Bushnell and Henniart, this comparison result is still not entirely obvious, the reason being that our constructions apply to a general group and thus cannot reference the ``traditional'' $L$-embedding available for $\tx{GL}_n$, and moreover our constructions have a very different structure from that of the rectifying character for $\tx{GL}_n$.

\tb{Acknowledgements:} This work has profited greatly from enlightening and inspiring mathematical conversations with Benedict Gross, Guy Henniart, Atsushi Ichino, Robert Kottwitz, Mark Reeder, and Loren Spice. It is a pleasure to thank these mathematicians for their interest and support.

\tableofcontents

\section{Notation} \label{sec:notation}
Throughout the paper, $F$ will denote a $p$-adic field, i.e. a finite extension of the field $\Q_p$ of $p$-adic numbers. We will write $O_F$ for the ring of integers, $\mf{p}_F$ for the maximal ideal, $k_F$ for the residue field -- a finite field of cardinality $q$ and characteristic $p$. We fix an algebraic closure $\ol{F}$ of $F$ and denote by $F^u$ the maximal unramified extension of $F$ contained in $\ol{F}$ of $F$. Let $\Gamma_F,W_F,I_F,P_F$ denote the Galois, Weil, inertia, and wild inertia groups of $\ol{F}/F$. When the field is clear from the context, we may drop the subscript $F$. On the other hand, if $E/F$ is a finite extension, we will denote by $\Gamma(E/F)$ or $\Gamma_{E/F}$ the relative Galois group, and will use the analogous notation for the other groups associated to the extension $E/F$.

The letters $G,H,J$ will often denote algebraic groups, while the Fraktur letters $\mf{g},\mf{h},\mf{j}$ will denote their Lie-algebras, and $\mf{g}^*,\mf{h}^*,\mf{j}^*$ their duals. Given an algebraic group $G$, we will write $Z_G$ for its center, and $A_G$ for the maximal split torus inside of $Z_G$. For an element $\gamma \in G$, we will write $\tx{Cent}(\gamma,G)$ or $G^\gamma$ for the centralizer of $\gamma$ in $G$, and $G_\gamma$ for its connected component. Given a ring $R$ over which the algebraic group $G$ is defined, we will write $G(R)$ for the set of $R$-points of $G$. It will sometimes be convenient to reuse the symbol $G$ also for the set of points of $G$ over an algebraically closed field that is understood from the context (e.g. it may be $\ol{F}$ if $G$ is defined over $F$).

If $G$ is a connected reductive group defined over $F$, we will denote by $\mc{B}^\tx{red}(G,F)$ the reduced Bruhat-Tits building of $G(F)$. Given $x \in \mc{B}^\tx{red}(G,F)$ we will write $G(F)_x$ for the stabilizer of $x$ for the action of $G(F)$ on $\mc{B}^\tx{red}(G,F)$, and given a non-negative real number $r$ we will write $G(F)_{x,r}$ and $G(F)_{x,r+}$ for the corresponding Moy-Prasad filtration subgroups \cite{MP96} of $G(F)_x$. In particular, the parahoric subgroup at $x$ is $G(F)_{x,0}$, and its pro-unipotent radical is $G(F)_{x,0+}$. We will also use the notation $G(F)_{x,r:s}$ to denote the quotient $G(F)_{x,r}/G(F)_{x,s}$, as was done in \cite{Yu01}. Similar notation will be used for the Moy-Prasad lattices inside the Lie-algebra $\mf{g}(F)$.

We will make frequent use of local class field theory and of Langlands' correspondence for tori over local fields. We normalize these correspondences following Langlands' article \cite{Lan97}. In particular, uniformizers of non\-archimedean local fields will correspond to the Frobenius automorphisms of their maximal unramified extension.

\section{Preliminaries} \label{sec:prelims}

The purpose of this section is to recall some material from \cite{RY}, \cite{GLRY}, and \cite{Yu01}, which is relevant to our situation. Let $G$ be a connected reductive group defined over $F$. We assume that $G$ splits over a tamely ramified extension of $F$. Let $\mf{g}$ be the Lie-algebra of $G$.

\subsection{Parahoric subgroups, Moy-Prasad filtrations, and invariant theory} \label{sec:parahorics}

Let $y \in \mc{B}^\tx{red}(G,F)$ be a rational point of order $e$ (\cite[\S3.3]{RY}). Assume that $p \nmid e$ and let $E/F^u$ be the smallest extension of order $e$. Then the group $G$ splits over $E$ and $y$ is a hyperspecial point in $\mc{B}^\tx{red}(G,E)$ \cite[\S4.2]{RY}. Choose a uniformizer $\omega \in E$. This choice provides a character
\[ \zeta : \Gamma_{E/F^u} \rw \mu_e(F^u),\qquad \sigma \mapsto \omega^{-1}\sigma(\omega), \]
where $\mu_e(F^u)$ denotes the subgroup of roots of unity in $F^u$ of order $e$.

Let $\ts{G}_y=G(E)_{y,0:0+}$ and $\ts{g}_y=\mf{g}(E)_{y,0:0+}$. Then $\ts{G}_y$ is a connected reductive group defined over $\ol{k_F}$ and $\ts{g}_y$ is its Lie algebra. For each $r \in \frac{1}{e}\Z$, multiplication by  $\omega^{-er}$ provides  a $\ts{G}_y$-equivariant isomorphism
\[ \mf{g}(E)_{y,r:r+} \rw \ts{g}_y. \]
The point $y$ is preserved by $\Gamma_{E/F^u}$, and we have $[\mf{g}(E)_{y,r:r+}]^{\Gamma_{E/F^u}} =\mf{g}(F^u)_{y,r:r+}$. Moreover, the action of $\Gamma_{E/F^u}$ on $\mf{g}(E)_{y,0}$ descends to an algebraic action on $\ts{g}_y$. The above isomorphism restricts to an isomorphism of $\ol{k_F}$-vector spaces
\[ \mf{g}(F^u)_{y,r:r+} \rw \ts{g}_y^{\zeta^r}, \]
where $\ts{g}_y^\zeta$ denotes the $\zeta$-isotypic eigenspace for the action of $\Gamma_{E/F^u}$. Furthermore, we have
\[ \left([\ts{G}_y]^{\Gamma_{E/F^u}}\right)^\circ=G(F^u)_{y,0:0+}. \]
The fact that $y$ is stable under Frobenius provides a $k_F$-structure on both the reductive group $\ts{G}_y$ and its Lie-algebra $\ts{g}_y$ and we have
\[ \left([\ts{G}_y]^{\Gamma_{E/F^u}}\right)^\circ(k_F) = G(F)_{y,0:0+},\qquad \mf{g}(F)_{y,r:r+} \rw \ts{g}_y^{\zeta^r}(k_F). \]
Note that the same discussion applies equally well to the dual $\mf{g}^*$ of the Lie-algebra $\mf{g}$.

We now recall a few notions. Let $H$ be a reductive group over some algebraically closed field and let $V$ be a rational representation of $H$. Let $K$ be the kernel of this representation. We call a vector $v \in V$ stable if
\[ [\tx{Stab}(v,H) : K] < \infty. \]
We call a vector in $\tx{Lie}(H)$ or $\tx{Lie}^*(H)$ regular semi-simple, if its connected centralizer for the (co)-adjoint action of $H$ is a maximal torus, and strongly-regular semi-simple, if its centralizer is a maximal torus. These two notions coincide in many cases, but not always, as for example over fields of small positive characteristic.

\begin{pro} \label{pro:stabvec} An element $X$ of $\ts{g}_y^{\zeta^r}$ or ${\ts{g}^*_y}^{\zeta^r}$ is a stable vector for the action of $\left([\ts{G}_y]^{\Gamma_{E/F^u}}\right)^\circ$ if and only if its inclusion into $\ts{g}_y$ or $\ts{g}^*_y$ is regular semi-simple and the action of $\Gamma_{E/F^u}$ on $\ts{S}=\tx{Stab}(v,\ts{G}_y)$ is elliptic, i.e. $X_*(\ts{S})^{\Gamma_{E/F^u}}=X_*(Z(\ts{G}_y))$.
\end{pro}
\begin{proof} The proof is the same as for \cite[Lemma 5.6]{GLRY}. \end{proof}

\subsection{Generic characters}

In this section we will recall the notion of generic characters. It was defined  in \cite[\S9]{Yu01}, generalizing an earlier definition of Kutzko for $\tx{GL}_n$. We will furthermore provide a characterization of generic character which will be useful later.

Let $S \subset G$ be an elliptic tamely-ramified maximal torus, let $\mf{s}$ be its Lie-algebra. and let $r>0$. A character $\chi : S(F) \rw \C^\times$ is called generic of depth $r$ if it satisfies two conditions. First, it must restrict trivially to $S(F)_{r+}$. If that is the case, then using the Moy-Prasad isomorphism
\[ \tx{MP}_S : S(F)_{r:r+} \rw \mf{s}(F)_{r:r+} \]
it gives rise to a character on $\mf{s}(F)_r/\mf{s}(F)_{r+}$. This is a finite-dimensional $k_F$-vector space and its dual is given by $\mf{s}^*(F)_{-r:(-r)+}$. Thus, given a non-trivial character $\psi : k_F \rw \C^\times$, there is a unique element $X \in \mf{s}^*(F)_{-r:(-r)+}$ such that
\[ \chi(Y) = \psi\<X,Y\> \quad \forall Y \in \mf{s}(F)_{r:r+}. \]
The $k_F$-line spanned by $X$ is independent of the choice of $\psi$.

Let $E/F^u$ be the splitting extension of $S \times F^u$. There is a non-zero element $z \in E$ such that $zX \in \mf{s}^*(E)_{0:0+}$. The $\ol{k_F}$-line $l_\chi$ spanned by $zX$ depends neither on the choice of $z$ nor on the choice of $\psi$, and is thus canonically associated to $\chi$. We can embed $\mf{s}^*$ into $\mf{g}^*$ as the 1-isotypic subspace of $\mf{g}^*$ for the coadjoint action of $S$. More precisely, the natural surjection $\mf{g}^* \rw \mf{s}^*$ which is dual to the embedding $\mf{s} \rw \mf{g}$ becomes an isomorphism when restricted to the $1$-isotypic eigenspace for the coadjoint action of $S$ on $\mf{g}^*$. In this way, $l_\chi$ becomes a  line inside $\mf{g}^*(E)_{x,0:0+}$, where $x$ is the point of $\mc{B}^\tx{red}(G,F)$ corresponding to $S$ \cite{Pr01}. This quotient is isomorphic to the $k_E$-points of the dual Lie algebra of the reductive $k_E$-group $G(E)_{x,0:0+}$. The second condition that $\chi$ must satisfy in order to be called generic is that the line $l_\chi$ be strongly-regular semi-simple, i.e. that its centralizer in $G(E)_{x,0:0+}$ for the coadjoint action be a maximal torus.

We can now give the following characterization of generic characters.

\begin{lem} \label{lem:genchar} Let $\chi : S(F) \rw \C^\times$ be a character trivial on $S(F)_{r+}$ and non-trivial on $S(F)_r$. Let $E/F$ be the splitting extension of $S$ and $N : S(E) \rw S(F)$ the norm map. Then $\chi$ is generic if and only if the following conditions hold:
\begin{enumerate}
\item For each root $\alpha \in R(S,G)$, the character
\[ E^\times_r/E^\times_{r+} \rw \C^\times,\qquad x \mapsto \chi(N(\alpha^\vee(x))) \]
is non-trivial.
\item The stabilizer of $\chi\circ N : S(E)_r \rw \C^\times$ in $\Omega(S,G)$ is trivial.
\end{enumerate}
\end{lem}
\begin{proof}
We fix a non-trivial character $\psi : k_F \rw \C^\times$ and let $X \in \mf{s}^*(F)_{-r:(-r)+}$ be the unique element such that
$\chi(y)=\psi\<X,\tx{MP}_S(y)\>$ for all $y \in S(F)_{r:r+}$. If $e$ is the ramification degree of $E/F$, then one computes for $y \in S(E)_{r:r+}$
\[ \chi(N(y)) = \psi(e\cdot\tr_{k_E/k_F}\<X,\tx{MP}_S(y)\>). \]
Since $p \nmid e$, the character $\Psi : k_E \rw \C^\times$, $x \mapsto \psi(e\cdot\tr_{k_E/k_F}(x))$ is non-trivial.

Let $\omega \in E$ be a uniformizer and $\alpha \in R(S,G)$. For $x \in \mf{p}_E^r/\mf{p}_E^{r+}$ one has
\[ \chi(N(\alpha^\vee(1+x))) = \Psi\<X,d\alpha^\vee(x)\> = \Psi\<\omega^r X,d\alpha^\vee(\omega^{-r}x)\>. \]
The line in $\mf{s}^*(E)_{0:0+}$ spanned by $\omega^rX$ is the line canonically associated to $\chi$, and it is regular if and only if $\<\omega^rX,d\alpha^\vee(1)\>$ is a non-zero element of $k_E$ for all $\alpha \in R(S,G)$. This is equivalent to the non-triviality of the linear form
\[ k_E \rw k_E,\qquad u \mapsto \<\omega^rX,d\alpha^\vee(u) \> \]
 and this in turn is equivalent to the non-vanishing of the character of $\mf{p}_E^r/\mf{p}_E^{r+}$ provided by the right-hand side of above equation (in the variable x). Thus the first condition in the statement of the lemma is equivalent to the regularity of the line $l_\chi$.

Given that, the second condition is then equivalent to the strong regularity of $l_\chi$, because for a given $w \in \Omega(S,G)$, the equality $\Psi\<X,Y\>=\Psi\<X,{^wY}\>$ for all $Y$ is equivalent to the equality $X={^{w^{-1}}X}$.

\end{proof}

\subsection{Construction of a map $(S,\chi) \mapsto \pi_{S,\chi}$} \label{sec:repconst}

Let $S \subset G$ be a tamely ramified maximal torus defined over $F$ and let $\chi : S(F) \rw \C^\times$ be a character. We assume that this data has the following properties.

\begin{cnd}\ \\[-20pt] \label{cnd:char}
\begin{enumerate}
\item The image of $I_F$ in $\tx{Aut}_{\ol{F}}(S)$ is generated by an elliptic regular element.
\item If $e$ is the ramification degree of the splitting extension of $S$, then $\chi$ restricts trivially to $S(F)_\frac{2}{e}$ and non-trivially to $S(F)_\frac{1}{e}$.
\item The character of $S(F)_\frac{1}{e}/S(F)_\frac{2}{e}$ induced by $\chi$ is generic.
\end{enumerate}
\end{cnd}

Given this data, an epipelagic supercuspidal representation of $G(F)$ of depth $\frac{1}{e}$ can be constructed as follows. The torus $S$ acts on the Lie algebra $\mf{g}$ of $G$ and decomposes it as
\[ \mf{g} = \mf{s} \oplus \mf{n} \]
where $\mf{s}$ is the Lie algebra of $S$ and and $\mf{n}$ is the sum of all isotypic subspaced on which $S$ acts non-trivially. This direct decomposition is defined over $F$. Let $y \in \mc{B}^\tx{red}(G,F)$ be the unique point corresponding to the embedding $j$ \cite{Pr01}. Then we have for all real numbers $r$ we have
\[ \mf{g}(F)_{y,r} = \mf{s}(F)_{y,r} \oplus \mf{n}(F)_{y,r}. \]
The character $\chi$ provides a character on
\[ S(F)_{\frac{1}{e}:\frac{2}{e}} \cong \mf{s}(F)_{\frac{1}{e}:\frac{2}{e}} \]
which can be extended to a character $\hat\chi_0$ on
\[ G(F)_{y,\frac{1}{e}:\frac{2}{e}} \cong \mf{g}(F)_{y,\frac{1}{e}:\frac{2}{e}} \]
using the above decomposition.
Since this character is invariant under the conjugation of $S(F)$ on $G(F)_{y,\frac{1}{e}}$ and its restriction to $S(F)_\frac{1}{e}$ agrees with $\chi$, we obtain a character
\[ \hat\chi : S(F)G(F)_{y,\frac{1}{e}} \rw \C^\times, \quad (s,g) \mapsto \chi(s)\hat\chi_0(g). \]

We put $r=\frac{1}{e}$ and $\ts{V}=\mf{g}(F^u)_{y,r:r+}$. Then $\ts{V}$ is a vector space over $\ol{k_F}$ with a $k_F$-structure and $\ts{V}(k_F)=\mf{g}(F)_{y,r:r+}$. Moreover, we have $\ts{V}^*=\mf{g}^*(F^u)_{y,-r:0}$.

\begin{pro} Let $\xi : k_F \rw \C^\times$ be a non-trivial character, and let $\lambda :  \ts{V}(k_F) \rw k_F$ be the unique linear form such that $\hat\chi_0 = \xi\circ\lambda$. Then
\begin{enumerate}
\item $\lambda$ is a stable vector for the action of $G(F^u)_{y,r:r+}$ on $\ts{V}^*$.
\item The stabilizer of $\lambda$ for the action of $G(F)_y$ on $\ts{V}^*$ is precisely $S(F)G(F)_{y,r}$.
\end{enumerate}
\end{pro}

\begin{proof}

The first statement follows immediately from Proposition \ref{pro:stabvec} and the genericity of $\chi$. We set out to prove the second statement.

Recall that $r=\frac{1}{e}$, let $X \in \mf{s}^*(F)_{-r}$ be a lift of $\lambda$, and let $g \in G(F)_y$. Put $Y=\tx{Ad}(g)X$. Since
\[ \ts{V}^*(k_F) = \mf{g}^*(F)_{x,-r:0} \]
the coadjoint action of $g$ fixes $\lambda$ precisely when $X-Y \in \mf{g}^*(F)_0$. Our goal is to show that in this situation $g \in S(F)G(F)_{y,r}$. This requires several steps.

\begin{lem}  Let $t$ be a positive element of $\frac{1}{e}\Z$. Then the map
\[ g \mapsto (\tx{Ad}(h)-1)X \]
is an isomorphism of finite groups
\[ \frac{G(F)_{x,t}}{G(F)_{x,t+} \cdot S(F)_{x,t}} \lrw \frac{\mf{g}^*(F)_{x,t-r}}{\mf{g}^*(F)_{x,(t-r)+} + \mf{s}^*(F)_{x,t-r}} \]
\end{lem}

\begin{proof}
It is known that the image of the given map belongs to $\mf{g}^*(F)_{x,t-r}$ \cite[Appendix B.5]{DR09}. Since both sides are $k_F$-vector spaces of the same dimension, thus finite abelian groups of the same order, it will be enough to show that the map is an injective homomorphism.  We have
\begin{eqnarray*}
 (\tx{Ad}(h'h'')-1)X&=&\tx{Ad}(h')\tx{Ad}(h'')X-\tx{Ad}(h')X+\tx{Ad}(h')X-X\\
&=&\tx{Ad}(h')[(\tx{Ad}(h'')-1)X] + (\tx{Ad}(h')-1)X\\
&=&(\tx{Ad}(h'')-1)X + (\tx{Ad}(h')-1)X\\
\end{eqnarray*}
where the last equality holds because, due to the positivity of $t$, the coadjoint action of $G(F)_{x,t}$ on the quotient $\mf{g}^*(F)_{x,(t-r):(t-r)+}$ is trivial. To show injectivity, we consider the following diagram
\[ \xymatrix{
\frac{G(F)_{x,t}}{G(F)_{x,t+} \cdot S(F)_{x,t}}\ar[r]\ar[d]^{MP}&\frac{\mf{g}^*(F)_{x,t-r}}{\mf{g}^*(F)_{x,(t-r)+} + \mf{s}^*(F)_{x,t-r}}\ar@{^(->}[dd] \\
\frac{\mf{g}(F)_{x,t}}{\mf{g}(F)_{x,t+} + \mf{s}(F)_{x,t}}\ar@{^(->}[d]\\
\frac{\mf{g}(E)_{x,t}}{\mf{g}(E)_{x,t+} + \mf{s}(E)_{x,t}}\ar[d]^{\cdot\omega^{-et}}&\frac{\mf{g}^*(E)_{x,t-r}}{\mf{g}^*(E)_{x,(t-r)+} + \mf{s}^*(E)_{x,t-r}}\ar[d]^{\cdot \omega^{-e(t-r)}}\\
\frac{\mf{g}(E)_{x,0:0+}}{\mf{s}(E)_{0:0+}}\ar[r]^{H \mapsto [H,X]-X}&\frac{\mf{g}^*(E)_{x,0:0+}}{\mf{s}^*(E)_{0:0+}}
} \]
The inclusions come from the standard vanishing result in Galois cohomology, and the pairing $[]$ is the action of a Lie algebra on its dual, i.e. the differential of the coadjoint action. The regularity of $X$ in $\mf{g}^*(E)_{x,0:0+}$ implies that if $[X,H] \in \mf{s}^*(E)_{0:0+}$, then $H \in \mf{s}(E)_{0:0+}$, which shows that the bottom map is injective.
\end{proof}

\begin{lem} Assume that $X-Y \in \mf{g}^*(F)_{x,s}+\mf{s}^*(F)_0$, where $s \in \frac{1}{e}\Z_{\geq 0}$. Then there exists $h \in G(F)_{x,s+}$ such that $X-\tx{Ad}(h)Y \in \mf{g}^*(F)_{x,s+}+\mf{s}^*(F)_0$. \end{lem}
\begin{proof}
Write $X-Y = Z_1 + Z_2$ with $Z_1 \in \mf{g}^*(F)_{x,s}$ and $Z_2 \in \mf{s}^*(F)_0$. By the preceding lemma, choose $h \in G(F)_{x,s+}$ such that
\[ (\tx{Ad}(h^{-1})-1)X \in -Z_1 + \mf{g}^*(F)_{x,s+}+\mf{s}^*(F)_s. \]
Then we have
\begin{eqnarray*}
X-\tx{Ad}(h)Y&=&\tx{Ad}(h)[\tx{Ad}(h^{-1})X-Y]\\
&=&\tx{Ad}(h)[(\tx{Ad}(h^{-1})-1)X+Z_1+Z_2]\\
&\in&\tx{Ad}(h)[\mf{g}^*(F)_{x,s+}+\mf{s}^*(F)_0]\\
\end{eqnarray*}
The lattice $\mf{g}^*(F)_{x,s+}$ is preserved by $\tx{Ad}(h)$, while for an element $Z \in \mf{s}^*(F)_0$ we have
\[ \tx{Ad}(h)Z = (\tx{Ad}(h)-1)Z+Z \in \mf{g}^*(F)_{x,s+} + \mf{s}^*(F)_0. \]
\end{proof}

\begin{lem} There exists $h \in G(F)_{y,r}$ such that $\tx{Ad}(h)Y \in \mf{s}^*(F)_{-r}$. \end{lem}
\begin{proof}
If $Y \in \mf{s}^*(F)_{-r}$ there is nothing to prove. Otherwise, let $s$ be the largest element of $\frac{1}{e}\Z$ for which $X-Y \in \mf{g}^*(F)_{x,s}+\mf{s}^*(F)_0$. Since $X-Y \in \mf{g}^*(F)_{x,0}$, $s_1 \geq 0$ is positive. Applying the preceding lemma, we obtain $h_1 \in G(F)_{x,s+}$ such that $X-\tx{Ad}(h_1)Y \in \mf{g}^*(F)_{x,s+}+\mf{s}^*(F)_0$. Inductively we obtain a sequence $h_k \in G(F)_{x,s+\frac{k}{e}}$ with the property that
\[ X-\tx{Ad}(h_k\cdot h_{k-1}\cdot \dots \cdot h_1)Y \in \mf{g}^*(F)_{x,s+\frac{k}{e}}+\mf{s}^*(F)_0. \]
Let $h$ be the limit, for $k \rw \infty$, of the sequence of partial products $p_k = h_k \dots h_1$. Then we have
\[ X-\tx{Ad}(h)Y \in \mf{s}^*(F)_0. \]

\end{proof}

We are now ready to complete the proof. Recall that $X \in \mf{s}^*(F)_{-r}$ is generic, $g \in G(F)_y$, $Y=\tx{Ad}(g)X$, and $X-Y \in \mf{g}^*(F)_{y,0}$. Choose $h$ as in the above lemma, and put $Y'=\tx{Ad}(h)Y$. Then $Y' \in \mf{s}^*(F)_{-r}$ and
\[ X-Y' = X-Y-(\tx{Ad}(h)-1)Y \in \mf{g}^*(F)_{y,0}. \]
We claim that $w=hg \in S(F)$. Indeed, since both elements $X$ and $Y'$ belong to $\mf{s}^*(F)$ and are regular, $w$ belongs to the normalizer of $S$ in $G(F)_y$. We want to show that its image $\bar w$ in the Weyl group $\Omega(S,G)$ is trivial. Since $y$ is a special vertex in $\mc{B}^\tx{red}(G,E)$ and $E$ splits $S$, the Weyl group $\Omega(S,G)$ projects isomorphically to the Weyl group of $G_{x,E}(k_E)$. The elements $\omega X$ and $\omega Y'$ have the same image in $\mf{g}^*_{x,E}(k_E)$, and this image is strongly-regular. The statement follows.
\end{proof}

Let
\[ \pi_{S,\chi} := \textrm{c-ind}_{S(F)G(F)_{y,\frac{1}{e}}}^{G(F)} \hat\chi. \]
According to \cite{RY}, this is an irreducible supercuspidal representation of depth $\frac{1}{e}$.

\begin{fct} \label{fct:repequiv}
The representations $\pi_{S_1,\chi_1}$ and $\pi_{S_2,\chi_2}$ are isomorphic if and only if the pairs $(S_1,\chi_1)$ and $(S_2,\chi_2)$ are rationally conjugate.
\end{fct}

We now want to recall a result of DeBacker and Reeder from \cite{DR10}, pertaining to the genericity of the representation $\pi_{S,\chi}$. Assume that $G$ is quasi-split. Choose a non-degenerate $G$-invariant bilinear form $\<\>$ on $\mf{g}(F)$ and an additive character $\psi : F \rw \C^\times$. Let $(B,\psi_B)$ be a Whittaker datum. Write $B=TU$ and $\mf{g} = \mf{t} \oplus \mf{u} \oplus \ol{\mf{u}}$, where $\ol{U}$ is the unipotent radical of the Borel subgroup $T$-opposite to $B$. There exists a regular nilpotent element $E_- \in \ol{\mf{u}}(F)$ such that $\psi\<X,E_-\> = \psi_B(\tx{exp}(X))$ for all $X \in \mf{u}(F)$. Furthermore, there exists a regular semi-simple element $Y \in \mf{s}(F)$ such that for all $t \in S(F)_\frac{1}{e}$, $\chi_S(t) = \psi\<Y,\tx{MP}_S(t)\>$.

\begin{pro}[DeBacker-Reeder, \cite{DR10}] \label{pro:generic} The representation $\pi_{S,\chi}$ is $(B,\psi_B)$-generic if and only if $Y$ belongs to the Kostant-section associated to $E_-$.
\end{pro}
We remark that, while the statement of this result \cite[Prop. 4.10]{DR09} requires that $S$ be an unramified torus and its point in $\mc{B}^\tx{red}(G,F)$ be a vertex, the result holds without that assumption and the proof remains the same.

\subsection{Isocrystals and inner forms} \label{sec:isoinner}

We briefly recall some material from \cite{Kot85}, \cite{Kot97}, and \cite{Kal11a}, which will be used in the construction of $L$-packets and the study of their endoscopy. Let $L$ be a completion of $F^u$. Then $\ol{L}=L \otimes_{F^u} \ol{F}$ is an algebraic closure of $L$. In \cite{Kot85}, \cite{Kot97}, Kottwitz defines and studies the set $\tb{B}(G)$ of isomorphism classes of isocrystals with $G$-structure. It can be defined cohomologically as $H^1(W_F,G(\ol{L}))$, which is the same as the set of Frobenius-twisted conjugacy classes in $G(L)$. To such an object $b$, Kottwitz defines its Newton homomorphism $\nu_b : \mb{D} \rw G(L)$, which is a group homomorphism defined up to conjugation, and $\mb{D}$ is the pro-diagonalizable group whose character module is the trivial $\Gamma$-module $\Q$. Kottwitz shows that the natural map $H^1(\Gamma,G(\ol{F})) \rw H^1(W_F,G(\ol{L}))$ is injective and its image is precisely the set of $b$ with $\nu_b=1$. The (often larger) set for which $\nu_b$ factors through $Z(G)$ is called the set of \emph{basic} $G$-isocrystals, denoted by $\tb{B}(G)_\tx{bas}$. Kottwitz shows that every element of $\tb{B}(G)_\tx{bas}$ gives rise to an inner form of $G$.

We now consider the pullback diagram
\[ \xymatrix{
E(G,Z)\ar[d]^p\ar[r]&Z^1(W_F,G(\ol{L}))_\tx{bas}\ar[d] \\
Z^1(\Gamma,[G/Z](\ol{F}))\ar@{^{(}->}[r]&Z^1(W_F,[G/Z](\ol{L}))_\tx{bas}
} \]
Since the lower horizontal arrow is injective, so is the upper. Moreover, one can define \cite[\S2.1]{Kal11a} an equivalence relation on $E(G,Z)$ so that the square remains cartesian after passing to equivalence classes at all four corners. In doing so, the lower horizontal map becomes bijective, and thus so does the upper. The upshot is that $E(G,Z)$ allows us to select ``nice'' cocycles in each cohomology class in $B(G)_\tx{bas}$.

Let $S \subset G$ be an elliptic maximal torus. Then we can also form the set $E(S,Z)$. Just as in the case of $G$, it embeds into $Z^1(W_F,S(\ol{L}))$ (we can omit the subscript $\tx{bas}$ now, as it has no effect for tori). We claim moreover that, also as in the case of $G$, the map $E(S,Z) \rw \tb{B}(S)$ is surjective. We need to show that the map $H^1(\Gamma,[S/Z](\ol{F})) \rw H^1(W_F,[S/Z](\ol{L}))$ is bijective, which was tautological for the basic elements for $G/Z$, but it is not so here. It is nonetheless quite immediate -- the image of this map is the subset of $b$ for which $\nu_b$ is trivial. However, $\nu_b$ can be interpreted as an element of $X_*(S/Z)^\Gamma \otimes \Q$, and since $S$ is elliptic, $X_*(S/Z)^\Gamma=\{0\}$.

An inner twist of $G$ is a map $\xi : G \rw G'$ which is an isomorphism of algebraic groups defined over $\ol{F}$, and for which $\xi^{-1}\sigma(\xi) \in \tx{Inn}(G)$. Given two strongly-regular semi-simple elements $\gamma \in G(F)$ and $\gamma' \in G'(F)$, we say that $\gamma$ and $\gamma'$ are stably-conjugate (or related), if there exists $g \in G$ such that $\tx{Ad}(g)\gamma=\gamma'$. An extended pure inner twist is defined to be $(\xi,b) : G \rw G'$, where $\xi : G \rw G'$ is an inner twist, $b \in E(G,Z)$, and $\xi^{-1}\sigma(\xi) = \tx{Ad}(p(b)_\sigma)$. Two semi-simple elements $\gamma \in G(F)$ and $\gamma' \in G'(F)$ will be called stably-conjugate if there exists an equivalent extended pure inner twist $(\xi',b') : G \rw G'$ such that $\xi'(\gamma)=\gamma'$ and $b' \in B(G_\gamma)$. In that situation, we will give $b'$ the name $\tx{inv}_b(\gamma,\gamma')$, or if $b$ is understood, just $\tx{inv}(\gamma,\gamma')$. This element allows one to define compatible normalizations of transfer factors \cite[\S2.2]{Kal11a}.

\section{On toral invariants} \label{sec:torinv}

\subsection{Definition of the toral invariant} \label{sec:torinvdef}

Let $G$ be a connected reductive group defined over a local field $F$, and let $S \subset G$ be a maximal torus defined over $F$. This section will be devoted to the study of a certain invariant of the pair $(S,G)$. Consider the set $R(S,G)$ of roots of $S$. The Galois group $\Gamma$ of $F$ acts on this set. Following \cite{LS87}, we will call an orbit of this action \emph{symmetric}, if it is preserved by multiplication by $-1$. Otherwise, we will call the orbit \emph{asymmetric}. We will also need to pay attention to the action of the inertia subgroup $I \subset \Gamma$ on $R(S,G)$. A given $\Gamma$-orbit $\mc{O}$ decomposes as a disjoint union of $I$-orbits, and one has the following dichotomy: Either all $I$-orbits contained in $\mc{O}$ are stable under multiplication by $-1$, or none is. In the first case, we will call $\mc{O}$ \emph{inertially symmetric}, and in the second case \emph{inertially asymmetric}. We will say that a root $\alpha \in R(S,G)$ is (inertially) symmetric/asymmetric, if its $\Gamma$-orbit has this property. The sets of symmetric resp. inertially symmetric roots will be denoted by $R(S,G)_\tx{sym}$ resp. $R(S,G)_\tx{insym}$.

Given a root $\alpha \in R(S,G)$, one has the subgroups
\[ \Gamma_\alpha = \tx{Stab}(\alpha,\Gamma)\qquad\tx{and}\qquad \Gamma_{\pm\alpha} = \tx{Stab}(\{\alpha,-\alpha\},\Gamma). \]
One has $[\Gamma_{\pm\alpha}:\Gamma_\alpha] =1$ if $\alpha$ is asymmetric, and $[\Gamma_{\pm\alpha}:\Gamma_\alpha] =2$ if $\alpha$ is symmetric. Analogously, we have the subgroups $I_\alpha$ and $I_{\pm\alpha}$ of $I$ with $[I_{\pm\alpha}:I_\alpha]$ being equal to 1 resp. 2 when $\alpha$ is inertially asymmetric resp. symmetric. Let $F_\alpha$ and $F_{\pm\alpha}$ be the subfields of $\ol{F}$ fixed by $\Gamma_\alpha$ resp. $\Gamma_{\pm\alpha}$. Then $\alpha$ is symmetric if and only if $F_\alpha/F_{\pm\alpha}$ is a quadratic extension, and is inertially symmetric if and only if this extension is ramified.

The toral invariant that we are going to study in this section is a function
\[ f : R(S,G)_\tx{sym} \rw \{\pm 1\}, \]
which is defined as follows: Let $\alpha \in R(S,G)$ be a symmetric root. The 1-dimensional root subspace $\mf{g}_\alpha \subset \mf{g}$ corresponding to $\alpha$ is defined over $F_\alpha$, and we may choose a non-zero element $X_\alpha \in \mf{g}_\alpha(F_\alpha)$. Let $\tau \in \Gamma_{\pm\alpha} \sm \Gamma_\alpha$. Then $\tau X_\alpha$ is a non-zero element of $\mf{g}_{-\alpha}(F_\alpha)$, and
\[ f(X_\alpha) := \frac{[X_\alpha,\tau X_\alpha]}{H_\alpha} \]
is thus a non-zero element of $F_\alpha$. Here $H_\alpha \in \mf{s}(F_\alpha)$ is the coroot corresponding to $\alpha$. One checks right away that $f(X_\alpha) \in F_{\pm\alpha}^\times$, and that changing the choice of $X_\alpha$ multiplies $f(X_\alpha)$ by an element of $F_{\pm\alpha}^\times$ which is a norm from $F_\alpha$. Let $\kappa_\alpha : F_{\pm\alpha}^\times \rw \{\pm 1\}$ be the non-trivial character which kills all norms from $F_\alpha$. It follows that
\[ f(\alpha) = \kappa_\alpha\left(\frac{[X_\alpha,\tau X_\alpha]}{H_\alpha}\right) \]
is independent of the choice of $X_\alpha$. This is the invariant that will be the subject of this section. In the case where $F=\R$, this invariant is well known -- a symmetric root $\alpha$ is then customarily called imaginary, and it is further called compact if $f(\alpha)=-1$, and non-compact otherwise.

The following simple observation is sometimes useful.
\begin{fct} \label{fct:torinvgal} The function $f : R(S,G)_\tx{sym} \rw \{\pm 1\}$ is $\Gamma$-invariant.
\end{fct}

Since this invariant is associated to the pair $(G,S)$, we will sometimes denote it by $f_{(G,S)}$. It is clear that $f_{(G,S)} = f_{(G_\tx{ad},S_\tx{ad})}$. On the other hand, it is important to note that this invariant is sensitive to both the group $G$ and the torus $S$ -- if $S$ is replaced by a stably-conjugate torus, or if $G$ is replaced by an inner form, then the invariant will in general be different. We will study some of this behavior in the next subsections.

Before we begin with the study of the invariant $f_{(G,S)}$, it will be useful to recall another notion from \cite{LS87} which we will heavily use. A \emph{gauge} on $R(S,G)$ is a function $p : R(S,G) \rw \{\pm 1\}$ with the property that $p(-\alpha)=-p(\alpha)$. A choice of positive roots on $R(S,G)$ determines a gauge, but not all gauges arise in this way. If $p,q$ are two gauges, then there exists a sequence of gauges $p=p_0,p_1,\dots,p_n=q$ such that for each $i=0,\dots,n-1$, the gauges $p_i$ and $p_{i+1}$ disagree on a single pair $\{\gamma,-\gamma\}$ of roots. In other words, $p_i(\alpha) \neq p_{i+1}(\alpha)$ implies $\alpha \in \{\gamma,-\gamma\}$.

\subsection{A cohomological interpretation}
The sign $f(\alpha)$ associated to $\alpha \in R(S,G)_\tx{sym}$ can be given the following cohomological interpretation. Let $S_\alpha$ be the 1-dimensional anisotropic torus defined over $F_{\pm\alpha}$ and split over $F_\alpha$. Thus, $S_\alpha(\ol{F})=\ol{F}^\times$ with $\sigma \in \Gamma_{\pm\alpha}$ acting as
\[ \sigma_\alpha(x) = (\sigma x)^{\kappa_\alpha(\sigma)}, \]
where $\kappa_\alpha$ is now viewed as a character on $\Gamma_{\pm\alpha}$ via local class field theory. The inflation-restriction sequence and Hilbert's theorem 90 imply that the inflation map
\[ H^1(\Gamma_{\pm\alpha}/\Gamma_\alpha,S_\alpha(F_\alpha)) \rw H^1(\Gamma_{\pm\alpha},S_\alpha(\ol{F})) \]
is an isomorphism. On the other hand, we have the isomorphism
\[ H^1(\Gamma_{\pm\alpha}/\Gamma_\alpha,S_\alpha(F_\alpha)) \rw F_{\pm\alpha}^\times/N(F_\alpha^\times) \]
given by evaluating a 1-cocycle of $\Gamma_{\pm\alpha}/\Gamma_\alpha$ with values in $S_\alpha(F_\alpha)  = F_\alpha^\times$ at the non-trivial element of  $\Gamma_{\pm\alpha}/\Gamma_\alpha$. Composing the inverse of the first isomorphism with the second and then with the character $\kappa_\alpha$, we obtain an isomorphism
\[ \kappa_\alpha^\tx{coh} : H^1(\Gamma_{\pm\alpha},S_\alpha(\ol{F}))  \rw \{\pm 1\}. \]

\begin{pro}\label{pro:torinvcoh}\ \\[-20pt]
\begin{enumerate}
\item
For any non-zero element $X_\alpha \in \mf{g}_\alpha(\ol{F})$, the map
\[ f^\tx{coh}(X_\alpha) : \Gamma_{\pm\alpha} \rw \ol{F}^\times,\qquad \sigma \mapsto \frac{\sigma X_{\sigma^{-1}\alpha}}{X_\alpha} \]
is an element of $Z^1(\Gamma_{\pm\alpha},S_\alpha(\ol{F}))$.
\item If $X_{-\alpha} \in \mf{g}_{-\alpha}$ is such that $[X_\alpha,X_{-\alpha}]=H_\alpha$, then
\[ f^\tx{coh}(X_{-\alpha}) = f^\tx{coh}(X_\alpha)^{-1}. \]
\item The cohomology class $f^\tx{coh}(\alpha)$ of $f^\tx{coh}(X_\alpha)$ is independent of the choice of $X_\alpha$, and its image under $\kappa_\alpha^\tx{coh}$ is equal to $f(\alpha)$.
\end{enumerate}
\end{pro}
\begin{proof}
We note first that $f^\tx{coh}(X_\alpha)$ is well-defined, because $\sigma X_{\sigma^{-1}\alpha}$ is a non-zero element of $\mf{g}_\alpha(\ol{F})$ for each $\sigma \in \Gamma_{\pm\alpha}$. Before proving that $f^\tx{coh}(X_\alpha)$ is a 1-cocycle, it will be more convenient to prove the second part of the proposition. For this, note first that for all $\sigma \in \Gamma$, we have
\[ \sigma H_{\sigma^{-1}\alpha} = H_\alpha. \]
Then applying $\sigma \in \Gamma_{\pm\alpha}$ to the equation
\[ 1 = \frac{[X_{\sigma^{-1}\alpha},X_{-\sigma^{-1}\alpha}]}{H_{\sigma^{-1}\alpha}} \]
we obtain
\begin{eqnarray*}
1&=&\frac{[\sigma X_{\sigma^{-1}\alpha},\sigma X_{-\sigma^{-1}\alpha}]}{\sigma H_{\sigma^{-1}\alpha}}\\
&=&\frac{[f^\tx{coh}(X_\alpha,\sigma) X_\alpha, f^\tx{coh}(X_{-\alpha},\sigma)X_{-\alpha}]}{H_{\alpha}}\\
&=&f^\tx{coh}(X_\alpha,\sigma)f^\tx{coh}(X_{-\alpha},\sigma).
\end{eqnarray*}
Having shown the second claim, we now turn to $f^\tx{coh}(X_\alpha) \in Z^1(\Gamma_{\pm\alpha},S_\alpha(\ol{F}))$. For this, let $\sigma,\tau \in \Gamma_{\pm\alpha}$. Then
\begin{eqnarray*}
f^\tx{coh}(X_\alpha,\sigma\tau)&=&\frac{\sigma\tau X_{\tau^{-1}\sigma^{-1}\alpha}}{X_\alpha}\\
&=&\frac{\sigma\tau X_{\tau^{-1}\sigma^{-1}\alpha}}{\sigma X_{\sigma^{-1}\alpha}} \frac{\sigma X_{\sigma^{-1}\alpha}}{X_\alpha} \\
&=&\sigma\left(\frac{\tau X_{\tau^{-1}\sigma^{-1}\alpha}}{X_{\sigma^{-1}\alpha}}\right) f^\tx{coh}(X_\alpha,\sigma)
\end{eqnarray*}
Now
\[ \frac{\tau X_{\tau^{-1}\sigma^{-1}\alpha}}{X_{\sigma^{-1}\alpha}} = \begin{cases} f^\tx{coh}(X_\alpha,\tau)&,\sigma \in \Gamma_\alpha\\ f^\tx{coh}(X_{-\alpha},\tau)&, \sigma \in \Gamma_{\pm\alpha}\sm \Gamma_\alpha \end{cases} \]
and using part 2 which we just proved, we see that in both cases we have
\[ \sigma\left(\frac{\tau X_{\tau^{-1}\sigma^{-1}\alpha}}{X_{\sigma^{-1}\alpha}}\right) = \sigma_\alpha(f^\tx{coh}(X_\alpha,\tau)). \]
We now come to part 3. If $\tilde X_\alpha = cX_\alpha$ for some $c \in \ol{F}^\times$, then  $\tilde X_{-\alpha} = c^{-1}X_{-\alpha}$, and hence we have $\sigma(\tilde X_{\sigma^{-1}\alpha}) = \sigma_\alpha(c)\sigma X_{\sigma^{-1}\alpha}$. This implies that $f(\tilde X_\alpha,\sigma)=c^{-1}\sigma_\alpha(c)f(X_\alpha,\sigma)$, showing that the class of $f(X_\alpha)$ is independent of the choice of $X_\alpha$. Now let $X_\alpha \in \mf{g}_\alpha(F_\alpha)$. Then $f(X_\alpha)$ is the inflation of an element of $Z^1(\Gamma_{\pm\alpha}/\Gamma_\alpha,S_\alpha(F_\alpha))$, and thus
\[ \kappa_\alpha^\tx{coh}(f^\tx{coh}(\alpha)) = \kappa_\alpha(f^\tx{coh}(X_\alpha,\tau)) \]
for any $\tau \in \Gamma_{\pm\alpha} \sm \Gamma_\alpha$. But
\[ f^\tx{coh}(X_\alpha,\tau) = \frac{\tau X_{-\alpha}}{X_\alpha} = \tau f(X_\alpha)^{-1}\]
and hence
\[ \kappa_\alpha(f^\tx{coh}(X_\alpha,\tau)) = \kappa_\alpha(f(X_\alpha)) = f(\alpha). \]
\end{proof}

As a first application of the cohomological interpretation of $f(\alpha)$, we obtain a description of the behavior of $f(\alpha)$ under stable conjugacy.

\begin{pro} \label{pro:torinvstab} Let $(G,S)$ and $(G',S')$ be two pairs consisting of a connected reductive group and a maximal torus thereof, both defined over the local field $F$. Let $\xi : G \rw G'$ be an inner twist which restricts to an isomorphism $S \rw S'$ defined over $F$. Then
\begin{enumerate}
\item
\[ f_{(G',S')}(\xi\alpha) = f_{(G,S)}(\alpha) \cdot \kappa_\alpha(\eta_{t,\alpha}), \]
where for each $\alpha \in R(S,G)_\tx{sym}$, $\eta_{t,\alpha}$ is the image of $t_\sigma = \xi^{-1}{^\sigma\xi} \in H^1(F,S_\tx{ad})$ under
\[ \xymatrix{ H^1(F,S_\tx{ad})\ar[r]^-{\tx{Res}}&H^1(F_{\pm\alpha},S_\tx{ad})\ar[r]^\alpha& H^1(F_{\pm\alpha},S_\alpha)}. \]
\item If $t$ is the image of $\lambda \in H^{-1}(E/F,X_*(S_\tx{ad}))$ under the Tate-Nakayama isomorphism, then
\[ \kappa_\alpha(\eta_{t,\alpha}) = (-1)^{\<\lambda,\sum_{\sigma \in \Gamma/\Gamma_{\pm\alpha}}\sigma\alpha \>}. \]
\item If $S$ is elliptic, we have
\[ \prod_{\alpha \in R(S',G')_\tx{sym}/\Gamma} f_{(G',S')}(\alpha) = e(G)e(G')\prod_{\alpha \in R(S,G)_\tx{sym}/\Gamma} f_{(G,S)}(\alpha). \]
\end{enumerate}
\end{pro}
\begin{proof}
Choose $X_\alpha \in \mf{g}_\alpha(\ol{F})$. Then the first claim follows from part 3 of Proposition \ref{pro:torinvcoh} and the following computation.
\begin{eqnarray*}
f_{(G',S')}^\tx{coh}(\xi X_\alpha,\sigma) &= &\frac{\sigma \xi X_{\sigma^{-1}\alpha}}{\xi X_\alpha}
=  \frac{\xi(\xi^{-1}{^\sigma\xi})\sigma X_{\sigma^{-1}\alpha}}{\xi X_\alpha}\\
&=&\frac{\xi \alpha(\xi^{-1}{^\sigma\xi})f^\tx{coh}_{(G,S)}(X_\alpha,\sigma)X_\alpha}{\xi X_\alpha}\\
&=&\alpha(\xi^{-1}{^\sigma\xi})f^\tx{coh}_{(G,S)}(X_\alpha,\sigma).
\end{eqnarray*}
For the second claim, the functoriality and compatibility with connecting homomorphisms of the Tate-Nakayama isomorphism gives the diagram
\[ \xymatrix{
H^1(F,S)\ar[r]^{\tx{Res}}&H^1(F_{\pm\alpha},S)\ar[r]^\alpha&H^1(F_{\pm\alpha},S_\alpha)\\
H^{-1}(E/F,X_*(S))\ar[u]^\cong\ar[r]^{\tx{Res}}&H^{-1}(E/F_{\pm\alpha},X_*(S))\ar[u]^\cong\ar[r]^\alpha&H^{-1}(F_{\pm\alpha}/F_\alpha,\Z_{(-1)})\ar[u]\\
[X_*(S)_\Gamma]_\tx{tor}\ar[u]\ar[r]&[X_*(S)_{\Gamma_{\pm\alpha}}]_\tx{tor}\ar[u]\ar[r]&\Z/2\Z\ar[u]
} \]
The left horizontal arrow on the bottom is given by
\[ [\lambda] \mapsto [\sum_{\sigma \in \Gamma/\Gamma_{\pm\alpha}} \sigma^{-1}\lambda] \]
and the right horizontal arrow on the bottom is given by
\[ [\mu] \mapsto \<\mu,\alpha\>. \]
This proves the second claim. For the final claim, we need to show that
\[ \prod_{\alpha \in R(G,S)_\tx{sym}/\Gamma} \kappa_\alpha(\eta_{t,\alpha}) \]
is equal to $e(G)e(G')$. In view of the first part of this proposition, we may assume that $G$ is quasi-split. Using the second part, we see that the above product is equal to $(-1)$ raised to the power
\[ \sum_{\alpha \in R(G,S)_\tx{sym}/\Gamma} \sum_{\sigma \in \Gamma/\Gamma_{\pm\alpha}} \<\lambda, \sigma\alpha \>. \]
On the other hand, the Kottwitz sign $e(G')$  \cite{Kot83} is given by taking the image of $\xi^{-1}{^\sigma\xi}$ under the map
\[ \xymatrix{
H^1(F,S_\tx{ad})\ar[r]^-\partial&H^2(F,Z(G_\tx{sc}))\ar[r]^-\rho&H^2(F,\mb{G}_m)\ar[r]^{exp(2\pi i \tx{inv})}&\C^\times}, \]
where $\rho$ denotes half the sum of an arbitrary set of positive roots. Dual to the above sequence we have the sequence
\[ \xymatrix{
H^1(F,X^*(S_\tx{ad}))&\ar[l]_-\partial H^0(F,X^*(Z(G_\tx{sc})))&\ar[l]_-{\cdot\rho} H^0(F,\Z) \\
} \]
The character $\tx{exp}(2\pi i \tx{inv}) : H^2(F,\mb{G}_m) \rw \C^\times$ corresponds to the element $1 \in H^0(F,\Z)$. Thus the character on $H^1(F,S_\tx{ad})$ given by the first sequence is equal to the image of $1$ under the second sequence. This image is represented by the element $\sigma\rho-\rho$ of $Z^1(F,X^*(S_\tx{ad}))$. We want to compute the paring of the class of this cocycle with $\xi^{-1}{^\sigma\xi}$. This pairing is given by
\[ \exp( 2\pi i \tx{inv}((\sigma\rho-\rho) \cup(\xi^{-1}{^\sigma\xi})). \]
But we are assuming that $\xi^{-1}{^\sigma\xi}=(\lambda \cup \tx{Fund}_{E/F})$, where $\tx{Fund}_{E/F}$ denotes the fundamental class of $E/F$ in $H^2(F,\mb{G}_m)$. Thus
\begin{eqnarray*}
\exp( 2\pi i \tx{inv}((\sigma\rho-\rho) \cup(\xi^{-1}{^\sigma\xi})) &=& \exp( 2\pi i \tx{inv}((\sigma\rho-\rho)  \cup (\lambda \cup \tx{Fund}_{E/F}) )\\
&=&\exp( 2\pi i [E:F]^{-1}((\sigma\rho-\rho)\cup\lambda))\\
&=&\exp( 2\pi i [E:F]^{-1}(\sum_{\sigma \in \Gamma(E/F)} \sigma\< (\sigma^{-1}\rho-\rho),\lambda\>))\\
&=&\exp( 2\pi i \< -\rho,\lambda\>)\\
&=&(-1)^{\<2\rho,\lambda\>}
\end{eqnarray*}
What remains to be shown is that in $\Z/2\Z$ we have the equality
\[ \<2\rho,\lambda\> = \sum_{\alpha \in R(G,S)_\tx{sym}/\Gamma} \sum_{\sigma \in \Gamma_E/\Gamma_{\pm\alpha}} \<\lambda, \sigma\alpha \>.\]
Recall that $2\rho$ is the sum of an arbitrary choice of positive roots in $R(S,G)$. For any gauge $p : R(S,G) \rw \{\pm 1\}$, let
\[ R_p = \sum_{\substack{\beta \in R(S,G)\\p(\beta)=+1}} \beta. \]
It is clear that for two gauges $p,q$, we have $R_p=R_q$ in $X^*(S)/2X^*(S)$. If we take $p$ to be the gauge for which $R_p=2\rho$, and $q$ to be a gauge which is constant on every asymmetric orbit of $\Gamma$ in $R(S,G)$, then we see that $\<2\rho,\lambda\> = \<R_q,\lambda\>$ and moreover
\[ R_q = \sum_{\substack{\beta\in R(S,G)_\tx{sym}\\ q(\beta)=+1}} \beta. \]
The reason that the above sum runs only over the symmetric roots is that if an asymmetric root is $q$-positive, then so is its entire orbit, but the sum over that orbit is zero due to the ellipticity of $S$. It is now clear that in $X^*(S)/2X^*(S)$ we have the equality
\[ \sum_{\alpha \in R(G,S)_\tx{sym}/\Gamma} \sum_{\sigma \in \Gamma_E/\Gamma_{\pm\alpha}} \alpha = R_q. \]

\end{proof}

\subsection{A vanishing result}
The purpose of this section is to prove the following vanishing statement for the toral invariant $f : R(S,G)_\tx{sym} \rw \{\pm 1\}$ when $F$ is a non-archimedean local field.

\begin{pro} \label{pro:torinvvan} Assume that the action of $I$ on $X^*(S)$ is tame and generated by a regular elliptic element. Then $f(\alpha)=1$ for all inertially asymmetric $\alpha \in R(S,G)_\tx{sym}$.
\end{pro}

We will first prove a weaker result which holds under a more general hypothesis.

\begin{lem} \label{lem:torinvstabie} In the situation of Proposition \ref{pro:torinvstab}, assume that the action of $I$ on $X^*(S)$ is elliptic. Then for all inertially asymmetric $\alpha \in R(S,G)_\tx{sym}$, we have
\[ f_{(G',S')}(\xi\alpha) = f_{(G,S)}(\alpha). \]
\end{lem}

\begin{proof}[Proof of Lemma \ref{lem:torinvstabie}] According to Proposition \ref{pro:torinvstab} it is enough to show that for any set of representatives $X \subset \Gamma$ for the quotient $\Gamma / \Gamma_{\pm\alpha}$, we have
\[ \sum_{\sigma \in X} \sigma\alpha  \in 2Q, \]
where $Q$ is the span of $R(S,G)$ in $X^*(S)$. It is clear that this statement is independent of the particular choice of $X$. Moreover, the choice of $X$ is equivalent to the choice of a gauge on the orbit $\mc{O}$ of $\Gamma$ through $\alpha$. By assumption, each $I$-orbit inside of $\mc{O}$ is asymmetric. Thus, we may choose a gauge $p : \mc{O} \rw \{\pm 1\}$ which is constant on each $I$-orbit. Then the above sum takes the form
\[ \sum_{\substack{o \in \mc{O}/I \\ p(o)=+1}} \sum_{\beta \in o} \beta. \]
The inertial ellipticity of $S$ implies that the inner sum is always zero.
\end{proof}

Before proving Proposition \ref{pro:torinvvan}, we need a preparatory lemma.

\begin{lem} \label{lem:h1} Assume that the action of $I$ on $X^*(S)$ is tame and generated by a regular element. Let $\alpha \in R(S,G)_\tx{sym}$ be inertially asymmetric, and let $E$ be the splitting field of $S$. If $X_\alpha \in \mf{g}_\alpha(E)$, then
\[ f(\alpha) = (-1)^{\tx{val}_E(f^\tx{coh}(X_\alpha,\tau))} \]
for any $\tau \in \Gamma_{\pm\alpha} \sm \Gamma_\alpha$.
\end{lem}
\begin{proof}
Since $X_\alpha$ is fixed by $\Gamma_E$, the 1-cocycle $f^\tx{coh}(X_\alpha)$ is the inflation of an element of $Z^1(\Gamma_{\pm\alpha}/\Gamma_E,S_\alpha(E))$. According to Proposition \ref{pro:torinvcoh}, $f(\alpha)$ is the image of $f^\tx{coh}(X_{\alpha})$ under the isomorphism
\[ \xymatrix{
H^1(\Gamma_{\pm\alpha}/\Gamma_\alpha,S_\alpha(F_\alpha))\ar[d]\ar[r]^{\tx{Inf}}&H^1(\Gamma_{\pm\alpha}/\Gamma_E,S_{\alpha_0}(E)) \\ F_{\pm\alpha}^\times/N(F_\alpha^\times)\ar[r]^{\kappa_\alpha}&\{\pm 1\} } \]
We are assuming that $\alpha$ is inertially asymmetric, and hence $F_\alpha/F_{\pm\alpha}$ is an unramified extension. Thus, the composition of the vertical map and the lower horizontal map sends a class $[z]$ to
\[  (-1)^{\tx{val}_{F_\alpha}(z(\tau))}, \]
where $\tau$ is the non-trivial element of $\Gamma_{\pm\alpha}/\Gamma_\alpha$. To prove the lemma, it will be enough to show that in fact the full isomorphism $H^1(\Gamma_{\pm\alpha}/\Gamma_E,S_{\alpha_0}(E)) \rw \{\pm 1\}$ sends a class $[z]$ to
\[ (-1)^{\tx{val}_{E}(z(\tau))}, \]
where $\tau$ is now any element of $\Gamma_{\pm\alpha} \sm \Gamma_\alpha$. The regularity of the $I$-action implies that the extension $E/F_\alpha$ is unramified. Thus, we know that the isomorphism takes the desired form on 1-cocycles inflated from $Z^1(\Gamma_{\pm\alpha}/\Gamma_\alpha,S_\alpha(F_\alpha))$. What we need to show is that the above displayed expression is independent of the choice of $\tau$ and of the representative $z$ within its cohomology class.

To see independence of $\tau$, let $\sigma \in \Gamma_\alpha$. Since $H^1(\Gamma_\alpha/\Gamma_E,S_\alpha(E))$ is trivial, we can choose $c \in S_\alpha(E)=E^\times$ such that $z(\sigma)=c^{-1}{^\sigma c}$. Then
\[ \tx{val}_E(z(\tau\sigma)) = \tx{val}_E(z(\tau){^\tau(c^{-1}\cdot{^\sigma c})}) = \tx{val}_E(z(\tau)). \]
To see independence of the representative $z$, let $c \in E^\times$. Then
\[ \tx{val}_E((c^{-1}\cdot{^{\tau_\alpha}c})z(\tau)) = \tx{val}_E((c\cdot{^\tau c})^{-1}) +\tx{val}_E(z(\tau)) \in \tx{val}_E(z(\tau))+2\Z. \]
\end{proof}

\begin{proof}[Proof of Proposition \ref{pro:torinvvan}]
The first step in the proof will be to replace $S$ by a convenient torus in the quasi-split inner form of $G$. We may assume without loss of generality that $G$ is simply connected. Let $G_0$ be the quasi-split inner form of $G$, and let $(T_0,B_0,\{X_\alpha\})$ be a splitting of $G_0$. Let $\xi : G_0 \rw G$ be an inner twist such that $\xi(T_0)=S$. For $\sigma \in \Gamma$, let $w_S(\sigma) \in \Omega(T_0,G_0)$ be the image of $\xi^{-1}{^\sigma\xi} \in N(T_0,G_0)$. Then $w_S \in Z^1(\Gamma,\Omega(T_0,G_0))$. We are going to choose $a$-data for $R(S,G)$. Recall that $a$-data is a $\Gamma$-equivariant function $a : R(S,G) \rw \ol{F}^\times$ such that $a(-\alpha)=-a(\alpha)$. It is enough to specify $a$ on a set of representatives for $R(S,G)/(\Gamma \times \{\pm 1\})$. We fix such a set, and for $\alpha$ belonging to it, we set $a(\alpha)$ as follows:
\begin{enumerate}
\item If $\alpha$ is asymmetric, we put $a(\alpha)=1$.
\item If $\alpha$ is symmetric but inertially asymmetric, we choose an element of $\tx{Ker}(Tr : k_{F_\alpha} \rw k_{F_{\pm\alpha}})$ and let $a_\alpha$ be its Teichm\"uller representative.
\item If $\alpha$ is inertially symmetric, then we choose a uniformizer $\omega \in F_\alpha$ such that $\omega^2 \in F_{\pm\alpha}$, and put $a(\alpha)=\omega$.
\end{enumerate}
One checks easily that this provides a valid $a$-data for $R(S,G)$. We transport this $a$-data via $\xi$ to $a$-data for $R(T_0,G_0)$ for the action of $\Gamma$ on that set given by $w_S(\sigma)\sigma$. For $w \in \Omega(T_0,G_0)$, let $n(w) \in N(T_0,G_0)$ be the Springer lift of $w_S(\sigma)$, and let $x_S(\sigma) \in T_0$ be given by
\[ x_S(\sigma) = \prod_{\substack{\beta > 0\\ (w_S(\sigma)\sigma)^{-1}\beta < 0 }} \beta^\vee(a(\beta)), \]
where $\beta>0$ means $\beta \in R(T_0,B_0)$.
It is then shown in \cite[\S2.3]{LS87} that
\[ n_S(\sigma) :=  x_S(\sigma)n(w_S(\sigma)) \]
is an element of $Z^1(F,N(T_0,G_0))$. Since $H^1(F,G_0)=1$, we may choose $g \in G_0$ such that $g^{-1}{^\sigma g} = n_S(\sigma)$. Then $S_0:=\tx{Ad}(g)T_0$ is a maximal torus of $G_0$ which one easily checks is defined over $F$. Applying Lemma \ref{lem:torinvstabie} to the inner twist $\xi\circ\tx{Ad}(g^{-1}) : G_0 \rw G$ we see that it is enough to compute $f_{(G_0,S_0)}(\tx{Ad}(g)\xi^{-1}\alpha)$.

To lighten the notation, we now assume that $G=G_0$ and $S=S_0$. Let $\alpha_0 \in R(T_0,G_0)$ be such that $\tx{Ad}(g)\alpha_0=\alpha$. Let $X_{\alpha_0} \in \mf{g}_{\alpha_0}(\ol{F})$ and $X_{-\alpha_0} \in \mf{g}_{\alpha_0}(\ol{F})$ be such that $[X_{\alpha_0},X_{-\alpha_0}]=H_{\alpha_0}$. Then $f^\tx{coh}(\alpha) \in H^1(F_{\pm\alpha},S_\alpha)$ is the class of
\[ f^\tx{coh}(X_{\alpha_0},\sigma) = \frac{n_S(\sigma)\sigma X_{(w_S(\sigma)\sigma)^{-1}\alpha_0}}{X_{\alpha_0}}. \]
We are going to choose $X_{\alpha_0}$ in the following way: Choose $w \in \Omega(T_0,G_0)$ such that $w^{-1}\alpha_0 \in \Delta(T_0,B_0)$ and set $X_{\alpha_0} = n(w)(X_{w^{-1}\alpha_0})$. It is then known \cite[Exp. XXXIII, \S6]{SGA3}, that for all $\sigma \in \Gamma_{\pm\alpha}$,
\[ n(w_S(\sigma))\sigma X_{(w_S(\sigma)\sigma)^{-1}\alpha_0} = \epsilon_\sigma X_{\alpha_0}, \]
with $\epsilon_\sigma \in \{\pm 1\}$. Thus
\[ f^\tx{coh}(X_{\alpha_0},\sigma) = \alpha_0(x_S(\sigma))\epsilon_\sigma. \]
Now let $E$ be the splitting field of $S$. Any $\sigma \in \Gamma_E$ fixes all elements of the pinning $\{X_\alpha\}_{\alpha \in \Delta(T_0,B_0)}$, and hence it also fixes $X_{\alpha_0}$. Moreover, $n_S(\sigma)=1$. It follows that $f^\tx{coh}(X_{\alpha_0}) \in Z^1(\Gamma_{\pm\alpha}/\Gamma_E,S_\alpha(E))$. According to Lemma \ref{lem:h1}, we have
\[ f(\alpha) = (-1)^{\tx{val}_E(f^\tx{coh}(X_{\alpha_0},\tau))} = (-1)^{\tx{val}_E(\alpha_0(x_S(\tau)))}\]
for any $\tau \in \Gamma_{\pm\alpha} \sm \Gamma_\alpha$. Recalling the definition of $x_S(\tau)$, we arrive at the formula
\[ f(\alpha) = (-1)^\zeta,\quad\tx{where}\quad \zeta = \sum_{\substack{\beta > 0\\ (w_S(\tau)\tau)^{-1}\beta < 0 }} \tx{val}_E(a(\beta)) \<\alpha_0,\beta^\vee\>. \]
Recall that $\beta>0$ means $\beta \in R(T_0,B_0)$. From now on, it will be more convenient in terms of notation to sum over roots for $S$ instead. Thus, let
$B:=\tx{Ad}(g)B_0$. Then $B$ is a Borel subgroup of $G$ (not necessarily defined over $F$) containing $S$, and we have
\[ \zeta = \sum_{\substack{\beta > 0\\ \tau^{-1}\beta < 0 }} \tx{val}_E(a(\beta)) \<\alpha,\beta^\vee\>, \]
where now $\beta>0$ means $\beta \in R(S,B)$. According to our choice of $a$-data, we have
\[ \tx{val}_E(a(\beta)) = \begin{cases} 1&,I\beta=-I\beta\\ 0&,\tx{else} \end{cases}. \]
Thus we obtain
\[ \zeta = \sum_{\substack{\beta > 0\\ \tau^{-1}\beta < 0\\I\beta=-I\beta }} \<\alpha,\beta^\vee\>. \]

\begin{lem} For any gauge $p: R(S,G) \rw \{\pm 1\}$, let
\[ \zeta_p = \sum_{\substack{p(\beta) =+1\\ p(\tau^{-1}\beta)=-1\\I\beta=-I\beta }} \<\alpha,\beta^\vee\>. \]
Then for any two gauges $p,q$, we have $(-1)^{\zeta_p} = (-1)^{\zeta_q}$.
\end{lem}
\begin{proof}
We have
\begin{eqnarray*}
(-1)^{\zeta_p}&=&\prod_{\substack{p(\beta) =+1\\ p(\tau^{-1}\beta)=-1\\I\beta=-I\beta }} (-1)^{\<\alpha,\beta^\vee\>} \\
&=&\prod_{\substack{p(\beta) =+1\\ I\beta=-I\beta }} \Big(p(\beta)p(\tau^{-1}\beta)\Big)^{\<\alpha,\beta^\vee\>}
\end{eqnarray*}
We notice that each factor in the above product remains unchanged if we replace $\beta$ by $-\beta$. This allows us to rewrite the product as
\[ \prod_{\substack{\{\beta,-\beta\} \in R(S,G)/\{\pm 1\}\\ I\beta=-I\beta }} \Big(p(\beta)p(\tau^{-1}\beta)\Big)^{\<\alpha,\beta^\vee\>}.  \]
As already remarked, we may assume without loss of generality that there exists $\gamma \in R(S,G)$ such that if $p(\beta) \neq q(\beta)$ then $\beta \in \{\gamma,-\gamma\}$. If $I\gamma \neq -I\gamma$, then $\zeta_p=\zeta_q$, so assume $I\gamma=-I\gamma$. Each pair of roots $\{\beta,-\beta\}$ provides the same contribution to $(-1)^{\zeta_p}$ and $(-1)^{\zeta_q}$, except possibly the pairs $\{\gamma,-\gamma\}$ and $\tau\{\gamma,-\gamma\}$. These pairs can either be distinct, or equal. If they are equal, that is, if $\tau\gamma = \pm\gamma$, then $p(\gamma)p(\tau^{-1}\gamma) = q(\gamma)q(\tau^{-1}\gamma)$ and thus $(-1)^{\zeta_p}=(-1)^{\zeta_q}$. If on the other hand the pairs $\{\gamma,-\gamma\}$ and $\tau\{\gamma,-\gamma\}$ are distinct, then $(-1)^{\zeta_p}(-1)^{\zeta_q}$ is equal to

\[ \left(\frac{p(\gamma)p(\tau^{-1}\gamma)}{q(\gamma)q(\tau^{-1}\gamma)}\right)^{\<\alpha,\gamma\>}
\left(\frac{p(\tau\gamma)p(\gamma)}{q(\tau\gamma)q(\gamma)}\right)^{\<\alpha,\tau\gamma\>}. \]
The two exponents are negatives of each other, because $\tau^{-1}\alpha=-\alpha$, and hence above product is equal to
\[ \left(\frac{p(\tau\gamma)p(\tau^{-1}\gamma)}{q(\tau\gamma)q(\tau^{-1}\gamma)}\right)^{\<\alpha,\gamma\>}. \]
Our assumption that the pairs $\{\gamma,-\gamma\}$ and $\tau\{\gamma,-\gamma\}$ are distinct now implies that neither of $\tau\gamma$ and $\tau^{-1}\gamma$ can be equal to $\pm\gamma$, and thus the fraction reduces to $1$.
\end{proof}

Returning to the proof of Proposition \ref{pro:torinvvan}, the Lemma we just proved implies that
\[ f(\alpha) = (-1)^{\zeta_p} \]
for any gauge $p : R(S,G) \rw \{\pm 1\}$. We now choose $p$ as follows. First consider the subset of $R(S,G)$ given by
\[ \{ \beta \in R(S,G)| \<\alpha,\beta^\vee\> \neq 0 \wedge I\beta=-I\beta \}. \]
On this subset, we set
\[ p(\beta) = \tx{sgn}\<\alpha,\beta^\vee\>, \]
while outside of this subset, we choose $p$ arbitrarily. For this gauge we have
\begin{eqnarray*}
p(\tau^{-1}\beta)&=&\tx{sgn}\<\alpha,\tau^{-1}\beta^\vee\>\\
&=&\tx{sgn}\<\tau\alpha,\beta^\vee\>\\
&=&\tx{sgn}\<-\alpha,\beta^\vee\> = -p(\beta).
\end{eqnarray*}
Hence
\[ f(\alpha) = (-1)^{\zeta_p} = \prod_{\substack{\{\beta,-\beta\} \in R(S,G)/\{\pm 1\}\\ I\beta=-I\beta }} (-1)^{\<\alpha,\beta^\vee\>} =
\prod_{\substack{\beta>0\\ I\beta=-I\beta }} (-1)^{\<\alpha,\beta^\vee\>}.  \]
According to \cite[Ch. VI, \S1, no. 10, Prop. 29]{Bou}, we have $\sum_{\beta>0} \beta^\vee = 2\rho^\vee$ with $\<\alpha,2\rho^\vee\> \in 2\Z$, thus
\[ f(\alpha) = \prod_{\substack{\beta>0\\ I\beta \neq -I\beta }} (-1)^{\<\alpha,\beta^\vee\>}.  \]
One checks immediately that the above product is equal to
\[ \prod_{\substack{p(\beta)=+1\\ I\beta \neq -I\beta }} (-1)^{\<\alpha,\beta^\vee\>} = (-1)^{\left\<\alpha,\sum\limits_{\substack{p(\beta)=+1\\ I\beta \neq -I\beta }} \beta^\vee\right\>} \]
for any gauge $p: R(S,G) \rw \{\pm 1\}$. We now choose a gauge $p$ which is constant on each asymmetric $I$-orbit in $R(S,G)$. Then we have
\[ \sum_{\substack{p(\beta)=+1\\ I\beta \neq -I\beta }} \beta^\vee = \sum_{\substack{o \in R(S,G)/I \\ p(o)=+1\\ o \neq -o }}\ \ \sum_{\beta \in o} \beta^\vee. \]
The ellipticity of the inertial action now implies that the inner sum is zero for each $o$.
\end{proof}

\subsection{Relation to the Weil constant} \label{sec:torinvweil}

We continue to assume that $F$ is a non-archimedean local field. Consider the root space decomposition of the Lie algebra of $G$
\[ \mf{g} = \mf{s} \oplus \mf{n},\qquad \mf{n} = \bigoplus_{\alpha \in R(S,G)} \mf{g}_\alpha. \]
Let $B$ be an $N(S,G)(F)$-invariant symmetric bilinear form on $\mf{n}(F)$, and $\psi : F \rw \C^\times$ a non-trivial additive character. Then we can consider the Weil-constant $\gamma_\psi(\mf{n}(F),B)$ \cite[\S VIII]{Wal95}.

\begin{lem} \label{lem:torinvweil} We have the equality
\[ \gamma_\psi(\mf{n}(F),B) = \prod_{\alpha \in R(S,G)_\tx{sym}/\Gamma}f_{(G,S)}(\alpha)\kappa_\alpha(B_\alpha)\lambda_{F_\alpha/F_{\pm\alpha}}(\psi\circ\tx{tr}_{F_\alpha/F_{\pm\alpha}}). \]
Here $\lambda$ is the Langlands $\lambda$-function \cite[Thm. 2.1]{LanArt} (see also \cite[\S1.5]{BH05b}), and $B_\alpha \in F_{\pm\alpha}^\times$ is the number $B(X_\alpha,Y_\alpha)$ where $X_\alpha \in \mf{g}_\alpha(F_\alpha)$ and $Y_\alpha \in \mf{g}_{-\alpha}(F_\alpha)$ are any elements with $[X_\alpha,Y_\alpha]=H_\alpha$.
\end{lem}
\begin{proof}
We have a decomposition of vector spaces defined over $F$
\[ \mf{n} = \bigoplus_{ \mc{O} \in R(S,G)/\Gamma} \mf{n}_\mc{O},\qquad\tx{where}\qquad \mf{n}_\mc{O} = \bigoplus_{\alpha \in \mc{O}} \mf{g}_\alpha. \]
Hence
\[ \gamma_\psi(\mf{n}(F),B) = \prod_{\mc{O} \in R(S,G)/\Gamma} \gamma_\psi(\mf{n}_\mc{O}(F),B). \]
If $\mc{O}$ is asymmetric, the space $\mf{n}_\mc{O}(F)$ is $B$-isotropic, and thus $\gamma_\psi(\mf{n}_\mc{O}(F),B)=1$. If $\mc{O}$ is symmetric, then choosing a root $\alpha \in \mc{O}$ and an element $X_\alpha \in \mf{g}_\alpha(F_\alpha)$ we obtain an isomorphism of $F$-vector spaces
\[ F_\alpha \rw \mf{n}_\mc{O}(F),\qquad x \mapsto \sum_{\sigma \in \Gamma/\Gamma_\alpha} \sigma(xX_\alpha). \]
The pull-back of $B$ to $F_\alpha$ under this isomorphism is equal to
\[ (x,y) \mapsto \tx{tr}_{F_{\pm\alpha}/F}[(x\tau(y)+y\tau(x))B(X_\alpha,\tau X_\alpha)], \]
where $\tau \in \Gamma_{\pm\alpha} \sm \Gamma_\alpha$, and hence
\[ \gamma_\psi(\mf{n}_\mc{O}(F),B) = \gamma_{\tilde\psi}(F_\alpha,(x,y) \mapsto (x\tau(y)+y\tau(x))), \]
where $\tilde\psi : F_{\pm\alpha} \rw \C^\times$ is the character $x \mapsto \psi(\tx{tr}_{F_{\pm\alpha}/F}(B(X_\alpha,\tau X_\alpha)\cdot x))$. The calculation in the proof of \cite[Lemma 1.2]{JL70} shows then that
\[ \gamma_\psi(\mf{n}_\mc{O}(F),B) = \lambda_{F_\alpha/F_{\pm\alpha}}(\tilde\psi). \]
To complete the proof, we note that $B(X_\alpha,\tau X_\alpha) = f_{(G,S)}(\alpha) B_\alpha$ and then use the properties of the Langlands constant.
\end{proof}

\begin{cor} \label{cor:weilinner} Assume that the maximal torus $S$ is elliptic. Let $\xi : G \rw G'$ be an inner twist and $S' \subset G'$ a maximal torus. Assume that $\xi$ restricts to an isomorphism $S \rw S'$ defined over $F$. Let $B$ be a symmetric $\tx{Ad}$-invariant non-degenerate bilinear forms on $\mf{g}(F)$, and let $B'$ be its transfer to $\mf{g'}(F)$ as in \cite{Wal95}. Then
\[ \gamma_\psi(\mf{g}(F),B)\gamma_\psi(\mf{g'}(F),B') = e(G)e(G'). \]
\end{cor}
\begin{proof} This follows at once from Lemma \ref{lem:torinvweil} and Proposition \ref{pro:torinvstab}. \end{proof}

\subsection{A result of Kottwitz on $\epsilon$-factors}

This is a convenient place to review a recent result of Kottwitz on the relationship between Weil-constants and epsilon factors. In this subsection only, we assume that $F$ is any non-archimedean local field, without assumptions on its characteristic or its residual characteristic, and denote by $\ol{F}$ a fixed separable closure of $F$. On the space $\mf{n}$ introduced in the last section one can define a canonical $N(S,G)$-invariant $\ol{F}$-valued quadratic form. Namely, on each plane $\mf{g}_\alpha \oplus \mf{g}_{-\alpha}$ it is defined by the rule
\[ (X,Y) \mapsto \frac{[X,Y]}{H_\alpha}. \]
This assignment is unchanged if we replace $\alpha$ by $-\alpha$, and on the space $\mf{n}$, which is the direct sum of these hyperplanes for all $\{\alpha,-\alpha\} \in R(S,G)/\{\pm 1\}$, we take the sum of the individual forms defined above. One sees easily that this quadratic form is $\Gamma$-equivariant, and hence defines an $F$-valued quadratic form on $\mf{n}(F)$. We will call this form $\tx{can}$. Kottwitz's result then is the following.

\begin{thm} Let $T$ be a minimal Levi in the quasi-split inner form of $G$. Then
\[ \epsilon_L(X^*(S)_\C-X^*(T)_\C,\psi) = e(G)\gamma_\psi(\mf{n}(F),\tx{can}), \]
where $\epsilon_L$ is the Langlands normalization \cite[\S3.6]{Tat77} of the $\epsilon$-factor of the degree-0 virtual $\Gamma$-representation $X^*(S)_\C-X^*(T)_\C$.
\end{thm}

Combining this Theorem with Lemma \ref{lem:torinvweil}, we obtain the following Corollary.
\begin{cor} \label{cor:torinveps} Let $T$ be a minimal Levi in the quasi-split inner form of $G$. Then
\[ \epsilon_L(X^*(S)_\C-X^*(T)_\C,\psi) = e(G)\prod_{\alpha \in R(S,G)_\tx{sym}/\Gamma}f_{(G,S)}(\alpha)\lambda_{F_\alpha/F_{\pm\alpha}}(\psi\circ\tx{tr}_{F_\alpha/F_{\pm\alpha}}). \]
\end{cor}

\subsection{The character associated to the toral invariant} \label{sec:torinvchar}

We continue to assume that $F$ is a non-archimedean local field and assume now moreover that the residual characteristic of $F$ is not $2$. Using the toral invariant $f_{(G,S)}$ we will define a character
\[ \epsilon_f : S(F) \rw \C^\times. \]
For this, we first give for each $\alpha \in R(S,G)$ a character $\epsilon_\alpha : F_\alpha^\times \rw \C^\times$ as follows. If $\alpha$ is asymmetric, we take $\epsilon_\alpha=1$. If $\alpha$ is symmetric and inertially asymmetric, then we consider the group $k_{F_\alpha}^\times/k_{F_{\pm\alpha}}^\times$. This is a cyclic group of even order, and we take on it the unique character which sends any generator to $f_{(G,S)}(\alpha)$. We then inflate this character to $O_{F_\alpha}^\times$ and take $\epsilon_\alpha$ to be its unique extension to $F_{\alpha}^\times$ which restricts trivially to $F_{\pm\alpha}^\times$. Finally, if $\alpha$ is inertially symmetric, we take the unramified character on $F_\alpha^\times$ given by $\epsilon_\alpha(x) =  f_{(G,S)}(\alpha)^{\tx{val}_{F_\alpha}}(x)$.

The set of characters $\{\epsilon_\alpha|\ \alpha \in R(S,G)\}$ obtained in this way satisfies the assumptions of \cite[Cor 2.5.B]{LS87} and thus gives rise to a character $\epsilon_f$ of $S(F)$.

\begin{lem} \label{lem:torinvchar} Assume that the action of $I$ on $X^*(S)$ is tame and generated by a regular elliptic element. Then for every $\gamma \in S(F)$ whose root values are topologically semi-simple, we have
\[ \epsilon_f(\gamma) = \prod_{\substack{\alpha \in R(S,G)_\tx{sym}/\Gamma \\ \alpha(\gamma) \neq 1}} f_{(G,S)}(\alpha). \]
For every $\gamma$ whose root values are topologically unipotent, we have $\epsilon_f(\gamma)=1$.
\end{lem}
\begin{proof}
Following the argument of the proof of \cite[Lemma 3.5.A]{LS87}, we see that
\[ \epsilon_f(\gamma) = \prod_{\alpha \in R(S,G)_\tx{asym}/(\Gamma \times \{\pm 1\})} \epsilon_\alpha(\gamma^\alpha)^{-1} \prod_{\alpha \in R(S,G)_\tx{sym}/\Gamma} \epsilon_\alpha(\delta^\alpha)^{-1} \]
for certain elements $\gamma^\alpha,\delta^\alpha \in F_\alpha^\times$. According to the choice of $\epsilon_\alpha$, the first product vanishes. Moreover, by Proposition \ref{pro:torinvvan}, the second product runs only over inertially symmetric orbits. To evaluate that product, we need to describe the element $\delta^\alpha$. For this, we consider the exact sequence
\[ \xymatrix{ 1\ar[r]&F_{\pm\alpha}^\times\ar[r]&F_\alpha^\times\ar[r]^{1-\tau}&F_\alpha^\times\ar[r]^{1+\tau}&F_{\pm\alpha}^\times}, \]
where $\tau \in \Gamma_{\pm\alpha} \sm \Gamma_\alpha$. The element $\alpha(\gamma) \in F_\alpha^\times$ belongs to the kernel of $1+\tau$ and we let $\delta_\alpha \in F_\alpha^\times$ be any element which maps to $\alpha(\gamma)$ under $1-\tau$.

If $\alpha(\gamma)$ is topologically unipotent, then so is $\delta^\alpha$, but then $\epsilon_\alpha(\delta^\alpha)=1$. This shows the second claim.
If $\alpha(\gamma)$ is topologically semi-simple, then since $F_\alpha/F_{\pm\alpha}$ is ramified, this means that $\alpha(\gamma) \in F_{\pm\alpha}^\times$. Belonging to the kernel of $1+\tau$ then implies $\alpha(\gamma) \in \{\pm 1\}$. If $\alpha(\gamma)=1$, we may take $\delta^\alpha=1$, and if $\alpha(\gamma)=-1$, we may take $\delta^\alpha=\omega$, where $\omega \in F_\alpha^\times$ is a uniformizer such that $\omega^2 \in F_{\pm\alpha}$. Then we see that
\[ \epsilon_\alpha(\delta^\alpha) = \begin{cases} f_{(G,S)}(\alpha)&,\alpha(\gamma) \neq 1 \\ 1&,\tx{else} \end{cases}. \]
It follows that
\[ \epsilon_f(\gamma) = \prod_{\alpha \in R(S,G)_\tx{insym}/\Gamma} f_{(G,S)}(\alpha). \]
Applying again Proposition \ref{pro:torinvvan}, we see that this expression is equal to the right hand side of the claimed equality.
\end{proof}

\section{Construction of epipelagic $L$-packets} \label{sec:packs}

\subsection{Construction of $L$-packets} \label{sec:lpackconst}
We now assume that $G$ is quasi-split. Let $\hat G$ be the complex dual group of $G$. Let us recall what data this entails (see also \cite[\S1]{Kot84}). First, $\hat G$ comes equipped with an action of the Galois group $\Gamma$ by algebraic automorphisms. Furthermore there exist splittings $(T,B,\{X_\alpha\})$ and $(\hat T,\hat B,\{X_{\alpha^\vee}\})$ of $G$ and $\hat G$ respectively, which are fixed by the action of $\Gamma$, and an isomorphism of $\Gamma$-modules $X_*(T) \rw X^*(\hat T)$ which identifies the $B$-simple coroots with the $\hat B$-simple roots. For any other two $\Gamma$-fixed splittings $(T^1,B^1\{X_\alpha^1\})$ and $(\hat T^1,\hat B^1,\{X_{\alpha^\vee}^1\})$ there exists a unique isomorphism $X_*(T^1) \rw X^*(\hat T^1)$ induced by the one for $T$ and $\hat T$. The splittings themselves are not part of the data of $\hat G$, but the compatible system of isomorphisms $X_*(T) \rw X^*(\hat T)$ is.

Let $^LG = \hat G \rtimes W_F$ be the Weil-form of the $L$-group of $G$. We consider Langlands parameters
\[ \phi : W_F \rw {^LG}, \]
which are subject to the following conditions:
\begin{cnd}\ \\[-20pt] \label{cnd:parm}
\begin{enumerate}
\item $\hat T = \tx{Cent}(\phi(P_F),\hat G)$ is a maximal torus of $\hat G$ belonging to a $\Gamma$-fixed splitting.
\item The image of $\phi(I_F)$ in $\Omega(\hat T,\hat G) \rtimes I_F$ is generated by a regular elliptic element.
\item If $w \in I_F^{\frac{1}{m}+}$, where $m$ is the order of the regular elliptic element, then $\phi(w)=(1,w)$.
\end{enumerate}
\end{cnd}

We will call such Langlands parameters \emph{epipelagic}. In this section we are going to construct to each epipelagic Langlands parameter $\phi$ and each inner twist $\xi : G \rw G'$ a packet $\Pi_{\phi,\xi,G'}$ of epipelagic representations of $G'(F)$. We will argue that $\Pi_{\phi,\xi,G'}$ depends only on the $\hat G$-conjugacy class of $\phi$.

Let $\hat S$ be the Galois-module whose underlying abelian group is the complex torus $\hat T$ and whose Galois action is provided by
\[ \phi : W_F \rw N(\hat T,\hat G)\rtimes W_F \rw \Omega(\hat T,\hat G) \rtimes W_F \rw \tx{Aut}_\tx{alg}(\hat T). \]
It is clear that this action factors through a finite quotient of $W_F$, and this quotient is the Galois group of a finite tamely ramified extension $K/F$. Let $S$ be the algebraic torus defined over $F$ whose complex dual is $\hat S$. The splitting field of $S$ is $K$.

In Section \ref{sec:chispec} we are going to construct a $\hat G$-conjugacy class of embeddings ${^Lj} : {^LS} \rw {^LG}$. These embeddings will be tamely ramified, in the sense that ${^Lj}(1,w)=(1,w)$ for all $w \in P_F$. Given this conjugacy class, there exists an element $^Lj$ in it satisfying the following:
\begin{cnd}\label{cnd:ljx}\ \\[-20pt]
\begin{itemize}
\item $^Lj(\hat S) = \hat T$.
\item The following two group homomorphisms are equal:
\[\xymatrix{
W_F\ar[rr]^-\phi&&N(\hat T,\hat G)\rtimes W_F\ar[r]&\Omega(\hat T,\hat G) \rtimes W_F\\
W_F\ar[r]^-{\iota_2}&\hat S \rtimes W_F\ar[r]^-{^Lj}&N(\hat T,\hat G)\rtimes W_F\ar[r]&\Omega(\hat T,\hat G) \rtimes W_F
}\]
\end{itemize}
\end{cnd}

These properties imply that the image of $^Lj$ contains the image of $\phi$, and thus we obtain a factorization
\[ \phi = {^Lj} \circ \phi_{S,{^Lj}}. \]
The local Langlands correspondence for tori attaches to $\phi_{S,{^Lj}}$ a character
\[ \chi_{S,{^Lj}} : S(F) \rw \C^\times. \]
The character $\chi_{S,{^Lj}}$ depends on the choice of $^Lj$ within its $\hat G$-conjugacy class. The dependence is as follows.
\begin{lem} \label{lem:changelj} Let $^Lj_1,{^Lj_2} : {^LS} \rw {^LG}$ be $\hat G$-conjugate and satisfy Conditions \ref{cnd:ljx}. Then ${^Lj_2} = \tx{Ad}(n){^Lj_1}$ for
an element $n \in N(\hat T,\hat G)$ whose projection to $\Omega(\hat T,\hat G)$ is fixed by ${^Lj_1}(\Gamma)$. If $h \in \Omega(S,G)(F)$ is the element with
\[ {^Lj_1}\circ\hat{\tx{Ad}}(h) = \tx{Ad}(n)\circ{^Lj_1} \]
then
\[ \chi_{S,{^Lj_2}} = \chi_{S,{^Lj_1}}\circ\tx{Ad}(h^{-1}). \]
\end{lem}
\begin{proof}
The first claim follows directly from the fact that both ${^Lj_i}$ satisfy Conditions \ref{cnd:ljx}. For the second, we compute
\[ {^Lj_1}\circ\phi_{S,{^Lj_1}}=\phi={^Lj_2}\circ\phi_{S,{^Lj_2}} = \tx{Ad}(n){^Lj_1}\circ\phi_{S,{^Lj_2}} = {^Lj_1}\circ\hat{\tx{Ad}}(h)\circ\phi_{S,{^Lj_2}}, \]
from which we conclude that
\[ \chi_{S,{^Lj_2}} = \chi_{S,{^Lj_1}}\circ\tx{Ad}(h^{-1}). \]
\end{proof}

\begin{lem} For every choice of $^Lj$ within its $\hat G$-conjugacy class, every inner twist $\xi : G \rw G'$, and every admissible embedding $j : S \rw G'$ of $S$ into an inner form $G'$ of $G$, the pair $j_*(S,\chi_{S,{^Lj}})$ satisfies the Conditions \ref{cnd:char}.
\end{lem}
\begin{proof}
The first item in \ref{cnd:char} is immediate, the second follows directly from the second item in \ref{cnd:parm}. For the third, we note that since the restriction of $^Lj$ to $P_F$ is trivial, the third item in \ref{cnd:parm} implies that $\phi_{S,{^Lj}}$ restricts trivially to $I_F^{\frac{1}{m}+}$. The claim now follows from \cite[\S7.10]{Yu09}.

We now come to the fourth condition in \ref{cnd:char}. According to Lemma \ref{lem:genchar}, we must analyze the character $\chi_{S,{^Lj}} \circ N$, where $N : S(E) \rw S(F)$ is the norm map and $E/F$ is the splitting extension of $S$. The Langlands parameter of this character is $\phi_{S,{^Lj}}|_{W_E}$. Note that since $^Lj$ is tamely ramified, the restrictions of $\phi_{S,{^Lj}}$ and $\phi$ to wild inertia coincide. For any $\alpha^\vee \in R^\vee(S,G) = R(\hat S,\hat G)$, the character $\chi_{S,{^Lj}}\circ N \circ \alpha^\vee$ is given as the composition
\[ E^\times \stackrel{\tx{Art}^{-1}}{\lrw} W_E/W_E^c \stackrel{\phi_{S,{^Lj}}}{\lrw} \hat S \stackrel{\alpha^\vee}{\lrw} \C^\times \]
and its non-vanishing on $U_E^1$ is equivalent to the non-vanishing of the restriction of $\alpha^\vee\circ\phi_{S,{^Lj}}$ to wild inertia. This is implied by the first item in Conditions \ref{cnd:parm}. Moreover, the statement that no Weyl element stabilizes $\chi\circ N|_{S(E)_\frac{1}{m}}$ is equivalent to the statement that no Weyl element stabilizes the restriction of $\phi_{S,{^Lj}}$ to wild inertia, which also follows from the first item in Conditions \ref{cnd:parm}.

\end{proof}

Now fix an inner twist $\xi : G \rw G'$ and an $L$-embedding $^Lj : {^LS} \rw {^LG}$. For every admissible embedding $j : S \rw G'$, we obtain the pair $j_*(S,\chi_{S,{^Lj}})$ of a maximal torus of $G'$ and a on it, and this pair satisfies Conditions \ref{cnd:char}. Let $\epsilon_j : S(F) \rw \C^\times$ be the pull-back along $j$ of the character $\epsilon_f$ on $jS(F)$ defined in Section \ref{sec:torinvchar}. By Lemma \ref{lem:torinvchar}, the pair $j_*(S,\chi_{S,{^Lj}}\cdot\epsilon_j)$ also satisfies Conditions \ref{cnd:char}. We let
\[ \pi_{j,{^Lj}} = \pi_{j_*(S,\chi_{S,{^Lj}}\cdot\epsilon_j)} \]
be the irreducible supercuspidal representation of $G'(F)$ given by the construction in Section \ref{sec:repconst}. According to Lemma \ref{lem:changelj}, we have for any $h \in \Omega(S,G)(F)$ the relation
\begin{equation} \pi_{j,{^Lj}\circ\hat{\tx{Ad}}(h)} = \pi_{j\circ\tx{Ad}(h),{^Lj}}. \label{eq:changelj} \end{equation}

We now set
\[ \Pi_{\phi,\xi,G'} \]
to be the set isomorphism classes of $\pi_{j,{^Lj}}$, where $j$ runs over the set of rational classes of admissible embeddings $j:S \rw G'$, and ${^Lj} : {^LS} \rw {^LG}$ is a fixed $L$-embedding in its $\hat G$-conjugacy class. The above relation shows that changing the choice of ${^Lj}$ has no effect on the set $\Pi_{\phi,\xi,G'}$. Furthermore, the set $\Pi_{\phi,\xi,G'}$ remains unchanged if we replace $\phi$ by $\tx{Ad}(g)\phi$ for any $g \in \hat G$. For if we at the same time replaced ${^Lj}$ by $\tx{Ad}(g){^Lj}$, the two changes cancel. Thus, we may as well write $\Pi_{\Phi,\xi,G'}$, where $\Phi = \tx{Ad}(\hat G)\phi$.

This completes the construction of the sets $\Pi_{\Phi,\xi,G'}$, apart from the specification of the $\hat G$-conjugacy class of $L$-embeddings ${^LS} \rw {^LG}$. We will turn to this matter shortly, but first we want to show that our packets satisfy Shahidi's generic packet conjecture \cite{Sha90}.

\begin{pro} \label{pro:unigen} Fix a Whittaker datum $(B,\psi_B)$ on $G$. Then the set $\Pi_{\Phi,G}$ contains a unique $(B,\psi)$-generic constituent.
\end{pro}
\begin{proof}
Choose a representative $^Lj : {^LS} \rw {^LG}$ within its $\hat G$-conjugacy class, and choose a representative $j : S \rw G$ within each rational conjugacy class of admissible embeddings. The set of pairs $\{ j_*(S,\chi_{S,{^Lj}}) \}$ is a set of representatives for the rational conjugacy classes in a single stable class. For each $j$, the character $\epsilon_j$ is trivial on $S(F)_{0+}$. Thus the set of pairs $\{ j_*(S,\chi_{S,{^Lj}}\cdot\epsilon_j|_{S(F)_{0+}}) \}$ is still set of representatives for the rational conjugacy classes in a single stable class.
Recall from the discussion preceding Proposition \ref{pro:generic} that we can associate to each pair $j_*(S,\chi_{S,{^Lj}}\cdot\epsilon_j)$ an element $Y_j \in j\mf{s}(F)$ which describes the restriction of $j_*(\chi_{S,{^Lj}}\cdot\epsilon_j)$ to $jS(F)_{0+}$. Then the set $\{ Y_j \}$ also forms a set of representatives for the rational classes in a single stable class. The claim now follows from Proposition \ref{pro:generic}.
\end{proof}

\subsection{Construction of the $L$-embedding $^Lj : {^LS} \rw {^LG}$} \label{sec:chispec}
We maintain the notation that $\phi : W_F \rw {^LG}$ is a Langlands parameter satisfying Conditions \ref{cnd:parm}, $\xi : G \rw G'$ is an inner twist, and $j : S \rw G'$ is an admissible embedding of the torus $S$ constructed from $\phi$ in the previous subsection into $G'$.

Associated to $j$ is a canonical $\hat G$-conjugacy class of embeddings $\hat j : \hat S \rw \hat G$ of complex algebraic groups. In \cite{LS87}, Langlands and Shelstad have described a procedure which provides a $\hat G$-conjugacy class $[{^Lj}]$ of $L$-embeddings ${^Lj} : {^LS} \rw {^LG}$ extending $\hat j$. However, the construction of $[{^Lj}]$ requires the choice of auxiliary data, which they call $\chi$-data. Different choices of $\chi$-data lead to different $\hat G$-conjugacy classes $[{^Lj}]$, and in our construction this would lead to different characters $\chi_{S,{^Lj}}$, hence to different representations $\pi_j$. In this subsection we will describe the correct choice of $\chi$-data.

Recall that $W_F$ and $I_F$ denote the Weil group and inertia group of $F$. For convenience, we will drop the subscript $F$. We consider the action of $W$ on $R(\hat T, \hat G)$ given by $\phi$. Let $\alpha \in R(\hat T,\hat G)$ and put
\[ W_\alpha = \tx{Stab}(W,\alpha)\qquad W_{\pm\alpha} = \tx{Stab}(W,\{\alpha,-\alpha\}). \]
Denote the fixed fields for the action of these groups on $\ol{F}$ by $F_\alpha$ and $F_{\pm\alpha}$. Then $F_\alpha/F_{\pm\alpha}$ is an extension of degree at most $2$. A set of $\chi$-data is a set $\{ \chi_\alpha | \alpha \in R(S,G) \}$, where each $\chi_\alpha$ is a character
\[ \chi_\alpha : F_\alpha^\times \rw \C^\times, \]
and such that the set $\{\chi_\alpha\}$ satisfies the conditions of \cite[\S2.5]{LS87}. The most important of these conditions is that $\chi_\alpha$ be trivial on $N_{F_\alpha/F_{\pm\alpha}}(F_\alpha^\times)$ (this is slightly stronger than the original condition, but it is what we will use).

If $\alpha$ is asymmetric, we are forced to take $\chi_\alpha=1$. If $\alpha$ is symmetric but inertially asymmetric, then $F_\alpha/F_{\pm\alpha}$ is unramified, and there exists a unique unramified character $\chi_\alpha$ satisfying the imposed condition. If $\alpha$ is inertially symmetric, and if we require that $\chi_\alpha$ be tamely-ramified (unramified is not an option any more), then there are exactly two characters satisfying the imposed condition. It is between these two that we need to choose, and we will use the arithmetic information encoded in $\phi$ to do so.

It will be enough to specify the character on an arbitrary uniformizer $\omega \in F_\alpha^\times$, since these elements generate the multiplicative group $F_\alpha^\times$. To that end, consider the restriction of $\phi$ to the wild inertia subgroup $P$. Its image belongs to $\hat T$. Composing this restriction with the root $\alpha$ we obtain a homomorphism.
\[ \xymatrix{ P\ar@{^{(}->}[r]&W\ar[r]^\phi&\hat T\ar[r]^\alpha&\C^\times}. \]
It can be shown that this homomorphism extends to $W_{F_\alpha}$, and hence provides by local class field theory a homomorphism
\[ \xi_\alpha : U_{F_\alpha}^1 \rw \C^\times. \]
By assumption on $\phi$, this homomorphism is trivial on $U_{F_\alpha}^2$. Using the uniformizer $\omega$, we obtain a character
\[ \xi_{\alpha,\omega} : \xymatrix{ k_{F_\alpha}\ar[rr]^-{x\mapsto \omega x+1}&&U_{F_\alpha}^1/U_{F_\alpha}^2\ar[r]^-{\xi_\alpha}&\C^\times}. \]
With this, we define
\[ \chi_\alpha(\omega) = \lambda_{F_\alpha/F_{\pm\alpha}}(\xi_{\alpha,\omega})^{-1}, \]
where $\lambda$ is the Langlands $\lambda$-function \cite[Thm. 2.1]{LanArt} (see also \cite[\S1.5]{BH05b}). This concludes the construction of characters $\chi_\alpha : F_\alpha^\times \rw \C^\times$.

\begin{lem} The set $\{\chi_\alpha|\ \alpha \in R(\hat T,\hat G)\}$ just constructed is a set of $\chi$-data for the action of $W$ on $R(\hat T,\hat G)$ given by $\phi$. \end{lem}
\begin{proof}
Let $\alpha \in R(\hat T,\hat G)$ and $\sigma \in \Gamma$. We need to show the following points
\begin{itemize}
\item $\chi_{-\alpha} = \chi_\alpha^{-1}$
\item $\chi_{\sigma\alpha} = \chi_\alpha\circ\sigma^{-1}$
\item $\chi|_{F_{\pm\alpha}^\times}$ is equal to the character corresponding to the extension $F_\alpha/F_{\pm\alpha}$ by local class field theory.
\end{itemize}
All of these points are obvious when the extension $F_\alpha/F_{\pm\alpha}$ is either trivial or unramified, so we focus on the case when this extension is ramified. First, replacing $\alpha$ by $-\alpha$ replaces $\xi_\alpha$ by $\xi_\alpha^{-1}$, hence $\xi_{\alpha,\omega}$ by $\xi_{\alpha,\omega}\circ(x \mapsto -x)$, and hence multiplies $\lambda_{F_\alpha/F_{\pm\alpha}}(\xi_{\alpha,\omega}^{-1})$ by $\left(\frac{-1}{q}\right)$. Thus we obtain
\[ \chi_{-\alpha}(\omega) = \left(\frac{-1}{q}\right)\chi_\alpha(\omega). \]
But it is known that $\lambda_{F_\alpha/F_{\pm\alpha}}^2=\left(\frac{-1}{q}\right)$, and the first point follows. The second point follows from the fact that $\sigma$ provides an isomorphism of extensions $F_\alpha/F_{\pm\alpha} \rw F_{\sigma\alpha}/F_{\pm\sigma\alpha}$ which transports $\xi_{\alpha,\omega}$ to $\xi_{\sigma\alpha,\sigma\omega}$ and from Fact \ref{fct:torinvgal}. For the third point, we must show that $\chi_\alpha$ restricts to the Legendre symbol on the Teichm\"uller representatives in $O_{F_\alpha}^\times$, and kills $-\omega^2$ for each uniformizer $\omega$. The first claim follows from the fact that multiplying $\omega$ by the Teichm\"uller representative of $u \in k_{F_\alpha}^\times$ replaces by $\xi_{\alpha,\omega}$ by $\xi_{\alpha,\omega}(u\cdot)$, while the second is equivalent to the claim that $\chi_\alpha(\omega^2)=\left(\frac{-1}{q}\right)$.

\end{proof}

\subsection{Parameterization of $L$-packets} \label{sec:lpackparm}

Our goal in this section is to provide a parameterization of the packets $\Pi_{\Phi,\xi,G'}$ from Section \ref{sec:lpackconst} in terms of $\hat G$ in accordance with the local Langlands conjecture. As has already been observed by Vogan \cite{Vog93}, the notion of an inner form is not rigid enough to allow for such a parameterization. This necessitates a rigidification of this notion. In other words, we must endow an inner form with additional structure and consider two inner forms as different if their additional structures differ, even if the underlying inner forms are the same. We will use the concept of extended pure inner forms, originally due to Kottwitz \cite{Kot85,Kot97}. We refer the reader to \cite{Kal11a} for an exposition of this notion, as well as the related notions of rational and stable conjugacy across extended pure inner forms.

The first step towards a parameterization is to replace the sets $\Pi_{\Phi,G'}$ with the single set $\Pi_\Phi$ consisting of equivalence classes of quadruples $(G^b,\xi,b,\pi)$ where $(\xi,b) : G \rw G^b$ is an extended pure inner twist (in particular $\xi : G \rw G^b$ is an inner twist), and $\pi \in \Pi_{\Phi,G^b}$. The notion of equivalence of quadruples, which is explained in \cite{Kal11a}, is such that if we fix an extended pure inner twist $(\xi,b) : G \rw G^b$, then the subset of $\Pi_\Phi$ consisting of elements corresponding to that extended pure inner twist is precisely the set $\Pi_{\Phi,\xi,G^b}$ discussed in Section \ref{sec:lpackconst}.

For each $\phi \in \Phi$ we have the diagonalizable group $S_\phi = \tx{Cent}(\phi,\hat G)$. Given $\phi,\phi' \in \Phi$, there exists $g \in \hat G$ with $\tx{Ad}(g)\phi=\phi'$, and this element provides an isomorphism $\tx{Ad}(g) : S_\phi \rw S_{\phi'}$. This isomorphism is independent of the choice of $g$, because $S_\phi$ is abelian. Thus we obtain a system $\{S_\phi\}_{\phi \in \Phi}$ of diagonalizable groups together with transition isomorphisms, and we call it $S_\Phi$.

In this section we are going to construct, given a Whittaker datum $(B,\psi_B)$ for $G$, a canonical bijection
\[ X^*(S_\Phi) \rw \Pi_\Phi. \]
Fix $\phi \in \Phi$. Choose any $L$-embedding ${^Lj}: {^LS}  \rw {^LG}$ within the $\hat G$-conjugacy class determined in Section \ref{sec:chispec} and satisfying Conditions \ref{cnd:ljx}. It provides a character $\chi_{S,{^Lj}} : S(F) \rw \C^\times$ as in Section \ref{sec:lpackconst}. It provides furthermore an isomorphism of diagonalizable groups
\[ \hat S^\Gamma \rw \tx{Cent}(\phi,\hat G). \]
Composing this isomorphism with the Kottwitz isomorphism we obtain an isomorphism
\begin{equation} \label{eq:d1} X^*(\tx{Cent}(\phi,\hat G)) \rw X^*(\hat S^\Gamma) = X_*(S)_\Gamma \rw \tb{B}(S). \end{equation}
According to Proposition \ref{pro:unigen} we can find an embedding $j_0 : S \rw G$ in the stable class dual to $[{^Lj}]$, such that the representation $\pi_{j_0,\chi_{S,^Lj}}$ is the unique generic constituent of the packet $\Pi_{\phi,\tx{id},G}$.

Given $\rho \in X^*(\tx{Cent}(\phi,\hat G))$, let $b_\rho \in E(S,Z)$ be any element mapping in $B(S)$ to the image of $\rho$ under \eqref{eq:d1}. We use $j_0$ to further map $b_\rho$ to $E(G,Z)$. Let $\xi_\rho : G \rw G^{b_\rho}$ be the corresponding inner twist. Then $j_\rho := \xi_\rho\circ j_0 : S \rw G'$ is an embedding defined over $F$, and we obtain the representation
\[ \pi_\rho := \pi_{j_\rho} \]
of the group $G'(F)$.

\begin{pro} The map
\[ \rho \mapsto (G^{b_\rho},\xi_\rho,b_\rho,\pi_\rho) \]
provides a well-defined bijection
\[ X^*(S_\Phi) \rw \Pi_\Phi \]
which is independent of the choice of $^Lj$ within its $\hat G$-conjugacy class.
\end{pro}

\begin{proof}
The fact that the equivalence class of the quadruple $(G^{b_\rho},\xi_\rho,b_\rho,\pi_\rho)$ depends only on the image of $b_\rho$ in $\tb{B}(S)$, and hence only on $\rho$, follows from the same argument as in the proof of \cite[Lemma 3.3.2]{Kal11a} and will be taken as proved. We will now show that, for a fixed $\phi$, the map $X^*(S_\phi) \rw \Pi_\Phi$ constructed above is independent of the choice of $^Lj$ within its conjugacy class. Applying Lemma \ref{lem:changelj} we see that any other $L$-embedding is of the form ${^Lj}\circ\hat{\tx{Ad}(h)}$. This latter $L$-embedding then gives rise, according to the same lemma, to the character $\chi_{S,{^Lj}}\circ\tx{Ad}(h^{-1}) = \tx{Ad}(h)_*\chi_{S,{^Lj}}$. We now have
\[ [j_0]_*(S,\chi_{S,{^Lj}}\cdot\epsilon_{j_0}) = [j_0\circ\tx{Ad}(h)]_*(S,\chi_{S,{^Lj}\circ\hat{\tx{Ad}}(h)}\cdot\epsilon_{j_0}) \]
and we see that changing ${^Lj}$ to ${^Lj}\circ\hat{\tx{Ad}(h)}$ has the effect of changing $j_0$ to $j_0\tx{Ad}(h)$. One sees immediately that the element $b_\rho$ associated to $\rho$ remains unchanged, and then so does the quadruple $(G^{b_\rho},\xi_\rho,b_\rho,\pi_\rho)$.

We have just shown that the map $X^*(S_\phi) \rw \Pi_\Phi$ depends only on $\phi$, and not on the choice of $^Lj$. The fact that it is bijective follows from Fact \ref{fct:repequiv} and the fact that $\tb{B}(S)$ classifies the rational conjugacy classes of embeddings of $S$ into the extended pure inner forms of $G$.

Finally, we claim that the system of maps $X^*(S_\phi) \rw \Pi_\Phi$ we have constructed for each $\phi \in \Phi$ provide a well-defined map $X^*(S_\Phi) \rw \Pi_\Phi$. This means that for any two $\phi,\phi'$ we have the diagram
\[ \xymatrix{X^*(S_\phi)\ar[rd]\ar[dd]\\&\Pi_\Phi\\X^*(S_\phi')\ar[ru] }\]
where the diagonal maps are given by the above construction, and the vertical map comes from the transition isomorphism $S_\phi \rw S_{\phi'}$. To that end, let $g \in \hat G$ be such that $\tx{Ad}(g)\phi = \phi'$. If $\hat S$ and $\hat S'$ are the $\Gamma$-modules associated to $\phi$ and $\phi'$ in the construction in Section \ref{sec:lpackconst}, then $\tx{Ad}(g)$ provides an isomorphism $\hat S \rw \hat S'$. If $^Lj_X : {^LS} \rw {^LG}$ is an embedding satisfying Conditions \ref{cnd:ljx} with respect to $\phi$, then $\tx{Ad}(g){^Lj_X}$ is one satisfying these conditions with respect to $\phi'$. In particular, we have $\phi_{S'}=\tx{Ad}(g)\phi_S$. The claim now follows.

\end{proof}

\subsection{Compatibility with the formal degree conjecture}
Fix an inner twist $\xi :G \rw G'$. In this section we will show that the $L$-packets $\Pi_{\xi,G'}$ satisfy the formal degree conjecture of Hiraga-Ichino-Ikeda \cite{HII08}. Let us briefly recall the statement of the conjecture. Let $\psi : F \rw \C^\times$ be a non-trivial character, and let $\mu_{G'/A,\psi}$ be the Haar measure on $G'(F)/A(F)$ defined in \cite{GG99}. Here $A$ is the maximal split torus in the center of $G$. Then the conjecture states that for each $\pi \in \Pi_{\xi,G'}$,
\[ d(\pi; \mu_{G'/A,\psi}) = \frac{\<1,\pi\>}{|\mc{S}_\phi^\natural|}|\gamma(0,\tx{Ad}\circ\phi,\psi)|. \]
Here, the left hand side is the formal degree if $\pi$ relative to the given Haar measure. On the right hand side, $\tx{Ad}$ denotes the representation of ${^LG}$ on the vector space $\tx{Lie}(\hat G)/\tx{Lie}(Z(\hat G)^\Gamma)$, and $\gamma(s,\tx{Ad}\circ\phi,\psi)$ is the $\gamma$-factor of the $W_F$-representation $\tx{Ad}\circ\phi$. The group $\mc{S}_\phi^\natural$ is the component group of the intersection of $S_\phi$ with the subgroup of $\hat G$ which is dual to $G/A$. Finally, $\<1,\pi\>$ is the value at $1$ of the character $\<\cdot,\pi\>$ of an irreducible representation conjecturally assigned to $\pi$. Since in our case the relevant finite group is abelian, this quantity is always equal to $1$.

The main result of this section is the following.
\begin{thm} \label{thm:formal} The $L$-packet $\Pi_{\xi,G'}$ satisfies the formal degree conjecture. \end{thm}
\begin{proof}

We will first compute the right hand side. Recall that by definition,
\[ \gamma(s,V,\psi) = \epsilon(s,V,\psi)\frac{L(1-s,V^\vee)}{L(s,V)}. \]
The representation $\tx{Ad}$ being self-dual, we see
\[ \gamma(0,\tx{Ad}\circ\phi,\psi) = \epsilon(0,\tx{Ad}\circ\phi,\psi)\frac{L(1,\tx{Ad}\circ\phi)}{L(0,\tx{Ad}\circ\phi)}. \]

\begin{lem} Assume that $\psi$ is of order zero. Then \label{lem:epsilon}
\[ \log_q |\epsilon(0,\tx{Ad}\circ\phi,\psi)| = \frac{1}{2}(\dim \mf{\hat g} - \dim \mf{\hat z}^I + \#R/e), \]
where $R$ is the absolute root system of $G$ and $e$ is the order of the regular elliptic element giving the action of inertia on $\hat T$ through $\phi$.
\end{lem}
\begin{proof}
We know that $\epsilon(0,V,\psi)=w(V)q^{a(V)(\frac{1}{2}-s)}$ where $w(V)$ is the root number (being equal to $\det V(-1)$ when $V$ is self-dual), and $a(V)$ is the Artin conductor. For the representation $V=\tx{Ad}\circ\phi$, the lower ramification filtration ends at $D_2=\{1\}$, and we have
\[ \mf{\hat g}^{D_1} = \mf{\hat t}\qquad\tx{and}\qquad \mf{\hat g}^{D_0} = \mf{\hat z}^I, \]
according to Conditions \ref{cnd:parm}. It follows that
\begin{eqnarray*}
a(V)&=&\dim([\mf{\hat g}/\mf{\hat z}^\Gamma]/[\mf{\hat z}^I/\mf{\hat z}^\Gamma])+\frac{1}{e}\dim([\mf{\hat g}/\mf{\hat z}^\Gamma]/[\mf{\hat t}/\mf{\hat z}^\Gamma]) \\
&=&\dim(\mf{\hat g})-\dim(\mf{\hat z}^I)+\#R/e
\end{eqnarray*}
\end{proof}

Let $M=X_*(Z/A)^I$. This is a finite-rank free $\Z$-module with Frobenius action. In fact, it is the cocharacter module of the maximal reductive quotient $\ts{Z}/\ts{A}$ of the special fiber of the Iwahori group scheme of $[Z/A](F)$. In particular, $(M \otimes \ol{k_F}^\times)^\tx{Fr}$ is the set of $k_F$-points of $\ts{Z}/\ts{A}$.

\begin{lem} \label{lem:dualcfrmod} The $\C[\tx{Fr}]$-modules $M \otimes \C$ and $\mf{\hat z}^I/\mf{\hat z}^\Gamma$ are dual to each other.
\end{lem}
\begin{proof}
The $F$-torus $Z^\circ$ has as its complex dual torus $\hat Z/\hat Z_\tx{der}$. Thus the $\C[\Gamma]$-module $X_*(Z^\circ)\otimes \C$ is dual to $X_*(\hat Z/\hat Z_\tx{der}) \otimes \C$. The inclusion $X_*(\hat Z) \rw X_*(\hat Z/\hat Z_\tx{der})$ becomes an isomorphism after tensoring with $\C$, and we see that the dual of the $\C[\Gamma]$-module $X_*(Z) \otimes \C$ is $\mf{\hat z}$. Thus $M\otimes \C$, being equal to $\tx{cok}([X_*(Z)\otimes \C]^\Gamma \rw [X_*(Z) \otimes \C]^I)$, is dual to $\tx{ker}(\mf{\hat z}_I \rw \mf{\hat z}_\Gamma)$, which itself is equal to $\tx{cok}(\mf{\hat z}^\Gamma \rw \mf{\hat z}^I)$.
\end{proof}

\begin{lem} \label{lem:l1}
\[ |L(1,\tx{Ad}\circ\phi)| = q^{\dim(\mf{\hat z}^I/\mf{\hat z}^\Gamma)}\cdot|(M\otimes \ol{k_F}^\times)^\tx{Fr}|^{-1}. \]
\end{lem}
\begin{proof}
We have $L(1,\tx{Ad}\circ\phi) = \tx{det}(1-q^{-1}\tx{Ad}(\phi(\tx{Fr}))|(\mf{\hat g}/\mf{\hat z}^\Gamma)^{\tx{Ad}(\phi(I))})^{-1}$. Arguing as in the proof of Lemma \ref{lem:epsilon}, we see $(\mf{\hat g}/\mf{\hat z}^\Gamma)^{\tx{Ad}(\phi(I))} = (\mf{\hat g})^{D_0}/\mf{\hat z}^\Gamma = \mf{\hat z}^I/\mf{\hat z}^\Gamma$. Thus
\begin{eqnarray*}
L(1,\tx{Ad}\circ\phi)^{-1}&=&\det(1-q^{-1}\tx{Fr}| \mf{\hat z}^I/\mf{\hat z}^\Gamma ) \\
&=&\det(q^{-1}\tx{Fr}(q\tx{Fr}^{-1}-1)| \mf{\hat z}^I/\mf{\hat z}^\Gamma ) \\
&=&q^{-\dim(\mf{\hat z}^I/\mf{\hat z}^\Gamma)}\cdot\det(\tx{Fr}|\mf{\hat z}^I/\mf{\hat z}^\Gamma)\cdot\det(q\tx{Fr}^{-1}-1|\mf{\hat z}^I/\mf{\hat z}^\Gamma)
\end{eqnarray*}
The second of the three factors is a product of roots of unity. The third factor equals up to a sign $\det(1-q\tx{Fr}^{-1}|\mf{\hat z}^I/\mf{\hat z}^\Gamma)$, and using Lemma \ref{lem:dualcfrmod} we obtain
\[ \det(1-q\tx{Fr}^{-1}|\mf{\hat z}^I/\mf{\hat z}^\Gamma) = \det(1-q\tx{Fr}|M\otimes\C) = (M\otimes\ol{k_F}^\times)^\tx{Fr}. \]
\end{proof}

\begin{lem} \label{lem:l0}
\[ |L(0,\tx{Ad}\circ\phi)| = |M_\tx{Fr}|^{-1}. \]
\end{lem}
\begin{proof}
We have
\begin{eqnarray*}
L(0,\tx{Ad}\circ\phi)^{-1}&=&\det(1-\tx{Fr}|\mf{\hat z}^I/\mf{\hat z}^\Gamma)\\
&=&\det(\tx{Fr}|\mf{\hat z}^I/\mf{\hat z}^\Gamma)\cdot\det(\tx{Fr}^{-1}-1|\mf{\hat z}^I/\mf{\hat z}^\Gamma)\\
&=&\det(\tx{Fr}|\mf{\hat z}^I/\mf{\hat z}^\Gamma)\cdot\det(\tx{Fr}-1|M \otimes \C)\\
&=&\det(\tx{Fr}|\mf{\hat z}^I/\mf{\hat z}^\Gamma)\cdot|M/(\tx{Fr}-1)M|
\end{eqnarray*}
The first factor is again a product of roots of unity.
\end{proof}

\begin{lem} \label{lem:snatural}
\[ |\mc{S}^\natural| = |X_*(S/A)_\Gamma|. \]
\end{lem}
\begin{proof}
We have
\[ S^\natural = S_\phi \cap \hat{G/A} = \hat S^\Gamma \cap \hat{G/A} = \hat{S/A}^\Gamma = \tx{Hom}(X_*(S/A)_\Gamma,\C^\times). \]
The torus $S/A$ being anisotropic, the abelian group $X_*(S/A)_\Gamma$ is finite, hence the result.
\end{proof}

Combining the above results, we obtain the following expression for the right hand side of the formal degree conjecture.
\begin{cor} \label{cor:formalrhs}
\[ \frac{|\gamma(0,\tx{Ad}\circ\phi,\psi)|}{|\mc{S}^\natural|} = \frac{q^{\frac{1}{2}(\dim(\mf{\hat g})+\dim(\mf{\hat z}^I)+\#R/e)-\dim(\mf{\hat z}^\Gamma)}}{|M_\tx{Fr}|^{-1}|X_*(S/A)_\Gamma||(M \otimes\ol{k_F}^\times)^\tx{Fr}|}. \]
\end{cor}

We now turn to the computation of the formal degree. The first step is to compare the Haar measure of Gross-Gan \cite{GG99} with the one of DeBacker-Reeder \cite[\S5]{DR09}.

\begin{lem} \label{lem:compmeas} Let $\nu_{G'}$ be the Haar measure on $G'(F)$ normalized as in \cite[\S5]{DR09}. Assume that $\psi$ has order zero and let $\mu_{G'}$ be the Haar measure on $G'(F)$ defined in \cite{GG99} with respect to the $\psi$-self-dual absolute value on $F$. Then
\[ \nu_{G'} = q^{\frac{1}{2}\dim(G)}\mu_{G'}. \]
\end{lem}
\begin{proof}
We take $I^+$ to be the pro-unipotent radical of an Iwahori-subgroup of $G'(F)$. Then according to \cite[(4.11)]{Gr97} and \cite[\S5]{GG99}, we have
\[ -\log_q\vol(I^+,\mu_{G', \psi}) = \frac{1}{2}a(M_{G})+\sum_{d \geq 1} (d-1)\dim(V_d^I) + \tx{rank}_{F^u}G, \]
where
\[ M_G = \bigoplus_{d \geq 1} V_d(1-d) \]
is the motive of $G$ (the quasi-split inner form of $G'$), and
\[ a(M_G) = \sum_{d \geq 1} (2d-1)a(V_d) \]
is its Artin conductor \cite[(4.3)]{GG99}. Since $G$ is tamely-ramified, so are all the $\Gamma$-modules $V_d$, and thus
\[ a(M_G) = \sum_{d \geq 1} (2d-1)\dim(V_d/V_d^I). \]
A short computation reveals
\[ -\log_q\vol(I^+,\mu_{G', \psi}) = \frac{1}{2}\sum_{d \geq 1}\Big((2d-1)\dim(V_d)-\dim(V_d^I)\Big) + \tx{rank}_{F^u}G. \]
On the other hand, we have from \cite[\S1]{Gr97} the formulas
\[ \sum_{d \geq 1}(2d-1)\dim(V_d) = \dim(G)\qquad\sum_{d \geq 1} V_d = E\qquad \tx{rank}_{F^u}G = \dim(E^I), \]
which bring us to
\[ -\log_q\vol(I^+,\mu_{G', \psi}) = \frac{1}{2}(\dim(G)+\dim(E^I)). \]
On the other hand,
\[ \vol(I^+,\nu_{G'}) = [I:I^+]^{-1}\vol(I,\nu_{G'}) = |\tx{Lie}(\ts{I})(k_F)|^{-\frac{1}{2}} = q^{-\frac{1}{2}\dim(\ts{I})}, \]
where $\ts{I}$ is the reductive quotient of the special fiber of the Iwahori group scheme of $G'$. It's dimension is equal to the dimension of $E^I$, and the proof is complete.
\end{proof}

We now proceed to compute the formal degree of a constituent $\pi$ of the $L$-packet $\Pi_{\phi,\xi,G'}$. This representation is of the form (Section \ref{sec:repconst})
\[ \pi = \pi_{S',\chi'} = \textrm{c-ind}_{S'(F)G'(F)_{y,0+}}^{G'(F)} \hat\chi, \]
where $S'$ is a maximal torus of $G'$ which is the image of an admissible embedding $j : S \rw G'$, $y$ is the point in the reduced Bruhat-Tits building of $G'$ determined by $S'$, and $\hat\chi$ is a one-dimensional character of the inducing group $S'(F)G'(F)_{y,0+}$. It follows that
\[ \deg(\pi;\mu_{G'/A,\psi}) = \vol\left(\frac{S'(F)G'(F)_{y,0+}}{A(F)};\mu_{G'/A,\psi}\right)^{-1}. \]
We have the exact sequence
\[ 1 \rw \frac{G'(F)_{y,0+}}{A(F)_{0+}} \rw \frac{S'(F)G'(F)_{y,0+}}{A(F)} \rw \frac{S'(F)}{S'(F)_{0+}A(F)} \rw 1. \]
The final term of this sequence is a finite abelian group. Applying Lemma \ref{lem:compmeas}, we see
\[ \deg(\pi;\mu_{G'/A,\psi}) = q^{\frac{1}{2}\dim(G'/A)}\vol\left(\frac{G'(F)_{y,0+}}{A(F)_{0+}};\nu_{G'/A}\right)^{-1}\left|\frac{S'(F)}{S'(F)_{0+}A(F)}\right|^{-1}. \]
\begin{lem} \label{lem:vol1}
We have
\[ -\log_q\vol\left(\frac{G'(F)_{y,0+}}{A(F)_{0+}};\nu_{G'/A}\right) = \frac{1}{2}(\tx{rk}(M)+\#R/e). \]
\end{lem}
\begin{proof}
By definition,
\[ \vol(G'(F)_{y,0+};\nu_{G'}) = |\tx{Lie}(\ts{G'}_y)(k_F)|^{-\frac{1}{2}}, \]
where $\ts{G'}_y$ is the reductive quotient of the special fiber of the parahoric of $G'$ associated to $y$. Thus, we have
\[ -\log_q\vol(G'(F)_{y,0+};\nu_{G'}) = \frac{1}{2}\dim(\ts{G'}_y). \]

Let $\ts{G'}_{E,y}$ be the reductive quotient of the special fiber of the parahoric of $G' \times E$ associated to $y$. As argued in Section \ref{sec:parahorics}, we then have
\[ -\log_q\vol(G'(F)_{y,0+};\nu_{G'}) = \frac{1}{2}\dim(\ts{G'}_{E,y}^I). \]
Using the root-decomposition of the Lie-algebra of $\ts{G'}_{E,y}$ with respect to the adjoint action of the reduction of $S$, and the fact that $I$ acts elliptically on $S$ and regularly on $R(S,G)$, we see that $\dim(\ts{G'}_{E,y}^I)=\dim(\ts{Z})+\#R/e$.

On the other hand, we have $-\log_q\vol(A(F)_{0+},\nu_A) = \frac{1}{2}\dim(A)$. Recalling that $M$ is the cocharacter module of $\ts{Z}/\ts{A}$, we obtain the result.
\end{proof}

\begin{lem} \label{lem:vol2} We have
\[ \left|\frac{S'(F)}{S'(F)_{0+}A(F)}\right| = |X_*(S/A)_I^\tx{Fr}|\cdot|(M\otimes\ol{k_F}^\times)^\tx{Fr}|.\]
\end{lem}
\begin{proof}
The torus $S'$ being tamely ramified, the group $S'(F)_{0+}$ surjects onto $[S'/A](F)_{0+}$, and we see that
\[ \left|\frac{S'(F)}{S'(F)_{0+}A(F)}\right| = \left|\frac{[S'/A](F)}{[S'/A](F)_{0+}}\right|. \]
It is known \cite{HR08} that
\[ [[S'/A](F):[S'/A](F)_0] = |X_*(S/A)_I^\tx{Fr}|, \]
so it remains to compute $[[S'/A](F)_0:[S'/A](F)_{0+}]$. This is the number of $k_F$-points in the reductive quotient of the special fiber of the parahoric group scheme of $S'/A$. This is the connected component of the $I$-fixed points in the reductive quotient of the special fiber of the parahoric group scheme of $[S'/A] \times E$, and its character module is $M$. We conclude
\[ [[S'/A](F)_0:[S'/A](F)_{0+}] = |(M \otimes \ol{k_F}^\times)^\tx{Fr}| \]
and the statement follows.
\end{proof}
The last two lemmas allow us to conclude
\[ \tx{deg}(\pi;\mu_{G'/A,\psi}) = \frac{q^{\frac{1}{2}(\dim(G'/A)+\rk(M)+\#R/e)}}{|X_*(S/A)_I^\tx{Fr}|\cdot|(M\otimes\ol{k_F}^\times)^\tx{Fr}|}. \]

Recall that $M=X_*(Z/A)^I$, and hence
\[ \tx{rk}(M)=\dim(Z^I)-\dim(A)=\dim(\mf{\hat z}^I)-\dim(\mf{\hat z}^\Gamma). \]
Thus, comparing the formula for $\tx{deg}(\pi;\mu_{G'/A,\psi})$ above with the one in Corollary \ref{cor:formalrhs} and noticing that $X_*(S/A)^I=M$ by the ellipticity of the $I$-action on $S$, we see that the proof of Theorem \ref{thm:formal} will be complete once we establish the following elementary algebraic lemma.

\begin{lem} Let $\Gamma$ be an extension of the finite cyclic group $B$ by the finite cyclic group $A$, and let $M$ be a finite-rank free $\Z$-module with $\Gamma$-action such that $M^\Gamma=0$. Then
\[ |M_\Gamma| = |M_A^B| \cdot |M^A_B|. \]
\end{lem}
\begin{proof}
We have the exact sequence
\[ 0 \rw M_{A,\tx{tor}} \rw M_A \rw M^A \rw M^A/\tx{tr}_A M \rw 0, \]
where the middle map is given by the trace of the $A$-action. If we let $M_{A,\tx{free}}$ be the maximal torsion-free quotient of $M_A$, we see that the $B$-module $M_{A,\tx{free}}$ is a submodule of the $B$-module $M^A$. Our assumption implies that $H^0(B,M^A)=0$, and hence
\[ H^{2n}_\tx{Tate}(B,M^A) = 0,\qquad \forall n \in \Z. \]
By what we just argued, the same holds also with $M^A$ replaced by $M_{A,\tx{free}}$. Consider the exact sequence
\[ 0 \rw M_{A,\tx{tor}} \rw M_A \rw M_{A,\tx{free}} \rw 0. \]
On the one hand, the above cohomology vanishing statement implies that $H_1(B,M_{A,\tx{free}})=0$, and hence
\begin{equation} \label{eq:el1}
|M_\Gamma| = |M_{A,\tx{free},B}|\cdot|M_{A,\tx{tor},B}|.
\end{equation}
On the other hand, from $H^0(B,M_{A,\tx{free}})=0$ we obtain
\begin{equation} \label{eq:el2}
|M_A^B| = |(M_{A,\tx{tor}})^B|=|M_{A,\tx{tor},B}|.
\end{equation}
Next, we consider the exact sequence
\[ 0 \rw M_{A,\tx{free}} \rw M^A \rw M^A/\tx{tr}(M) \rw 0 \]
and obtain from $H_1(B,M^I)=0$ the exact sequence
\[ 0 \rw H_1(B,M^A/\tx{tr}_A(M)) \rw M_{A,\tx{free},B} \rw M^A_B \rw (M^A/\tx{tr}_A(M))_B \rw 0. \]
We claim that the outer terms in that sequence have the same cardinality. To save notation, let $X=M^A/\tx{tr}_A(M)$. Then
\[ |H_1(B,X)| = |H^0_\tx{Tate}(B,X)| = X^B/\tx{tr}_B(X). \]
The map $\tx{tr}_B : X \rw X$ is induced from the map $\tx{tr}_B : M^A \rw M^A$, but the latter map takes image in $M^\Gamma=0$. Thus
\[ |H_1(B,X)| = |X^B| = |X_B| \]
as claimed, and we conclude $|M_{A,\tx{free},B}| = |M^A_B|$. Combining this with \eqref{eq:el1} and \eqref{eq:el2} finishes the proof of the lemma, and with it also the proof of the theorem.
\end{proof}

\end{proof}

\subsection{Central and cocentral characters} \label{sec:centchar}

In this Section we want to study how the $L$-packets we have constructed behave with respect to central and cocentral characters. We begin by
giving a reformulation of two constructions, originally due to Langlands \cite[\S10]{Bo77}. The first construction assigns to any Langlands parameter $\phi : W_F \rw {^LG}$ a character $\chi_{\phi,Z} : Z(F) \rw \C^\times$. The second construction assigns to any Langlands parameter $\phi_z : W_F \rw{^LG}$ which factors through $\hat Z \rtimes W_F$ a character $\chi_{\phi_z,G} : G(F) \rw \C^\times$ which is trivial on the image of $G_\tx{sc}(F)$ in $G(F)$ (we will call such characters \emph{cocentral}). Here $Z$ is the center of $G$, and $\hat Z$ is the center of $\hat G$. These constructions are related to the conjectural local Langlands correspondence as follows:
\begin{itemize}
\item If $\phi : W_F \rw {^LG}$ is a Langlands parameter with corresponding packet $\Pi_\phi$, then the central character of each constituent of $\Pi_\phi$ is equal to $\chi_{\phi,Z}$.
\item Moreover, for any $\phi_z : W_F \rw \hat Z \rtimes W_F$, one has $\Pi_{\phi_z \otimes \phi} = \chi_{\phi_z,G} \otimes \Pi_\phi$.
\end{itemize}

The constructions given in \cite[\S10]{Bo77}, which follow Langlands' original exposition in \cite{Lan88}, proceed by replacing the group $G$ by certain auxiliary groups. We are going to give here a reinterpretation of these constructions that avoids the use of these auxiliary groups. This considerations are quite general, and are valid for any connected reductive group $G$ defined over any local field $F$ of characteristic zero. We will adopt this level of generality while giving them. Afterwards, we will consider again the case of tamely ramified $p$-adic groups  and show that the above conjectural properties hold for our $L$-packets.

For now, let $F$ by a local field of characteristic zero, and let $G$ by any connected reductive group defined over $F$. The key to our approach lies in the cohomological pairings for complexes of tori of length 2 constructed by Kottwitz and Shelstad in \cite[\S A]{KS99}, as well as the systematic use of the cohomology of crossed modules. We refer the reader to the paper \cite{No11} for an exposition on the cohomology of crossed modules, with the warning that our notation is shifted by $1$ from the notation used there, so that $H^0$ in our notation corresponds to $H^{-1}$ in the notation of \cite{No11}. This shift ensures that if $T \rw U$ is a map of tori, then its cohomology as a crossed module agrees with its hypercohomology when viewed as a complex placed in degree $0$ and $1$, and the latter situation is the one studied in \cite[\S A]{KS99}.

The crossed modules that we will be using are $[G_\tx{sc} \rw G]$ and $[\hat G_\tx{sc} \rw \hat G]$. These crossed modules are endowed with symmetric braidings: Let $g,h \in G$ and choose $g_\tx{sc},h_\tx{sc} \in G_\tx{sc}$ whose images in $G_\tx{ad}$ agree with those of $g$ resp. $h$. Then the symmetric braiding on $[G_\tx{sc} \rw G]$ is defined by $\{g,h\} := g_\tx{sc}^{-1}h_\tx{sc}^{-1}g_\tx{sc}h_\tx{sc}$. The braiding on $[\hat G_\tx{sc} \rw \hat G]$ is defined in the same way. The existence of these symmetric braidings ensures that $H^1(F,G_\tx{sc} \rw G)$ and $H^2(W_F,\hat G_\tx{sc} \rw \hat G)$ are abelian groups.

\begin{pro} \label{pro:dualities} There are canonical isomorphisms
\[ \tx{Hom}_\tx{cts}(Z(F),\C^\times) \cong H^2(W_F,\hat G_\tx{sc} \rw \hat G) \]
and
\[ \tx{Hom}_\tx{cts}(H^1(F,G_\tx{sc} \rw G),\C^\times) \cong H^1(W_F,\hat Z). \]
\end{pro}
\begin{proof}
Let $\hat T$ be a maximal torus of $\hat G$ that is part of a $\Gamma$-invariant splitting. The inclusions
\[ [\hat T_\tx{sc} \rw \hat T] \lw [\hat Z_\tx{sc} \rw \hat Z] \rw [\hat G_\tx{sc} \rw \hat G] \]
induce isomorphisms on $W_F$-cohomology. According to \cite[Lemma A.3.A]{KS99}, $H^2(W_F,\hat T_\tx{sc} \rw \hat T)$ is the group of characters of $H^0(F,T\rw T_\tx{ad})$, where $T$ is the minimal Levi in the quasi-split inner form of $G$. This establishes the first isomorphism, since the map $[Z \rw 1] \rw [T \rw T_\tx{ad}]$ provides again an isomorphism on $\Gamma$-cohomology.

The second isomorphism is proved in an analogous fashion. Namely, we have
\[ H^1(W_F,\hat Z) = H^1(W_F,\hat T \rw \hat T_\tx{ad}). \]
Let us consider the natural map from $\hat T_\tx{ad}^\Gamma$ to the above group. It is known \cite{Kot84} that this map factors through $\pi_0(\hat T_\tx{ad}^\Gamma)$. Thus, according to \cite[Lemma A.3.B]{KS99}, we have a canonical isomorphism
\[ H^1(W_F,\hat T \rw \hat T_\tx{ad}) = \tx{Hom}_\tx{cts}(H^1(F,T_\tx{sc} \rw T),\C^\times). \]
The second isomorphism in the statement of the proposition now follows from the fact that
\[ [T_\tx{sc} \rw T] \lw [Z_\tx{sc} \rw Z] \rw [G_\tx{sc} \rw G] \]
induce isomorphisms on $\Gamma$-cohomology.
\end{proof}

Proposition \ref{pro:dualities} relates to Langlands' constructions as follows: The exact sequence of crossed modules
\[ 1 \rw [1 \rw \hat G] \rw [\hat G_\tx{sc} \rw \hat G] \rw [\hat G_\tx{sc} \rw 1] \rw 1 \]
induces a map
\[ H^1(W_F,\hat G) = H^2(W_F,1 \rw \hat G) \rw H^2(W_F,\hat G_\tx{sc} \rw \hat G).  \]
On the other hand, the dual exact sequence
\[ 1 \rw [1 \rw G] \rw [G_\tx{sc} \rw G] \rw [G_\tx{sc} \rw 1] \rw 1 \]
induces a long exact sequence on cohomology
\[ 1 \rw \ker(G_\tx{sc} \rw G)(F) \rw G_\tx{sc}(F) \rw G(F) \rw H^1(F,G_\tx{sc} \rw G) \rw H^1(F,G_\tx{sc}). \]
Thus we obtain an injection $G(F)/G_\tx{sc}(F) \rw H^1(F,G_\tx{sc} \rw G)$ which in the $p$-adic case is also bijective due to Kneser's theorem on the vanishing of $H^1(F,G_\tx{sc})$.
\begin{pro}
The maps
\[ H^1(W_F,\hat G) \rw H^2(W_F,\hat G_\tx{sc} \rw \hat G) \cong \tx{Hom}_\tx{cts}(Z(F),\C^\times) \]
and
\[ H^1(W_F,\hat Z) \cong \tx{Hom}_\tx{cts}(H^1(F,G_\tx{sc}\rw G),\C^\times) \rw \tx{Hom}_\tx{cts}(G(F)/G_\tx{sc}(F),\C^\times) \]
coincide with Langlands' constructions.
\end{pro}

\begin{proof}
Since the proofs of the two statements are quite similar, we will only sketch the second one. First recall the construction of $\chi_{\phi_z,G}$, following the exposition in \cite{Bo77}. Choose a $z$-extension
\[ 1 \rw D \rw \tilde G \rw G \rw 1. \]
The reader is referred to \cite[\S1]{Kot82} for a review of $z$-extensions. Recall in particular that $D$ is an induced torus and $\tilde G$ has a simply-connected derived group. On the dual side we obtain the exact sequence
\[ 1 \rw \hat G \rw \hat{\tilde G} \rw \hat D \rw 1 \]
and the center of $\hat{\tilde G}$ is connected and dual to the torus $\tilde G/G_\tx{sc}$. Composing $\phi_z$ with the inclusion $Z(\hat G) \rw Z(\hat{\tilde G})$ we obtain an element of $H^1(W_F,Z(\hat{\tilde G}))$, hence a character on $[\tilde G/G_\tx{sc}](F)$, which we can pull back to a character on $\tilde G(F)$. One checks easily that this character restricts trivially to $D(F)$, and hence provides a character on $\tilde G(F)/D(F) = G(F)$, where the last equality holds due to the cohomological triviality of $D$.

The claim now follows from the commutativity of the diagram
\[ \xymatrix{
1\ar[d]&1\ar[d]\\
H^1(W_F,Z(\hat G))\ar[r]\ar[d]&\tx{Hom}_\tx{cts}(G(F),\C^\times)\ar[d]\\
H^1(W_F,Z(\hat{\tilde G}))\ar[r]\ar[d]&\tx{Hom}_\tx{cts}(\tilde G(F),\C^\times)\ar[d]\\
H^1(W_F,\hat D)\ar[r]&\tx{Hom}_\tx{cts}(D(F),\C^\times)
} \]
and the exactness of both vertical sequences, as well as the fact that when $G$ has a simply-connected derived group, so that we can take $\tilde G=G$, Langlands' original construction coincides with the one given here.
\end{proof}

We now assume that $F$ is $p$-adic, $G$ is tamely-ramified, and $\phi$ satisfies Conditions \ref{cnd:parm}.

\begin{pro} Let $(S,\chi_{S,X})$ be the pair obtained from $\phi$ using any choice of $\chi$-data $X$. Then
\begin{enumerate}
\item $\chi_{S,X}|_{Z(F)}=\chi_{\phi,Z}$.
\item If $\phi_z \in Z^1(W_F,\hat Z)$ and $\chi'_{S,X}$ is the character on $S(F)$ obtained from $\phi_z \otimes \phi$ and the same $\chi$-data $X$, then $\chi'_{S,X} = \chi_{\phi_z,G}|_{S(F)} \otimes \chi_{S,X}$.
\end{enumerate}
\end{pro}

Before we go into the proof, we want to remark that the above statements make sense. For the first one, the stable class of embeddings $S \rw G$ provides a canonical embedding $Z \rw S$, which is what is used to define the restriction to $Z(F)$. For the second statement, we use again the embeddings $S \rw G$ to restrict $\chi_{\phi_z,G}$ to $S(F)$. This restriction is again independent of the particular embedding.

\begin{proof}
During this proof, we will identify homomorphisms $W_F \rw {^LG}$ with 1-cocycles $W_F \rw \hat G$ using the semi-direct product structure on $^LG$.
For the first claim, we must show that the images of $\phi \in H^1(W_F,\hat G)=H^2(W_F,1 \rw \hat G)$ and $\phi_{S,X} \in H^1(W_F,\hat S)=H^2(W_F,1 \rw \hat S)$ in $H^2(W_F,Z(\hat G_\tx{sc}) \rw Z(\hat G))$ under the following maps coincide:
\[ \xymatrix{
H^2(W_F,1 \rw \hat S)\ar[r]&H^2(W_F,\hat S_\tx{sc} \rw \hat S)\ar[d]^\cong\\
&H^2(W_F,Z(\hat G_\tx{sc}) \rw Z(\hat G))\\
H^2(W_F,1 \rw \hat G)\ar[r]&H^2(W_F, \hat G_\tx{sc} \rw \hat G)\ar[u]^\cong
} \]
To that end, recall that $\phi(w)=\phi_{S,X}(w)\cdot\xi_X(w)$, where $\xi_X : W_F \rw N(\hat T)$ is the 1-cocycle with $^Lj_X(s,w)=(s\xi_X(w),w)$. The key point that makes everything work is the fact that $\xi_X$ lifts to a 1-cocycle $\xi_X^\tx{sc} : W_F \rw N(\hat T_\tx{sc})$. This follows from the construction of Langlands-Shelstad. We can find continuous 1-cochains $c_1 : W_F \rw \hat S_\tx{sc}$ and $c_2 : W_F \rw Z(\hat G)$ such that $\phi_{S,X}(w)=\bar c_1(w) \cdot c_2(w)$.
We have $\partial c_1 \in Z^2(W_F,Z(\hat G_\tx{sc}))$, and the elements $(\partial c_1^{-1},c_2)$ and $(1,\phi_{S,X})$ of $H^1(W_F,\hat S_\tx{sc} \rw \hat S)$ are equal. On the other hand, the element $(\partial c_1^{-1},c_2)$ of $H^2(W_F,Z(\hat G_\tx{sc}) \rw Z(\hat G))$, when mapped to $H^2(W_F,Z\hat G_\tx{sc} \hat G)$, is equal to the element $(\partial c_1^{-1}\cdot\partial(c_1 \cdot \xi_X^\tx{sc}),c_2 \cdot \bar c_1 \cdot \xi_X)$, which, by virtue of $\partial\xi_X^\tx{sc}=1$, equals $(1,\phi_{S,X}\cdot\xi_X)$, which is $(1,\phi)$.

The second claim follows from the fact that the map $S(F) \rw G(F)/G_\tx{sc}(F)$ is dual to the map $H^1(W_F,Z(\hat G)) \rw H^1(W_F,\hat S)$, as well as the fact that $\phi'_{S,X}=\phi_{z,G} \otimes \phi_{S,X}$.

\end{proof}

\section{Stability and endoscopy}

\subsection{Stable and $s$-stable characters of epipelagic $L$-packets} \label{sec:schar}

Let $\phi : W_F \rw {^LG}$ be an epipelagic parameter. For any inner twist $\xi : G \rw G'$, we constructed in Section \ref{sec:lpackconst} the $L$-packet $\Pi_{\phi,\xi,G'}$. Moreover, given a Whittaker datum $(B,\psi_B)$ for $G$, we have constructed in Section \ref{sec:lpackparm} the compound packet $\Pi_\phi$, which consists of equivalence classes of quadruples $(G^b,\xi_b,b,\pi)$, and a bijection
\[ X^*(S_\phi) \rw \Pi_\phi,\qquad\rho \mapsto (G^{b_\rho},\xi_\rho,b_\rho,\pi_\rho). \]
The map $\Pi_\phi \rw B(G)_\tx{bas}$ which sends a quadruple $(G^b,\xi_b,b,\pi)$ to the equivalence class of $b$ is surjective and its fibers are the sets $\Pi_{\phi,\xi_b,G^b}$. Not all of the sets $\Pi_{\phi,\xi,G'}$ however occur as fibers of this map.

For every $\xi : G \rw G'$, we can consider the stable character of the $L$-packet $\Pi_{\phi,\xi,G'}$, which is the function on $G'(F)$ given by
\[ S\Theta_{\phi,\xi} = e(G')\sum_{\pi \in \Pi_{\phi,\xi,G'}} \Theta_\pi. \]
The quantity $e(G')$ is the Kottwitz sign \cite{Kot83} of $G'$.

Moreover, for a triple $(G^b,\xi_b,b)$ and an element $s \in S_\phi$, we can consider the $s$-stable character of the $L$-packet for $(G^b,\xi_b,b)$, which is the function on $G^b(F)$ given by
\[ \Theta_{\phi,b}^s = e(G^b)\sum_{\substack{\rho \in X^*(S_\phi)\\ \rho \mapsto b}} \rho(s) \Theta_{\pi_\rho}. \]
The sum runs over those characters $\rho \in X^*(S_\phi)$ whose restriction to $X^*(Z(\hat G)^\Gamma)$ corresponds to $b \in B(G)_\tx{bas}$ under the Kottwitz isomorphism \cite[Prop 5.6]{Kot85}. We have
\[ S\Theta_{\phi,\xi_b} = \Theta_{\phi,b}^1. \]
We will now provide formulas for the functions $S\Theta_{\phi,\xi}$ and $\Theta^s_{\phi,b}$ using the work of Adler and Spice \cite{AS10} on character values for supercuspidal representations. In order to be able to derive our formulas, we need to assume that the residual characteristic of $F$ is large enough. A convenient lower bound is given by $p \geq (2+e)n$, where $e$ is the ramification degree of $F/\Q_p$, and $n$ is the smallest dimension of a faithful rational representation of $G$. Under this assumption, DeBacker and Reeder have shown in the appendices to \cite{DR09} that the following statements hold.
\begin{enumerate}
\item For every inner form $G'$ of $G$, the exponential map $\exp : \mf{g'}(F^u)_{0+} \rw G'(F^u)_{0+}$ is defined and provides a bijection which is equivariant for the adjoint action of $G'(F^u)$ as well as for the action of Frobenius.
\item There exists a bilinear form $\<\>$ on $\mf{g'}(\ol{F})$ which is invariant with respect to the adjoint action of $G'(\ol{F})$ as well as the Galois action, and such that for any maximal torus $T \subset G$, the element $\<H_\alpha,H_\alpha\>$ is a unit in $\ol{F}$ for all coroots $H_\alpha$ of $T$.
\end{enumerate}
We will denote by $\log : G'(F^u)_{0+} \rw \mf{g'}(F^u)_{0+}$ the inverse of $\exp$. Fix a character $\psi : F \rw \C^\times$ which is trivial on $\mf{p}_F$ and non-trivial on $O_F$. Let $(S,\chi)$ be a pair of a maximal torus of $G'$ and a character of $S(F)$ satisfying Conditions \ref{cnd:char}. Then the main result of \cite{AS10} provides the following formula for the character of $\pi_{S,\chi}$.
\begin{equation} \Theta_{\pi_{S,\chi}}(\gamma) = \frac{D_{J}(\gamma_{>0})}{D_{G'}(\gamma)}\sum_{\substack{g \in J(F) \lmod G'(F)/S(F)\\ g^{-1}\gamma_0g \in S(F)}}\chi_g(\gamma_0)D_{J}(Y_g) \hat\mu^{J}_{Y_g}(\log(\gamma_{>0})). \label{eq:1char} \end{equation}
We need to explain the notation. First $\gamma=\gamma_0 \cdot \gamma_{>0}$ is a topological Jordan decomposition modulo $Z(G)^\circ$ of $\gamma$ \cite{Sp08}. Not every element $\gamma$ has such a decomposition, but it is part of the character formula that if $\gamma$ does not have this decomposition, then $\Theta_{\pi_{S,\chi}}(\gamma)=0$. Moreover, since $S$ is tamely ramified, so is $Z(G)^\circ$, and thus the map $G'(F)_{\tx{rs},0+} \rw [G'/Z(G)^\circ](F)_{\tx{rs},0+}$ is surjective, so we can always arrange that $\gamma_{>0} \in G'(F)_{0+}$, which we will assume.

Next, the group $J$ is the connected centralizer of $\gamma_0$ in $G'$. We let $Y \in \mf{s}(F)$ be a generic element for which
\[ \chi_S(x) = \psi\<Y,x-1\> \]
for all $x \in S(F)_{0+}$. Then $Y_g = \tx{Ad}(g)Y$, and $\chi_g = \tx{Ad}(g)_*\chi$. The quantities $D()$ denote the Weyl-discriminants in the corresponding group or Lie-algebras. For example, $D_{G'}(\gamma)$ denotes the square root of the product of $|\alpha(\gamma)-1|$ where $\alpha$ runs over all roots of the centralizer of $\gamma$ in $G'$, while $D_{J}(Y_g)$ denotes the square root of the product of $|d\alpha(Y_g)|$ where $\alpha$ runs over all roots of the centralizer of $Y_g$ in $J$. Finally, $\hat\mu_{Y_g}^{J}$ denotes the Fourier-transform of the orbital integral at $Y_g$ in the Lie-algebra of $J$. The orbital integral is taken with respect to the measure on $J(F)/\tx{Cent}(Y_g,J)(F)$ which is the quotient of the measures on the two groups normalized as in \cite[\S5.1]{DR09}. The Fourier-transform is taken using the bi-character on $\mf{g'}(F)$ given by $\psi\<\>$.

Before we state our formula, we recall a different normalization of the function $\hat\mu$, which is used by Waldspurger in \cite{Wal97}. It is the function
\[ \hat\iota^J_Y(X) = D_J(Y)D_J(X)\hat\mu^J_Y(X). \]

\begin{pro} \label{pro:schar} We have that $S\Theta_{\phi,\xi}(\gamma)$ is equal to
\[ e(G')|D_{G'}(\gamma)|^{-1}\sum_{[j]}j_*[\chi_{S,{^Lj}}\cdot\epsilon_j](\gamma_0)\sum_{[k]} \hat\iota^{J}({kY},\log(\gamma_{>0})), \]
while $\Theta^s_{\phi,b}(\gamma)$ is equal to
\[ e(G^b)|D_{G^b}(\gamma)|^{-1}\sum_{[j]}j_*[\chi_{S,{^Lj}}\cdot\epsilon_j](\gamma_0)\sum_{[k]} \<\tx{inv}(j_0,k),s\> \hat\iota^{J}(kY,\log(\gamma_{>0})). \]
In both sums, $[j]$ runs over the $J$-stable conjugacy classes of embeddings $S \rw J$ whose composition with the inclusion $J \rw G'$ (resp. $J \rw G^b$) is admissible, and $[k]$ runs over the set of $J(F)$-conjugacy classes inside $[j]$.
\end{pro}
\begin{proof}
We will derive both formulas simultaneously. Consider for a moment $\Theta^s_{\phi,b}$. Its definition involves a sum over the set $\{ \rho \in X^*(S_\phi)| \rho \mapsto b\}$ and the map $\rho \mapsto \pi_\rho$ constructed in Section \ref{sec:lpackparm}. The latter map was constructed from a choice of a Whittaker datum $(B,\psi_B)$, and a choice of an $L$-embedding ${^Lj} : {^LS} \rw {^LG}$ within the $\hat G$-conjugacy class specified in Section \ref{sec:chispec}. We make these choices. As argued in Section \ref{sec:lpackparm}, this determines an isomorphism $X^*(S_\phi) \rw B(S)$, as well as an admissible embedding $j_0 : S \rw G$ for which $\pi_{j_0}$ is $(B,\psi_B)$-generic. Under the isomorphism $X^*(S_\phi) \rw B(S)$,
the sum over the set $\{ \rho \in X^*(S_\phi)| \rho \mapsto b\}$ is translated  to a sum over $\{ \lambda \in B(S)| \lambda \mapsto b\}$ where $\lambda \mapsto b$ is taken under the map $B(j_0) : B(S) \rw B(G)_\tx{bas}$. We have a bijection from the set of admissible embeddings of $S$ into $G^b$ to the set $\{ \lambda \in B(S)| \lambda \mapsto b\}$, and this bijection sends an embedding $j : S \rw G^b$ to the element $\tx{inv}(j_0,j)$. Moreover, if $\rho \mapsto j$ under the isomorphism $X^*(S_\phi) \rw B(S)$, then
\[ \rho(s) = \<\tx{inv}(j_0,j),s\>. \]
We conclude that
\begin{equation} \Theta^s_{\phi,b} = e(G^b)\sum_{[j]} \<\tx{inv}(j_0,j),s\> \Theta_{\pi_j}, \label{eq:schar1} \end{equation}
where the sum runs over the set of $G^b(F)$-conjugacy classes of admissible embeddings $j : S \rw G^b$.

Now consider $S\Theta_{\phi,\xi,G'}$. By construction, the set $\Pi_{\phi,\xi,G'}$ is in bijection with the set of $G'(F)$-conjugacy classes of admissible embeddings $j : S \rw G'$, with $j$ mapping to $\pi_j$. Thus
\[ S\Theta_{\phi,\xi,G'} = e(G')\Theta_{\pi_j}. \]
We see that the formulas for $\Theta^s_{\phi,b}$ and $S\Theta_{\phi,\xi,G'}$ are the same apart from the factor $\<\tx{inv}(j_0,j),s\>$. Because of this, it will be enough to treat the case of $\Theta^s_{\phi,b}$, the argument for $S\Theta_{\phi,\xi,G'}$ being the same except that the factor $\<\tx{inv}(j_0,j),s\>$ can be set to $1$ in that case.

Inserting the Adler-Spice formula \eqref{eq:1char} into \eqref{eq:schar1}, we see that if $\gamma \in G^b(F)_\tx{rs}$ has the topological Jordan decomposition mod $Z(G)^\circ$ given by $\gamma=\gamma_0 \cdot \gamma_{>0}$ with $\gamma_{>0} \in G^b(F)_{0+}$, then $\Theta^s_{\phi,b}(\gamma)$ equals
\begin{equation*}
e(G^b) \frac{D_{J}(\gamma_{>0})}{D_{G^b}(\gamma)}
\sum_{[j]} \<\tx{inv}(j_0,j),s\>\ssum{\substack{g \in J(F) \lmod G^b(F)/S(F)\\ g^{-1}\gamma_0g \in jS(F)}}\chi^j_g(\gamma_0)D_{J'}(Y^j_g) \hat\mu^{J}_{Y^j_g}(\log(\gamma_{>0})).
\label{eq:schar2}
\end{equation*}
Here $\chi^j=j_*(\chi_{S,{^Lj}}\cdot\epsilon_j)$, and $Y^j=jY$. As $[j]$ traverses the first summation set and $g$ traverses the second, the composition $\tx{Ad}(g)j$ traverses the set of $J(F)$-conjugacy classes of embeddings $k : S \rw J$ which are $G^b$-stably conjugate to $j_0$. This set can be traversed by first summing over the set $[j]$ of $J$-stable classes of embeddings $S \rw J$ which are $G^b$-stably conjugate to $j_0$, and then summing over the set $[k]$ of $J(F)$-conjugacy classes of embeddings $S \rw J$ inside each $[j]$. With this, we obtain that $\Theta^s_{\phi,b}(\gamma)$ equals
\begin{equation*}
e(G^b) \frac{D_{J}(\gamma_{>0})}{D_{G^b}(\gamma)}
\sum_{[j]} \sum_{[k]} \<\tx{inv}(j_0,k),s\> \chi^k(\gamma_0)D_{J}(Y^k) \hat\mu^{J}_{Y^k}(\log(\gamma_{>0})).
\label{eq:schar3}
\end{equation*}
To complete the proof, we observe that the term $\chi^k(\gamma_0)$ is constant in $k$ and depends only on $j$.
\end{proof}

\subsection{Stability and transfer to inner forms} \label{sec:stab}

\begin{thm} \label{thm:stab} Let $\xi : G \rw G'$ be an inner twist, and let $\gamma \in G(F)_\tx{rs}$ and $\gamma' \in G(F)_\tx{rs}$ be related elements. Then
\[ S\Theta_{\phi,1,G}(\gamma) = S\Theta_{\phi,\xi,G'}(\gamma'). \]
\end{thm}
\begin{proof}
Let $T \subset G$ and $T' \subset G'$ be the centralizers of $\gamma$ and $\gamma'$. By assumption, we may modify $\xi$ within its equivalence class so that $\xi(\gamma)=\gamma'$. Then $\xi : T \rw T'$ is an isomorphism defined over $F$. The element $\gamma$ has a topological Jordan decomposition modulo $Z(G)^\circ$ if and only if $\gamma'$ does, and we may assume this is the case, for otherwise both sides of the claimed equality are zero. Let us then choose such a decomposition $\gamma=\gamma_0 \cdot \gamma_{>0}$ with $\gamma_0 \in T(F)$ and $\gamma_{>0} \in T(F)_{0+}$ and set $\gamma'_0 = \xi(\gamma_0)$ and $\gamma'_{>0}=\xi(\gamma_{>0})$. Then $\gamma'=\gamma'_0\cdot \gamma'_{>0}$ is such a decomposition of $\gamma'$. Letting $J$ and $J'$ be the connected centralizers of $\gamma_0$ and $\gamma'_0$, the map $\xi$ descends to an inner twist $J \rw J'$. This inner twist provides a bijection between the stable classes of elliptic maximal tori in $J$ and those in $J'$, and this bijection restricts to a bijection $j \leftrightarrow j'$ between the set of $J$-stable classes of embeddings $j : S \rw J$ which are $G$-admissible when composed with the inclusion $J \rw G$, and the corresponding set for $J'$. Moreover, if $j \leftrightarrow j'$, then $j^{-1}(\gamma_0)=j'^{-1}(\gamma'_0)$.

Applying Proposition \ref{pro:schar} to both $G$ and $G'$, we see that we must show for each pair of embeddings $j : S \rw J$ and $j' : S \rw J'$ with $j \leftrightarrow j'$ that
\[ e(G'){j'}_*[\chi_{S,{^Lj}}\cdot\epsilon_{j'}](\gamma'_0)\sum_{[k']} \hat\iota^{J'}({k'Y},\log(\gamma'_{>0})) \]
is equal to
\[ e(G){j}_*[\chi_{S,{^Lj}}\cdot\epsilon_{j}](\gamma_0)\sum_{[k]} \hat\iota^{J}({kY},\log(\gamma_{>0})) \]
The fundamental results of Waldspurger \cite{Wal97}, \cite{Wal06}, \cite{Ngo10} reduce our task to showing the equation
\[ e(G')\gamma_\psi(\mf{j'}(F),\<\>){j'}_*[\chi_{S,{^Lj}}\cdot\epsilon_{j'}](\gamma'_0) = e(G)\gamma_\psi(\mf{j}(F),\<\>){j}_*[\chi_{S,{^Lj}}\cdot\epsilon_{j}](\gamma_0). \]
Applying Lemma \ref{lem:torinvweil} we see that
\[ \frac{\gamma_\psi(\mf{j'}(F),\<\>)}{\gamma_\psi(\mf{j}(F),\<\>)}
=\prod_{\alpha \in R(J,G)_\tx{sym}/\Gamma} \frac{f_{(J',j'S)}(\alpha)}{f_{(J,jS)}(\alpha)} \]
On the other hand, Proposition \ref{pro:torinvstab} shows
\[ \frac{e(G')}{e(G)} = \prod_{\alpha \in R(S,G)_\tx{sym}/\Gamma} \frac{f_{(G',j'S)}(\alpha)}{f_{(G,jS)}(\alpha)} \]
Finally, Lemma \ref{lem:torinvchar} shows
\[ \frac{j'_*e_{j'}(\gamma'_0)}{j_*e_j(\gamma_0)} = \prod_{\alpha \in R(S,G)_\tx{sym} \sm R(S,J)_\tx{sym}/\Gamma} \frac{f_{(G',j'S)}(\alpha)}{f_{(G,jS)}(\alpha)}, \]
and the proof is complete.
\end{proof}

\subsection{Descent lemmas} \label{sec:desclem}

\begin{lem} \label{lem:conjdisc} Let $J$ be a disconnected linear algebraic group defined over $F$ whose connected component $J^\circ$ is reductive and quasi-split. Let $\gamma \in J^\circ(F)$ be an element such that $\tx{Cent}(\gamma,J) \subset J^\circ$. Then $\pi_0(J)(F)$ acts simply transitively on the set of those $J^\circ(\ol{F})$-conjugacy classes inside the $J(\ol{F})$-conjugacy class of $\gamma$ which are defined over $F$. Moreover, every such class has an $F$-point.
\end{lem}
\begin{proof}
The first statement is obvious; we turn to the second. Let $C=\tx{Ad}(J(\ol{F}))\gamma$. The action of $\pi_0(J)(\ol{F})$ on $C/\tx{Ad}(J^\circ(\ol{F}))$ factors through the map $\pi_0(J) \rw \tx{Out}(J)$. Fix an $F$-splitting of the quasi-split group $J^\circ$ provides a $\Gamma$-equivariant section $\tx{Out}(J^\circ) \rw \tx{Aut}(J^\circ)$ of the natural projection. Hence we obtain
\[ f : \pi_0(J)(F) \rw \tx{Out}(J^\circ)(F) \rw \tx{Aut}(J^\circ)(F). \]
For $x \in \pi_0(J)(F)$, the element $f(x)\gamma$ is an $F$-point in the $J^\circ$-class $\tx{Ad}(x\cdot J^\circ)\gamma$.
\end{proof}

We now let $J$ be a connected reductive group defined over $F$. We assume that for each $\gamma \in J(F)_\tx{rs}$, we have fixed a decomposition $\gamma = \gamma_0 \cdot \gamma_1$, with $\gamma_0,\gamma_1 \in T$, where $T=\tx{Cent}(\gamma,J)^\circ$. We moreover assume that for any admissible isomorphism $f : T \rw T'$ of maximal tori of $J$, the decomposition of $f(\gamma)$ is $f(\gamma_0)\cdot f(\gamma_1)$.

We are now going to establish two lemmas dealing with descent to the centralizer of $\gamma_0$, where $\gamma_0$ is the $0$-part in the decomposition $\gamma=\gamma_0\cdot\gamma_1$ of a strongly regular semi-simple element. The first lemma provides a comparison between stable conjugacy in $J$ and stable conjugacy in $J_{\gamma_0}$.

\begin{lem} \label{lem:decdesc}
Assume that $J$ is quasi-split. The map
\[ p : J(F)_\tx{sr}/\tx{st} \rw J(F)_\tx{ss}/\tx{st},\qquad \gamma \mapsto \gamma_0 \]
is well defined. Every stable class in the image of this map contains a point $y \in J(F)_\tx{ss}$ for which $J_y$ is quasi-split. Moreover,
\[ p_y : J_y(F)_1/\tx{st} \rw p^{-1}([y]),\qquad z \mapsto zy \]
is a surjection, where the subscript $1$ denotes all those elements $\gamma_1$ such that $y\gamma_1$ is strongly-regular semi-simple and has the decomposition $y\cdot\gamma_1$. Finally, the fibers of $p_y$ are torsors under $\pi_0(H^y)(F)$.
\end{lem}
\begin{proof}
For the first claim, we need to show that the map $J(F)_\tx{rs} \rw J(F)_\tx{ss}$ given by sending $\gamma$ to $\gamma_0$ respects stable conjugacy. If $\gamma$ and $\gamma'$ are stably-conjugate, say by $g \in J(\ol{F})$, then $\tx{Ad}(g) : T_\gamma \rw T_{\gamma'}$ is an isomorphism defined over $F$ between the centralizers of $\gamma$ and $\gamma'$. Being admissible, it carries $\gamma_0$ to $\gamma'_0$, and hence exhibits these two elements as stably conjugate.

Now let $y \in J(F)_\tx{ss}$ be an element whose stable class belongs to the image of $p$. According to \cite[Lemma 3.3]{Kot82}, we may choose $y$ within its stable class so that $J_y$ is quasi-split. To prove the surjectivity of $p_y$, let $\gamma \in J(F)$ be such that $\gamma_0$ is stably-conjugate to $y$, and let $g \in J(\ol{F})$ be an element which stably-conjugates $\gamma_0$ to $y$. Then $\tx{Ad}(g) : J_{\gamma_0} \rw J_y$ is an inner twist. The centralizer $T_\gamma$ is a maximal torus in $J_{\gamma_0}$ and contains $\gamma_1$. Since $J_y$ is quasi-split, there exists $h \in J_y$ such that $T' := \tx{Ad}(hg)T_\gamma$ is a maximal torus of $J_y$ defined over $F$, and $\tx{Ad}(hg) : T_\gamma \rw T'$ is defined over $F$. Then $\gamma' := \tx{Ad}(hg)\gamma$ belongs to $J_y(F)$ and has the decomposition $\gamma' = y \cdot \gamma'_1$, where $\gamma'_1 = \tx{Ad}(hg)\gamma_1$. We see that $p_y(\gamma'_1)$ is stably conjugate to $\gamma$, and this proves the surjectivity of $p_y$.

To prove the statement about the fibers of $p_y$, let $\gamma_1,\gamma'_1 \in J_y(F)_{\tx{sr},1}$ and assume that that $y\gamma_1$ and $y\gamma_1'$ are stably conjugate, say by $g \in J(\ol{F})$. Our assumptions on the decomposition imply then that $g \in J^y(\ol{F})$. We conclude that the fiber of $p_y$ through $\gamma_1$ is the set of $F$-points of $\tx{Ad}(J^y(\ol{F}))\gamma_1$. The statement now follows from Lemma \ref{lem:conjdisc}.
\end{proof}

The second descent lemma will deal with the notion of transfer of elements from an endoscopic group $H$ to $J$. If $\gamma_0 \in J(F)$ and $y \in H(F)$ are such that $y$ transfers to $\gamma_0$ and $H_y$ is quasi-split, then Langlands and Shelstad show \cite{LS90} that $H_y$ can be realized as an endoscopic group of $J_{\gamma_0}$. In order to do that, one needs to fix an isomorphism from a maximal torus in $H_y$ to a maximal torus in $J_{\gamma_0}$ which is admissible when considered as an isomorphism from a maximal torus of $H$ to a maximal torus of $J$. Conjugating this isomorphism by elements of $H_y$ or $J_{\gamma_0}$ has no effect on the way $H_y$ behaves as an endoscopic group of $J_{\gamma_0}$. We will write $\Xi(H_y,J_{\gamma_0})$ for the set of such isomorphisms of tori which send $y$ to $\gamma_0$, taken up to conjugation by elements of $H_y$ and $J_{\gamma_0}$.

\begin{lem} \label{lem:embdesc}
Let $H$ be an endoscopic group of $J$ and $S$ a torus defined over $F$. Let $j : S \rw J$ and $j^H : S \rw H$ be embeddings of $S$ as an elliptic maximal torus of $J$ and $H$, and let $\gamma \in J(F)_\tx{rs}$ be such that $\gamma_0 \in jS(F)$. Consider the set of triples
\[ \{ (y,\xi_y,j^H_y) \} \]
where $y$ runs over a fixed set of representatives $Y$ for the stable classes of preimages of $\gamma_0$, chosen so that $H_y$ is quasi-split, $\xi_y$ runs over $\Xi(H_y,J_{\gamma_0})$, and $j^H_y$ runs over the $H_y$-stable classes of embeddings $S \rw H_y$ whose composition with the natural inclusion $H_y \rw H$ is stably conjugate to $j^H$. Consider also the set
\[ \{ j_{\gamma_0} : S \rw J_{\gamma_0} \} \]
of $J_{\gamma_0}$-stable classes of embeddings whose composition with the inclusion $j_{\gamma_0} \rw J$ is stably conjugate to $j$. Then we claim that the map
\[ q : \{(y,\xi_y,j^H_y)\} \rw \{j_{\gamma_0}\} \]
given by transfer of stable classes of tori from $H_y$ to $J_{\gamma_0}$ is surjective and its fiber through a triple $(y,\xi_y,j^H_y)$ is a torsor under $\pi_0(H^y)(F)$.
\end{lem}
\begin{proof}
The map $q$ is well defined -- given a triple $(y,\xi_y,j^H_y)$, we may transfer the stable class of the maximal torus $j^H_y S$ of $H_y$ to the quasi-split form of $J_{\gamma_0}$. Since $S$ is elliptic for $J$, it is also elliptic for the quasi-split form of $J_{\gamma_0}$, and hence transfers further to $J_{\gamma_0}$ itself. Furthermore it is clear that $q$ is $\pi_0(H^y)(F)$-invariant.

For surjectivity, we start with $j_{\gamma_0}$. Then $y' = j^H\circ j_{\gamma_0}^{-1}(\gamma_0) \in H(F)$ is a preimage of $\gamma_0 \in J(F)$. There exists a unique element $y \in Y$ which is stably-conjugate to $y'$, say by $h \in H$. Then $\tx{Ad}(h) : H_{y'} \rw H_y$ is an inner twist and there exists $h' \in H_y$ such that $j^H_y := \tx{Ad}(h'h)\circ j^H : S \rw H_y$ is defined over $F$. Finally, take $\xi_y = j_{\gamma_0} \circ [j^H_y]^{-1}$.

For injectivity, let $(y,\xi_y,j^H_y)$ and $(y',\xi_{y'},j^H_{y'})$ be two triples giving rise to the same $J_{\gamma_0}$-stable class of embeddings $j_{\gamma_0} : S \rw J_{\gamma_0}$. Then we can choose $\xi_y$ and $\xi_{y'}$ within their equivalence classes in such a way that
\[ \xi_y \circ j^H_y = j_{\gamma_0} = \xi_{y'} \circ j^H_{y'}. \]
Setting $s=j_{\gamma_0}^{-1}(\gamma_0)$, this shows that $y=j_y^H(s)$ and $y'=j_{y'}^h(s)$. But the embeddings $j_y^H$ and $j_{y'}^H$ are $H$-stably conjugate, while $y,y' \in Y$. This forces $y=y'$. From this we see that $j_y^H$ and $j_{y'}^H$ are conjugate under $H^y(\ol{F})$, and the same element of $H^y$ conjugating these embeddings must then also conjugate $\xi_y$ to $\xi_{y'}$ in order to maintain the relationship with $j_{\gamma_0}$. The result now follows from Lemma \ref{lem:conjdisc}.

\end{proof}

\subsection{Endoscopic transfer} \label{sec:endotran}

We now fix an endoscopic datum $(H,s,\mc{H},{^L\eta})$ for $G$. To ease notation, we will assume that there exists an $L$-isomorphism $\mc{H} \rw {^LH}$. The general case can easily be reduced to this -- one has to replace $G$ by a $z$-extension as in \cite[\S4]{LS87}. Thus we may replace $\mc{H}$ by $^LH$ and assume that ${^L\eta}$ is an $L$-embedding ${^LH} \rw {^LG}$. We further assume that this $L$-embedding is tamely ramified. Fix a Whittaker datum $W=(B,\psi_B)$ for $G$. Then, for each extended pure inner twist $(\xi,b) : G \rw G^b$ we have \cite[\S2]{Kal11a} the Whittaker normalization of the transfer factor
\[ \Delta_{W,b} : H(F)_\tx{G-sr} \times G^b(F)_\tx{sr} \rw \C. \]
We emphasize that this is the normalization which, on the group $G$, coincides with the factor $\Delta'_\lambda$ defined in \cite{KS12}.
Let $\phi^H : W_F \rw {^LH}$ be a Langlands parameter and assume that its composition $\phi : W_F \rw {^LG}$ with ${^L\eta}$ is epipelagic. Then $\phi^H$ is itself epipelagic. Thus we have the $L$-packets $\Pi_{\phi^H,1,H}$ on $H(F)$ and $\Pi_{\phi,\xi,G^b}$ on $G^b(F)$. Our goal is to show that for all $\gamma^b \in G^b(F)_\tx{sr}$ the endoscopic character identity holds.
\begin{thm}\label{thm:endotran}
\[ \Theta^s_{\phi,b}(\gamma^b) = \sum_{\gamma^H \in H(F)_\tx{sr}/\tx{st}} \Delta_{W,b}(\gamma^H,\gamma^b)\frac{D^H(\gamma^h)^2}{D^{G^b}(\gamma^b)^2}S\Theta_{\phi^H}(\gamma^H). \]
\end{thm}
We will use the formulas from Proposition \ref{pro:schar}, and for this we need to fix topological Jordan decompositions mod $Z(G)^\circ$ for the elements involved. First, fix a topological Jordan decomposition $\gamma^b=\gamma^b_0\cdot\gamma^b_{>0}$ which is such that $\gamma^b_{>0} \in T_{\gamma^b}(F)_{0+}$. As remarked in Section \ref{sec:stab}, this is possible. We endow all stable conjugates of $\gamma^b$ with the induced decompositions. Moreover, for every stable class in $H$ which transfers to the stable class of $\gamma^b$, we fix a representative $\gamma^H$ and an admissible isomorphism $f : T_{\gamma^H} \rw T_{\gamma^b}$ and endow $\gamma^H$ with the decomposition of $\gamma^b$ transported over $f$. This fixes corresponding decompositions on all elements of $H(F)$ which transfer to $\gamma^b$.

Before we set out to prove this identity, we need to recall the results of Langlands and Shelstad \cite{LS90} on descent for transfer factors. For this, fix an element $\gamma^H \in H(F)$ which transfers to $\gamma^b$. We may modify $\gamma^H$ within its stable class to assume that $H_{\gamma^H_0}$ is quasi-split. Let $f : T_{\gamma^H} \rw T_{\gamma^b}$ be an admissible isomorphism sending $\gamma^H$ to $\gamma^b$. Then Langlands and Shelstad show how to produce from $(H,s,\mc{H},{^L\eta})$ an endoscopic datum $(H_{\gamma^H_0},s,\mc{H'},{^L\eta'})$ for $G^b_{\gamma^b_0}$. The isomorphism class of this endoscopic datum depends only on the class $\xi \in \Xi(H_{\gamma^H_0},G^b_{\gamma^b_0})$ of $f$. We cannot expect however that $\mc{H'}$ is an $L$-group of $H_{\gamma^H_0}$ and thus we have to take an extension $\tilde H_{\gamma^H_0}$ of $H_{\gamma^H_0}$ and an $L$-embedding ${^L\tilde\eta} : \mc{H} \rw {^L\tilde H_{\gamma^H_0}}$. We again assume that ${^L\tilde\eta}$ is tamely ramified. This data provides a canonical relative transfer factor
\[ \tilde H_{\gamma^H_0}(F)_\tx{sr} \times \tilde H_{\gamma^H_0}(F)_\tx{sr} \times G^b_{\gamma^b_0}(G)_\tx{sr} \times G^b_{\gamma^b_0}(G)_\tx{sr} \rw \C \]
and hence a family of absolute transfer factors
\[ \tilde H_{\gamma^H_0}(F)_\tx{sr} \times G^b_{\gamma^b_0}(G)_\tx{sr} \rw \C, \]
which is a torsor under $S^1$ acting by multiplication. The main theorem of \cite{LS90} then asserts that one of these transfer factors, which we will denote by $\Delta^\tx{desc}_{W,b,\xi}$, satisfies
\[ \Delta^\tx{desc}_{W,b,\xi}(\tilde \gamma^H_0 \cdot\tilde z,\gamma^b_0\cdot z^b) = \Delta_{W,b}(\gamma^H_0 \cdot z, \gamma^b_0\cdot z^b) \]
for all $\tilde z \in \tilde H_{\gamma^H_0}(F)_\tx{sr}$ and $z^b \in G^b_{\gamma^b_0}(G)_\tx{sr}$ which are sufficiently close to $\tilde\gamma^H_0$ and $\gamma^b_0$, where $\tilde\gamma^H_0 \in \tilde H_{\gamma^H_0}(F)$ is a fixed lift of $\gamma^H_0$. In the course of our proof, we will work with a version of the transfer factors which is missing the term $\Delta_{IV}$. In order to avoid confusion, we will denote these transfer factors by $\mathring{\Delta}$. The work of Langlands and Shelstad applies equally well to these transfer factors and asserts
\[ \mathring{\Delta}^\tx{desc}_{W,b,\xi}(\tilde \gamma^H_0 \cdot\tilde z,\gamma^b_0\cdot z^b) = \mathring{\Delta}_{W,b}(\gamma^H_0 \cdot z, \gamma^b_0\cdot z^b). \]
We would like to know that this statement also holds when $\tilde z$ and $z^b$ are any strongly-regular semi-simple topologically unipotent elements. This follows from the work of Hales \cite{Hal93}, which, even though originally written for unramified groups and compact elements, applies equally well in our setting, as we will now review.

\begin{lem} \label{lem:ugly} Let $G$ be a tamely ramified connected reductive group defined over $F$, $(H,s,{^L\eta})$ a tamely ramified extended endoscopic triple. Let $\gamma \in G(F)$ be a strongly-regular semi-simple element and $\gamma = \gamma_0 \cdot \gamma_{>0}$ a commuting decomposition with $\gamma_0$ topologically semi-simple modulo $Z(G)^\circ$ and $\gamma_{>0}$ topologically unipotent. Fix a power $Q$ of $p$ for which $\alpha(\gamma_0)^Q=\alpha(\gamma_0)$ for all roots $\alpha$ of the centralizer of $\gamma$. Then, for any absolute transfer factor $\Delta$ and element $\gamma^H \in H(F)$ which transfers to $\gamma$ and has the compatible decomposition $\gamma^H=\gamma^H_0 \cdot \gamma^H_{>0}$, we have
\[ \mr{\Delta}(\gamma^H_0\cdot[\gamma^H_{>0}]^{Q^{2k}},\gamma_0\cdot\gamma_{>0}^{Q^{2k}}) = \mr{\Delta}(\gamma^H,\gamma). \]

\end{lem}
\begin{proof} Let us first compare $\mr{\Delta}([\gamma_H]^{Q^{2k}},\gamma^{Q^{2k}})$ with $\mr{\Delta}(\gamma_H,\gamma)$.
One sees immediately that the factors $\Delta_I$ and $\Delta_{III_1}$ are equal for both pairs $([\gamma_H]^{Q^{2k}},\gamma^{Q^{2k}})$  and $(\gamma_H,\gamma)$, as they only depend on the tori $T^H \subset H$ and $T \subset G$ containing the elements $\gamma^H$ and $\gamma$, as well as the isomorphism $T^H \rw T$ mapping $\gamma^H$ to $\gamma$, and this data remains unaffected when we replace the elements by their powers. It remains to examine the factors $\Delta_{II}$ and $\Delta_{III_2}$. The factor $\Delta_{II}$ is again the same for both pairs of elements, due to \cite[Corollary 10.3]{Hal93}, whose argument remains valid in our setting. The factor $\Delta_{III_2}$ however is different. It is a tamely ramified character of the centralizer of $\gamma$ in $G$ whose restriction to $Z(G)(F)$ is equal to a character $\lambda_G$ that depends only on $G$. This is the content of \cite[\S3.5]{LS90}. Thus the character defining $\Delta_{III_2}$ is trivial on any power of $\gamma_{>0}$. Moreover its values at $\gamma_0^{Q^{2k}}$ and $\gamma_0$ differ by its value at $\gamma_0^{Q^{2k}-1}$, but this element is killed by all roots and hence is central. Thus we see
\[ \Delta_{III_2}([\gamma_H]^{Q^{2k}},\gamma^{Q^{2k}}) = \lambda_G(\gamma_0^{Q^{2k}-1}) \cdot \Delta_{III_2}(\gamma_H,\gamma). \]
Since all the other components of the transfer factor were equal at the two pairs $([\gamma_H]^{Q^{2k}},\gamma^{Q^{2k}})$  and $(\gamma_H,\gamma)$, we conclude that
\[ \mr{\Delta}([\gamma_H]^{Q^{2k}},\gamma^{Q^{2k}}) = \lambda_G(\gamma_0^{Q^{2k}-1}) \cdot \mr{\Delta}(\gamma_H,\gamma). \]
The statement now follows from a second application of \cite[\S3.5]{LS90}.
\end{proof}

\begin{proof}[Proof of Theorem \ref{thm:endotran}]
Consider the right hand side of the claimed equality. If the element $\gamma^H$ does not transfer to $\gamma^b$, the corresponding summand is zero. For any element $\gamma^H$ which does transfer to $\gamma^b$, we have fixed a decomposition $\gamma^H=\gamma^H_0\cdot\gamma^H_{>0}$. Let $Y$ be a set of representatives for the stable classes of the elements $\gamma^H_0$ obtained in this way, with the property that for each $y \in Y$, $H_y$ is quasi-split. Applying Lemma \ref{lem:decdesc}, we can write said right hand side as
\[ \sum_{y \in Y} |\pi_0(H^y)(F)|^{-1} \sum_{z \in H_y(F)_1/\tx{st}} \mr{\Delta}_{W,b}(yz,\gamma^b)\frac{D^H(yz)}{D^{G^b}(\gamma^b)}S\Theta_{\phi^H}(yz). \]
Applying Lemma \ref{lem:ugly}, we can rewrite this as
\[ \sum_{y \in Y} |\pi_0(H^y)(F)|^{-1} \sum_{z \in H_y(F)_1/\tx{st}} \mr{\Delta}_{W,b}(yz^{Q^{2k}},\gamma^b_0(\gamma^b_{>0})^{Q^{2k}})\frac{D^H(yz)}{D^{G^b}(\gamma^b)}S\Theta_{\phi^H}(yz). \]
Taking $k$ large enough, we can apply the Langlands-Shelstad descend theorem \cite[Theorem 1.6]{LS90} and conclude that the above term equals to
\[ \sum_{y \in Y} |\pi_0(H^y)(F)|^{-1} \sum_{z \in H_y(F)_1/\tx{st}} \mr{\Delta}^\tx{desc}_{W,b,\xi}(yz^{Q^{2k}},\gamma^b_0(\gamma^b_{>0})^{Q^{2k}})\frac{D^H(yz)}{D^{G^b}(\gamma^b)}S\Theta_{\phi^H}(yz). \]
Recall that $\xi$ was an element of $\Xi(H_{\gamma^H_0},G^b_{\gamma^b_0})$. This is the unique element of that set for which the pair of elements in $\mr{\Delta}^\tx{desc}_{W,b,\xi}$ is related. Thus the above expression is equal to
\[ \sum_{y \in Y} |\pi_0(H^y)(F)|^{-1} \sum_f \sum_{z \in H_y(F)_1/\tx{st}} \mr{\Delta}^\tx{desc}_{W,b,\xi}(yz^{Q^{2k}},\gamma^b_0(\gamma^b_{>0})^{Q^{2k}})\frac{D^H(yz)}{D^{G^b}(\gamma^b)}S\Theta_{\phi^H}(yz), \]
where $f$ now runs over the set $\Xi(H_{\gamma^H_0},G^b_{\gamma^b_0})$. The sum over $z$ can be extended to $H_y(F)_\tx{sr}/\tx{st}$, since for elements outside of $H_y(F)_1$ the transfer factor will be zero.

Applying Proposition \ref{pro:schar} to $S\Theta_{\phi^H}$, see that the above expression becomes
\begin{eqnarray*}
&&\sum_{y \in Y} |\pi_0(H^y)(F)|^{-1} \sum_\xi \sum_{z \in H_y(F)_\tx{sr}/\tx{st}} \mr{\Delta}^\tx{desc}_{W,b,\xi}(yz^{Q^{2k}},\gamma^b_0(\gamma^b_{>0})^{Q^{2k}})\frac{1}{D^{G^b}(\gamma^b)}\\
&&\sum_{[j^H]}j^H_*[\chi_{S,{^Lj^H}}\cdot\epsilon_{j^H}](y)\sum_{[k^H]} \hat\iota^{J^H}({k^HY},\log(z))
\end{eqnarray*}
We recall that $[j^H]$ runs over the set of $H_y$-stable classes of embeddings $S \rw H_y$ whose composition with $H_y \rw H$ is admissible, while $[k^H]$ runs over the set of $H_y(F)$-conjugacy classes inside of each $[j^H]$.

Applying \cite[Lemma 3.5.A]{LS90} to the factor $\mr{\Delta}_{W,b,\xi}^\tx{desc}$ and rearranging sums, the above expression becomes
\begin{eqnarray}
&&D^{G^b}(\gamma^b)|^{-1}\sum_{y \in Y} |\pi_0(H^y)(F)|^{-1} \sum_\xi \sum_{[j^H]}j^H_*[\chi_{S,{^Lj^H}}\cdot\epsilon_{j^H}](y) \Xi_{W,b,\xi}^\tx{desc}(\gamma^b_0)\nonumber\\
&&\sum_{z \in H_y(F)_\tx{sr}/\tx{st}} \mr{\Delta}^\tx{desc}_{W,b,\xi}(z^{Q^{2k}},(\gamma^b_{>0})^{Q^{2k}})
\sum_{[k^H]} \hat\iota^{J^H}({k^HY},\log(z)), \label{eq:rhs1}
\end{eqnarray}
where we have denoted by $\Xi_{W,b,\xi}^\tx{desc}$ the character of $Z(G^b_{\gamma^b_0})(F)$ which Langlands and Shelstad call $\lambda_G$. After potentially increasing $k$, we have
\[ \mr{\Delta}^\tx{desc}_{W,b,\xi}(z^{Q^{2k}},(\gamma^b_{>0})^{Q^{2k}}) = \Delta^{\mf{h}_y}_{W,b,\xi}(\log(z^{Q^{2k}}),\log((\gamma^b_{>0})^{Q^{2k}})), \]
where the factor on the right is the Lie-algebra transfer factor normalized compatibly with the one on the left.
With this, \eqref{eq:rhs1} becomes
\begin{eqnarray}
&&D^{G^b}(\gamma^b)|^{-1}\sum_{y \in Y} |\pi_0(H^y)(F)|^{-1} \sum_\xi \sum_{[j^H]}j^H_*[\chi_{S,{^Lj^H}}\cdot\epsilon_{j^H}](y) \Xi_{W,b,\xi}^\tx{desc}(\gamma^b_0)\nonumber\\
&&\sum_{Z \in \mf{h}_y(F)_\tx{rs}/\tx{st}} \Delta^{\mf{h}_y}_{W,b,\xi}({Q^{2k}}Z,{Q^{2k}}\log(\gamma^b_{>0}))
\sum_{[k^H]} \hat\iota^{J^H}({k^HY},Z). \label{eq:rhs2}
\end{eqnarray}

The factors $Q^{2k}$ can now be removed, as the Lie-algebra transfer factor is invariant under multiplication by squares from $F^\times$. Let $j : S \rw G^b_{\gamma^b_0}$ be the image of a triple $(y,\xi,j^H)$ in the above summation set under the map $q$ of Lemma \ref{lem:embdesc}. Waldspurger's fundamental result \cite[Thm. 1.2]{Wal95} combined with \cite{Wal97} and \cite{Ngo10} implies that
\[ \sum_{Z \in \mf{h}_y(F)_\tx{rs}/\tx{st}} \Delta^{\mf{h}_y}_{W,b,\xi}(Z,\log(\gamma^b_{>0}))
\sum_{[k^H]} \hat\iota^{J^H}({k^HY},Z) \]
is equal to
\[ \gamma_\psi(\mf{g}^b_{\gamma^b_0})\gamma_\psi(\mf{h}_y)^{-1}\sum_{[k]} \Delta^{\mf{h}_y}_{W,b,\xi}(j^H Y, kY) \hat\iota^{J}(kY,\log(\gamma^b_{>0})), \]
where the sum runs over the set of $G^b_{\gamma^b_0}(F)$-conjugacy classes of embeddings in the $G^b_{\gamma^b_0}$-stable class of $j$. This, combined with the statement of Lemma \ref{lem:embdesc} shows that \eqref{eq:rhs2} is equal to
\[
\sum_{[j]} j^H_*[\chi_{S,{^Lj^H}}\cdot\epsilon_{j^H}](y) \frac{\Xi_{W,b,\xi}^\tx{desc}(\gamma^b_0)}{|D^{G^b}(\gamma^b)|}\frac{\gamma_\psi(\mf{g}^b_{\gamma^b_0})}{\gamma_\psi(\mf{h}_y)}
\!\!\sum_{[k]} \Delta^{\mf{h}_y}_{W,b,\xi}(j^H Y, kY) \hat\iota^{J}(kY,\log(\gamma^b_{>0})).
\]

We now compare this expression with the left hand side of the equality claimed in Theorem \ref{thm:endotran}. Applying Proposition \ref{pro:schar} to $\Theta^s_{\phi,b}$, we see that it is equal to
\[ e(G^b)|D_{G^b}(\gamma)|^{-1}\sum_{[j]}j_*[\chi_{S,{^Lj}}\cdot\epsilon_j](\gamma^b_0)\sum_{[k]} \<\tx{inv}(j_0,k),s\> \hat\iota^{J}(kY,\log(\gamma^b_{>0})). \]

Thus, in order to prove Theorem \ref{thm:endotran}, we need to show the following equality for each embedding $j : S \rw G^b_{\gamma^b_0}$ whose composition with the inclusion $G^b_{\gamma^b_0} \rw G^b$ is admissible.
\begin{eqnarray} \label{eq:rest1}
&j^H_*[\chi_{S,{^Lj^H}}\cdot\epsilon_{j^H}](y) \Xi_{W,b,\xi}^\tx{desc}(\gamma^b_0)\frac{\gamma_\psi(\mf{g}^b_{\gamma^b_0})}{\gamma_\psi(\mf{h}_y)}\Delta^{\mf{h}_y}_{W,b,\xi}(j^H Y, kY)&\nonumber\\
&=&\\
&e(G^b)j_*[\chi_{S,{^Lj}}\cdot\epsilon_j](\gamma^b_0)\<\tx{inv}_b(j_0,j),s\>& \nonumber.
\end{eqnarray}
Recall here that the triple $(y,\xi,j^H)$ is related to $j$ via the map $q$ of Lemma \ref{lem:embdesc}.

As a first step, we undo the descent of the transfer factor. Choosing an element $z \in F^\times$ close to zero, we have
\begin{eqnarray*}
\Xi_{W,b,\xi}^\tx{desc}(\gamma^b_0)\Delta^{\mf{h}_y}_{W,1,\xi}(j^H Y, jY)
&=&\Xi_{W,b,\xi}^\tx{desc}(\gamma^b_0)\Delta^{\mf{h}_y}_{W,1,\xi}(z^2j^H Y, z^2jY)\\
&=&\Xi_{W,b,\xi}^\tx{desc}(\gamma^b_0)\mr{\Delta}^\tx{desc}_{W,1,\xi}(\exp(z^2j^H Y), \exp(z^2jY))\\
&=&\mr{\Delta}^\tx{desc}_{W,b,\xi}(y\exp(z^2j^H Y), \gamma^b_0\exp(z^2jY))\\
&=&\mr{\Delta}_{W,b}(y\exp(z^2j^H Y), \gamma^b_0\exp(z^2jY))\\
&=&\mr{\Delta}_{W,1}(y\exp(z^2j^H Y), \gamma_0\exp(z^2j_0Y))\<\tx{inv}(j_0,j),s\>,\\
\end{eqnarray*}
where $\gamma_0 = j_0j^{-1}(\gamma^b_0)$. With this, we see that Equation \ref{eq:rest1} becomes
\begin{equation}\label{eq:rest2}
\mr{\Delta}_{W,1}(y\exp(z^2j^H Y), \gamma_0\exp(z^2j_0Y))\frac{\gamma_\psi(\mf{g}^b_{\gamma^b_0})}{\gamma_\psi(\mf{h}_y)} =
e(G^b)\frac{j_*[\chi_{S,{^Lj}}\cdot\epsilon_j](\gamma^b_0)}{j^H_*[\chi_{S,{^Lj^H}}\cdot\epsilon_{j^H}](y)}.
\end{equation}

We now need to go into the structure of $\mr{\Delta}_{W,1}$. We can choose an $F$-spitting $\tx{spl} = (T,B,\{X_\alpha\})$ so that the Whittaker datum $W$ comes from $\tx{spl}$ and the additive character $\psi : F \rw \C^\times$ which we fixed in Section \ref{sec:schar}. The centralizers in $H$ and $G$ of the arguments of the transfer factor are $j^HS$ and $j_0S$, and $j_0(j^H)^{-1}$ is an admissible isomorphism between them carrying the first argument of the transfer factor to the second. We also need to choose $a$-data and $\chi$-data, which for now we take arbitrary. Then
\[ \mr{\Delta}_{W,1} = \epsilon_L(X^*(T)_\C-X^*(T^H)_\C,\psi)\Delta_I[\tx{spl},a]\Delta_{II}[a,\chi]\Delta_{III_2}[\chi]. \]
The reason that $\Delta_{III_1}$ is missing lies in our choice of $j_0(j^H)^{-1}$ as the admissible isomorphism. The element $\Delta_I$ does not depend on the elements $y\exp(z^2j^H Y)$ and $\gamma_0\exp(z^2j_0Y)$ directly, but rather only on the tori containing them, which as we remarked are $j^HS$ and $j_0S$. Moreover, these are also the tori centralizing the regular elements $j^HY$ and $j_0Y$. It follows that
\[ \Delta_I[\tx{spl},a](y\exp(z^2j^H Y), \gamma_0\exp(z^2j_0Y)) = \Delta_I[\tx{spl},a](j^HY,j_0Y), \]
with the right hand side as in \cite{Kot99}. The main result of this article shows that if $\tx{spl}_1$ is a splitting whose Kostant section meets the rational class of $j_0Y$ and if as $a$-data we take $\{d\alpha(j_0Y)| \alpha \in R(j_0S,G)\}$, then
\[ \Delta_I[\tx{spl}_1,a](j^HY,j_0Y) = 1. \]
On the other hand, recall from Proposition \ref{pro:generic} that $j_0Y$ meets the Kostant section associated to the regular nilpotent element $E_-$ in the Lie-algebra of the unipotent radical of the $T$-opposite Borel subgroup of $B$, which is specified by the equation
\[ \psi_B(\exp(X)) = \psi\<X,E_-\> \]
for all $X$ in the Lie-algebra of the unipotent radical of $B$. A straight-forward computation shows that
\[ E_- = \sum_{\alpha \in \Delta} \<X_\alpha,X_{-\alpha}\>^{-1} X_{-\alpha}, \]
where $[X_\alpha,X_{-\alpha}] = H_\alpha$. In other words, we have $\tx{spl}_1=\{ \<X_\alpha,X_{-\alpha}\>X_\alpha \}$. Using \cite[Lemma 4.2]{Kal12}, we then see that
\[ \Delta_I[\tx{spl},a](y\exp(z^2j^H Y), \gamma_0\exp(z^2j_0Y)) = 1, \]
provided we take as $a$-data for $R(j_0S,G)$ the expression
\[ a_\alpha = \<X_\alpha,X_{-\alpha}\>^{-1}d\alpha(j_0Y). \]
Turning to $\Delta_{III_2}$, the construction in \cite[\S3.5]{LS87} shows that if we take as $\chi$-data the one specified in Section \ref{sec:chispec}, then
\[ \Delta_{III_2}(y\exp(z^2j^H Y), \gamma_0\exp(z^2j_0Y)) = \frac{j_*[\chi_{S,{^Lj}}](\gamma_0\exp(z^2j_0Y))}{j^H_*[\chi_{S,{^Lj^H}}](y\exp(z^2j^HY))}. \]
Choosing $z$ small enough the right hand side becomes equal to $\frac{j_*[\chi_{S,{^Lj}}](\gamma^b_0)}{j^H_*[\chi_{S,{^Lj^H}}](y)}$.
Taking all this into account, equation \eqref{eq:rest2} becomes
\begin{eqnarray}\label{eq:rest3}
&\epsilon_L(X^*(T)_\C-X^*(T^H)_\C,\psi)\frac{\gamma_\psi(\mf{g}^b_{\gamma^b_0})}{\gamma_\psi(\mf{h}_y)}\Delta_{II}[a,\chi](y\exp(z^2j^H Y), \gamma_0\exp(z^2j_0Y))& \nonumber\\
&=&\\
&e(G^b)\frac{j_*\epsilon_j(\gamma^b_0)}{j^H_*\epsilon_{j^H}(y)}& \nonumber.
\end{eqnarray}

At this point, the $a$-data and $\chi$-data are specified, and we may compute the term $\Delta_{II}$. Recall that it is defined as the product over all roots $\alpha \in R(S,G) \sm R(S,H)$ of the expressions
\[ \chi_\alpha\left( \frac{j_0\alpha(\gamma_0\exp(z^2j_0Y))-1}{a_\alpha} \right). \]
We have chosen $\chi_\alpha=1$ for asymmetric roots $\alpha$, so we only need to compute the above expression for symmetric roots. To ease notation, we will identify $S$ with $j_0S$ via $j_0$ and suppress $j_0$ from the notation.

If $\alpha(\gamma_0)=1$, then we have
\[ \lim_{z \rw 0} \frac{\alpha(\exp(z^2Y))-1}{z^2} = d\alpha(Y), \]
and since $\chi_\alpha$ has order $2$ when restricted to $F^\times$, we see that for $z$ small enough
\[ \chi_\alpha\left( \frac{\alpha(\gamma_0\exp(z^2Y))-1}{a_\alpha} \right) = \chi_\alpha(\<X_\alpha,X_{-\alpha}\>) = \kappa_\alpha(\<X_\alpha,X_{-\alpha}\>). \]
If on the other hand $\alpha(\gamma_0) \neq 1$, then
\[ \lim_{z \rw 0} \alpha(\gamma_0\exp(z^2Y))-1 = \alpha(\gamma_0)-1, \]
and thus for $z$ small enough we have
\[ \chi_\alpha\left( \frac{\alpha(\gamma_0\exp(z^2Y))-1}{a_\alpha} \right) = \chi_\alpha\left(\frac{(\alpha(\gamma_0)-1)\<X_\alpha,X_{-\alpha}\>}{d\alpha(Y)}\right). \]
The element $\alpha(\gamma_0)-1$ is a unit in $F_\alpha$, because $\alpha(\gamma_0)$ is topologically semi-simple. The element $\<X_\alpha,X_\alpha\>$ is a unit in $F_{\pm\alpha}$ by choice of $\<\>$. On the other hand, $Y$ is generic of depth $-1/e$, and hence $d\alpha(Y)^{-1}$ is a uniformizer of $F_\alpha^\times$. All in all, we see that the argument of $\chi_\alpha$ is a uniformizer, which we will call $\omega$ for short.

We now recall the construction of $\chi_\alpha$ from Section \ref{sec:chispec}. If $\alpha$ is inertially asymmetric, then $\chi_\alpha$ is the unramified character sending each uniformizer of $F_\alpha$ to $-1$, and we see
\[ \chi_\alpha(\omega) = -1 = \lambda_{F_\alpha/F_{\pm\alpha}}(\psi\circ\tx{tr}_{F_{\pm\alpha}/F}). \]
If $\alpha$ is inertially symmetric, then we have
\[ \chi_\alpha(\omega) = \lambda_{F_\alpha/F_{\pm\alpha}}(\xi_{\alpha,\omega}^{-1}). \]
Recalling the definition of $\xi_{\alpha,\omega} : k_{F_\alpha} \rw \C^\times$, we see that for $x \in k_{F_\alpha}$, we have
\begin{eqnarray*}
\xi_{\alpha,\omega}(x)&=&\chi_{S,{^Lj}}(N_{F_\alpha/F}(\alpha^\vee(1+\omega x)))\\
&=&\chi_{S,{^Lj}}(1+\tx{Tr}_{F_\alpha/F}(\alpha^\vee(\omega x))) \\
&=&\psi\<\tx{Tr}_{F_\alpha/F}(\alpha^\vee(\omega x)),Y\> \\
&=&\psi\tx{Tr}_{F_\alpha/F}\<\alpha^\vee(\omega x),Y\> \\
&=&\psi\tx{Tr}_{F_\alpha/F}(\omega x\cdot\<H_\alpha,Y\>) \\
&=&\psi\tx{Tr}_{F_\alpha/F}(\omega x\cdot d\alpha(Y)\cdot \<X_\alpha,X_{-\alpha}\>) \\
&=&\psi\tx{Tr}_{F_\alpha/F}(x\cdot(\alpha(\gamma_0)-1) \cdot \<X_\alpha,X_{-\alpha}\>^2). \\
\end{eqnarray*}
The element $\alpha(\gamma_0)$ of $F_\alpha$ is topologically semi-simple and lies in the kernel of the norm to $F_{\pm\alpha}$. Since $F_\alpha/F_{\pm\alpha}$ is ramified, the only such elements are $1$ and $-1$. Being non-trivial, we see that $\alpha(\gamma_0)=-1$, and we obtain
\[ \xi_{\alpha,\omega}(x) = \psi\tx{Tr}_{F_{\pm\alpha}/F}(-4\<X_\alpha,X_{-\alpha}\>^2), \]
from which we conclude that, just like in the inertially asymmetric case, we have
\[ \chi_\alpha(\omega) = \lambda_{F_\alpha/F_{\pm\alpha}}(\psi\tx{Tr}_{F_{\pm\alpha}/F}). \]
Combining these results we see
\[ \Delta_{II} = \prod_{\alpha(\gamma_0) = 1} \kappa_\alpha(\<X_\alpha,X_{-\alpha}\>) \prod_{\alpha(\gamma_0) \neq 1} \lambda_{F_\alpha/F_{\pm\alpha}}(\psi\circ\tx{Tr}_{F_{\pm\alpha}/F}), \]
where in both products $\alpha$ runs over a set of representatives for the symmetric orbits of $\Gamma$ in $R(S,G)-R(S,H)$.

Returning to the proof of equation \eqref{eq:rest3}, we see that combining the above expression for $\Delta_{II}$ with Lemma \ref{lem:torinvweil}, Corollary \ref{cor:weilinner}, and Corollary \ref{cor:torinveps}, shows that \eqref{eq:rest3} is equivalent to
\begin{equation} \label{eq:rest4}
\prod_{H-H_y} f_{(H,j^HS)}(\alpha) \prod_{G-G_{\gamma_0}} f_{(G,j_0S)} e(G_{\gamma_0})e(G^b_{\gamma^b_0}) =
e(G^b)\frac{j_*\epsilon_j(\gamma^b_0)}{j^H_*\epsilon_{j^H}(y)}.
\end{equation}
Here the symbol $H-H_y$ denotes the set of symmetric $\Gamma$-orbits of roots of $S$ in $H$ outside of $H_y$, and the symbol $G-G_{\gamma_0}$ has the analogous meaning. According to Lemma \ref{lem:torinvchar} and Proposition \ref{pro:torinvstab}, we have
\[ j_*e_j(\gamma^b_0) = [j_0]_*e_{j_0}(\gamma_0) \cdot e(G)e(G^b)e(G_{\gamma_0})e(G^b_{\gamma^b_0}). \]
Thus Equation \eqref{eq:rest4} becomes
\begin{equation} \label{eq:rest5}
\prod_{H-H_y} f_{(H,j^HS)}(\alpha) \prod_{G-G_{\gamma_0}} f_{(G,j_0S)} =
\frac{[j_0]_*\epsilon_{j_0}(\gamma_0)}{j^H_*\epsilon_{j^H}(y)}.
\end{equation}
That this last equation holds is implied by Lemma \ref{lem:torinvchar}.
\end{proof}

\section{Epipelagic $L$-packets for $\tx{GL}_n$} \label{sec:epipackgln}

In this section we will show that the construction of epipelagic $L$-packets described in Section \ref{sec:packs} specializes in the case of the group $\tx{GL}_n$ to the local Langlands correspondence established by Harris-Taylor and Henniart. We put $G=\tx{GL}_n$ and then $\hat G = \tx{GL}_n(\C)$. Since $W_F$ acts trivially on $\hat G$, we will use $\hat G$ instead of the $L$-group $\hat G \times W_F$ and will regards Langlands parameters simply as $\hat G$-conjugacy classes of admissible homomorphisms $W_F \rw \hat G$. In order for epipelagic parameters to exist, we must assume that $p \nmid n$.

We will use the explicit description of the local Langlands correspondence for $\tx{GL}_n$ when $p \nmid n$ derived by Bushnell and Henniart in \cite{BH05a,BH05b}.
Central to this description is the notion of an admissible pair of degree $n$, which is a pair $(E/F,\xi)$ consisting of a tamely-ramified extension $E/F$ of degree $n$, and a character $\xi : E^\times \rw \C^\times$, subject to certain conditions. From an admissible pair, one can construct an irreducible $n$-dimensional representation $\sigma_{(E/F,\xi)}$ of $W_F$, and an irreducible supercuspidal representation $\pi_{(E/F,\xi)}$ of $\tx{GL}_n(F)$. The first one is obtained by taking $\sigma := \tx{Ind}_{W_E}^{W_F}\xi$. For the second one, there is an explicit construction,  described in \cite[\S2]{BH05a}, which uses the work of Bushnell and Kutzko \cite{BK93} on the classification of irreducible representations of $\tx{GL}_n(F)$. The following is one of the main results of \cite{BH05a,BH05b}:

\begin{thm}[Bushnell-Henniart]
Assume that $E/F$ is totally ramified. Under the local Langlands correspondence, the representation $\sigma_{(E/F,\xi)}$ corresponds to the representation $\pi_{(E/F,\xi\cdot\mu_\xi)}$, where $\mu_\xi : E^\times \rw \C^\times$ is a character which is explicitly given by \cite[Thm. 2.1]{BH05b}.
\end{thm}
The character $\mu_\xi$ is called the rectifying character of the admissible pair $(E/F,\xi)$.

The construction of this paper associates to a Langlands parameter $\phi : W_F \rw \hat G$ satisfying Conditions \ref{cnd:parm} a triple $(S,\chi_S,[j])$ where $S$ is a tamely ramified $F$-torus, $[j]$ is a stable class of embeddings of $S$ as a maximal torus of $G$, and $\chi_S : S(F) \rw \C^\times$ is a character on $S(F)$ well-defined up to $\Omega(S,G)(F)$. In the situation of $G=\tx{GL}_n$, the set $[j]$ consists of a single $G(F)$-conjugacy class. Thus, we have in fact a pair $(S,\chi_S)$ where $S \subset G$ is a maximal torus and $\chi_S$ is a character of $S(F)$. This pair is well-defined up to $G(F)$-conjugacy. The representation of $G(F)$ associated to $\phi$ is then the one obtained from the pair $(S,\chi_S\cdot \epsilon_f)$ using the construction of Section \ref{sec:repconst}. Here $\epsilon_f : S(F) \rw \C^\times$ is the character built from the toral invariant of $S$ in Section \ref{sec:torinvchar}. In fact, as we will show, it is a special feature of $\tx{GL}_n$ that all toral invariants are trivial. Thus, we may from now on ignore $\epsilon_f$.

The comparison of the construction given in this paper with the local Langlands correspondence for $\tx{GL}_n$ will proceed in the following way. Given a parameter $\phi : W_F \rw \hat G$ satisfying Conditions \ref{cnd:parm}, let $(E/F,\xi)$ be an admissible pair such that $\phi \cong \tx{Ind}_{W_E}^{W_F}\xi$, and let $(S,\chi_S)$ is the pair constructed from $\phi$. The main point is to show that there exists an isomorphism $S(F) \cong E^\times$ which identifies $\chi_S$ with $\xi\cdot\mu_\xi$. This will be shown by first obtaining in Section \ref{sec:admpairs} a Langlands parameter $\phi_\xi : W_F \rw \hat S$ for the character $\xi$ of $E^\times \cong S(F)$, then computing the character $\chi_S \cdot \xi^{-1}$ by examining its parameter $\phi_S \cdot \phi_\xi^{-1}$ in Section \ref{sec:chisxi}, and finally comparing in Section \ref{sec:cmprect} the result of this computation with the formula for the rectifying character given in \cite{BH05b}.

Once this is complete, we will further show in Section \ref{sec:epigln} that the representation of $G(F)$ obtained from the pair $(E/F,\chi_S)$ using the construction of \cite[\S2]{BH05a} is isomorphic to the one obtained from $(S,\chi_S)$ using the construction of Section \ref{sec:repconst}.

\subsection{From epipelagic parameters to admissible pairs} \label{sec:admpairs}

We are going to establish some notation which will be used subsequently and then show how the construction in \ref{sec:lpackconst} leads to admissible pairs in the sense of \cite{BH05a,BH05b}.

Let $\hat T$ be the standard (i.e. diagonal) maximal torus in $\hat G$. We label its coordinates by elements of $\Z/n\Z$, the top left coordinate being $0$, and the bottom right $n-1$. This labeling provides bijections from $\Z/n\Z$ to the standard bases of $X^*(T)$ and $X_*(T)$, and further identifies $\Omega(\hat T)$ with the group $\tx{Bij}(\Z/n\Z)$ of bijections of the set $\Z/n\Z$.

Let $\phi : W_F \rw \hat G$ be a Langlands parameter satisfying Conditions \ref{cnd:parm}. Conjugating by $\hat G$ if necessary we may assume that $\tx{Cent}(\phi(P_F),\hat G)$ is the standard torus $\hat T$. This implies that $\tx{im}(\phi) \subset N(\hat T)$.
The image of the composition
\[ W_F \rw N(\hat T) \rw \Omega(\hat T). \]
is the Galois group of a finite Galois extension $K/F$. Let us study this group.

Let $\tilde F$ be the maximal unramified sub-extension. The assumptions on $\phi$ imply that the extension $K/\tilde F$ is tamely ramified and its Galois group is generated by a Coxeter element, i.e. it is cyclic of order $n$. We choose a generator $s \in \tx{Gal}(K/\tilde F)$. Since all Coxeter elements in $\Omega(\hat T)$ are conjugate, we can further assume after conjugating $\phi$ under $N(\hat T)$ that $s$ is the element $+1 \in \tx{Bij}(\Z/n\Z)$.

We claim that the subgroup $\<s\> \subset \Gamma_{K/F}$ has a complement. Let $Q' \in \tx{Gal}(K/F)$ be a Frobenius element. We have $Q'sQ'^{-1} = s^q$. Let $Q \in \tx{Bij}(\Z/n\Z)$ be the element $\cdot q$. It also has the property $QsQ^{-1}=s^q$. Thus $Q$ differs from $Q'$ by an element of $\tx{Cent}(s,\Omega(\hat T))=\<s\>$, and this shows that $Q \in \Gamma_{K/F}$. Since the subgroups $\<+1\>$ and $\<\cdot q\>$ of $\tx{Bij}(\Z/n\Z)$ have trivial intersection, we see that
\[ \Gamma_{K/F} = \<s\> \rtimes \<Q\> \]
as claimed. We let $E=K^Q$. Observe that the extensions $K/E$ and $\tilde F/F$ are both unramified of equal degree, namely
\[ f = \tx{ord}(q,(\Z/n\Z)^\times). \]
Thus we have the diagram
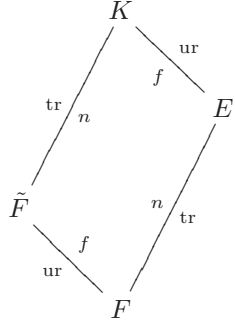
\[ \xymatrix{ &K\ar@{-}[ddl]_{\tx{tr}}^{n}\ar@{-}[dr]^{\tx{ur}}_{f}&\\&&E\\\tilde F\\&F\ar@{-}[ul]^{\tx{ur}}_f\ar@{-}[uur]_{\tx{tr}}^n&} \]
The extension $E/F$ is the only extension of the three that is in general not normal.

Write $\hat S$ for the Galois-module obtained from $\hat T$ by twisting via $\phi$, and let $S$ be the $F$-torus whose dual is $\hat S$. Since $X^*(\hat S)$ is a permutation $W_F$-module and the stabilizer of $\chi_0 \in X^*(\hat S)$, the basis element corresponding to $0 \in \Z/n\Z$, is $W_E$, we have
\[ \hat S = \tx{Ind}_{W_E}^{W_F} \C^\times \qquad\textrm{and hence}\qquad S = \tx{Res}_{E/F} \mb{G}_m. \]

 We consider the map
\[ \tx{strip} : N(\hat T) \rw N(\hat T) \]
which is the composition of the natural projection $N(\hat T) \rw \Omega(\hat T)$ with the standard injection $\Omega(\hat T) \rw N(\hat T)$ given by sending an element of $\Omega(\hat T) \cong S_n$ to the corresponding permutation matrix. One checks that the map
\[ \phi_\xi : W_F \rw \hat S,\qquad w \mapsto \phi(w)\cdot\tx{strip}(\phi(w))^{-1} \]
belongs to $Z^1(W_F,\hat S)$ and hence provides a character $\xi : S(F) \rw \C^\times$. Via the isomorphism $S(F) \cong E^\times$ we may view this as a character on $E^\times$, and hence also as a character on $W_E$. This character is nothing more then the image of $\phi_\xi$ under the isomorphism
$H^1(W_F,\tx{Ind}_{W_E}^{W_F}\C^\times) \cong H^1(W_E,\C^\times)$.

\begin{lem} \label{lem:paircmp} The $W_F$-representations $\phi$ and $\tx{Ind}_{W_E}^{W_F}\xi$ are isomorphic.
\end{lem}
\begin{proof} This is a straightforward computation. \end{proof}

According to Fact \ref{fct:repcmp} and Lemma \ref{lem:paircmp} our task is now to compare the character $\chi_S$ obtained from $\phi$ via factoring through the embedding $^Lj_X$ constructed by means of $\chi$-data (where the $\chi$-data is chosen according to Section \ref{sec:chispec}), with the product $\xi\cdot\mu_{\xi}$, where $\mu_{\xi} : E^\times \rw \C^\times$ is the rectifying character constructed in \cite{BH05a, BH05b}, associated to the admissible pair $(E/F,\xi)$ .

\subsection{Computation of $\chi_S\cdot\xi^{-1}$} \label{sec:chisxi}
Recall that the parameter $\phi_S$ for the character $\chi_S$ is given by $\phi = {^Lj_X}\circ\phi_S$. If we write $^Lj_X(s,w) = s\cdot\xi_X(w)$, then this translates to $\phi_S(w)\cdot\xi_X(w)=\phi(w)$. We conclude that the parameter for $\xi^{-1}\cdot\chi_S$ is given by
\[ \tx{strip}(\phi(w))\cdot\phi(w)^{-1}\cdot\phi(w)\cdot\xi_X(w)^{-1} = \tx{strip}(\phi(w))\cdot\xi_X(w)^{-1}. \]
As remarked in the previous section, under the isomorphism $S(F) \cong E^\times$ the character $\xi^{-1}\cdot\chi_S$ is given by the image of the above parameter under the Shapiro isomorphism $H^1(W_F,\hat S) \cong H^1(W_E,\C^\times)$. This isomorphism is obtained on the level of cochains by restriction to $W_E$ followed by composition with $\chi_0 \in X^*(\hat S)$. In other words, we have for $w \in W_E$
\begin{equation} \chi_S\cdot\xi^{-1}(\tx{Art}(w)) = \chi_0( \tx{strip}(\phi(w))\cdot\xi_X(w)^{-1} ). \label{eq:chidiff1} \end{equation}
Let $(\hat T,\hat B,\{X_{\alpha^\vee}\})$ be the standard splitting of $\hat G$. That is, $\hat B$ is the set of upper triangular matrices, and $\{X_{\alpha^\vee}\}$ consists of the matrices in $\mf{gl}_n(\C)$ all of whose entries are zero except for a unique non-zero entry in the upper off-diagonal. Then we have
\begin{equation} \xi_X(w) = r_{B,X}(w)n(\sigma_w), \label{eq:chidiff2} \end{equation}
where $\sigma_w \in \Omega(\hat T)$ is the projection of $\phi(w)$, $r_{B,X} : W_F \rw \hat T$ is a certain 1-chain, and $n(\sigma_w) \in N(\hat T)$ is a certain lift of $\sigma_w$, both of which are constructed in \cite[\S2]{LS87}.

\begin{pro} \label{pro:stdspr} For any $\alpha \in R(\hat S,\hat G)$, write $\alpha > 0$ if $\alpha \in R(\hat S,\hat B)$, and let $\alpha^\vee = x_\alpha - y_\alpha$, be the unique expression of $\alpha$ in terms of the standard basis of $X_*(\hat S)$. Then for any $w \in W_F$ the following equality holds
\[ \tx{strip}(\phi(w))\cdot n(\sigma_w)^{-1} = \prod_{\substack{\alpha>0 \\ \sigma_w^{-1}\alpha < 0}} y_\alpha(-1). \]
In particular
\[ \chi_0(\tx{strip}(\phi(w))\cdot n(\sigma_w)^{-1}) = 1. \]
\end{pro}
\begin{proof}
The proof of the first statement proceeds by induction on the length of a reduced expression of $\sigma_w$ in terms of $\hat B$-simple reflections. If the length is 1, the statement is immediate. The general case follows from \cite[VI.\S1.no 6. Cor 2]{Bou} and the fact that $\alpha \mapsto y_\alpha$ is $\Omega(\hat T)$-equivariant.

The second statement follows from the fact that $\chi_0(y_\alpha)=1$ for all $\alpha>0$.
\end{proof}

Equations \eqref{eq:chidiff1} and \eqref{eq:chidiff2} and Proposition \ref{pro:stdspr} imply
\begin{equation} \xi^{-1}\cdot\chi_S(\tx{Art}(w)) = \chi_0(r_{B,X}(w)^{-1}) \label{eq:chidiff4} \end{equation}
for all $w \in W_E$, and we turn to computing $r_{B,X}$. To lighten the notation, we will write $\Gamma$ instead of $\Gamma_{K/F}$ for the Galois group of the finite extension $K/F$. Using the bijections between $\Z/n\Z$ and the standard bases of $X^*(\hat T)$ and $X_*(\hat T)$ established in the previous section, the set $R(\hat T,\hat G)$ is identified with the complement of the diagonal in $\Z/n\Z \times \Z/n\Z$. Recall that $\hat S$ is the Galois module whose underlying abelian group is $\hat T$ and whose Galois-structure is given by $\phi$. Thus the set $R(\hat S,\hat G)$ is also identified with the complement of the diagonal in $\Z/n\Z \times \Z/n\Z$, but it has a twisted action of $\Gamma = \<s\> \rtimes \<Q\>$. We write $\Z/n\Z^\dagger$ for the set $\Z/n\Z-\{0\}$. The group $\<Q\>$ acts on $\Z/n\Z^\dagger$ with $Q$ being multiplication by $q$. The group $\<\pm 1\>$ acts on both sides by multiplication. The injection
\begin{equation} \eta : \Z/n\Z^\dagger \rw R(\hat S,\hat G),\qquad z \mapsto (0,z) \label{eq:orbitbij} \end{equation}
is $\<Q\>\times\<\pm 1\>$-equivariant and its image is a cross-section for the set of $\<s\>$-orbits in $R(\hat S,\hat G)$. In particular, it induces a bijection between the set of orbits of $\<Q\>$ on $\Z/n\Z^\dagger$ and the set of orbits of $\<s\>\rtimes\<Q\>$ on $R(\hat S,\hat G)$, which restricts to a bijection between the symmetric orbits on both sides (recall that the symmetric orbits are the ones preserved by $\<\pm 1\>$).

When $n$ is even, there is a unique inertially symmetric orbit in $R(\hat S,\hat G)$ (recall this notion from Section \ref{sec:chispec}), namely the one containing the element $(0,n/2)$. It corresponds to the unique symmetric orbit in $\Z/n\Z^\dagger$ which is a singleton. When $n$ is odd, there is no inertially symmetric orbit in $R(\hat S,\hat G)$ and no singleton symmetric orbit in $\Z/n\Z^\dagger$.

Let $\Xi \subset \Z/n\Z^\dagger$ be a set of representatives for the orbits of $\<Q\> \times \< \pm 1\>$. Then we have
\begin{equation} r_{B,X}(w) = \prod_{a \in \Xi} r_{B,X,a}(w). \label{eq:chidiff5} \end{equation}
We will now describe $r_{B,X,a}$, for a given $a \in \Xi$. Let $\Gamma_{\eta(a)}$ be the subgroup of $\Gamma$ stabilizing $\eta(a)$, and let $\Gamma_{\pm\eta(a)}$ be the subgroup stabilizing the set $\{\eta(a),-\eta(a)\}$. We have the following cases:
\begin{enumerate}
\item If the orbit of $a$ is not symmetric, and $\<Q^m\>$ is the stabilizer of $a$ in $\<Q\>$, then
\[ \Gamma_{\eta(a)} = \Gamma_{\pm\eta(a)} = \<Q^m\>. \]
A set of representatives $\dot\Gamma_{\pm\eta(a)}$ for $\Gamma/\Gamma_{\pm\eta(a)}$ is given by
\[ \dot\Gamma_{\pm\eta(a)} = \{ s^kQ^t| k \in \Z/n\Z, 0 \leq t < m \}. \]
\item If the orbit of $a$ is symmetric but not a singleton, and $\<Q^m\>$ is the stabilizer of the set $\{a,-a\}$ in $\<Q\>$, then
\[ \Gamma_{\eta(a)} = \< Q^{2m} \> \qquad\tx{and}\qquad \Gamma_{\pm\eta(a)} = \< s^aQ^m \>. \]
A set of representatives $\dot\Gamma_{\pm\eta(a)}$ for $\Gamma/\Gamma_{\pm\eta(a)}$ is given by
\[ \dot\Gamma_{\pm\eta(a)} = \{ s^kQ^t| k \in \Z/n\Z, 0 \leq t < m \}. \]
\item If $2|n$ and $a=\frac{n}{2}$, then
\[ \Gamma_{\eta(a)} = \<Q\> \qquad\tx{and}\qquad \Gamma_{\pm\eta(a)} = \<s^\frac{n}{2}\> \rtimes \<Q\>. \]
A set of representatives $\dot\Gamma_{\pm\eta(a)}$ for $\Gamma/\Gamma_{\pm\eta(a)}$ is given by
\[ \dot\Gamma_{\pm\eta(a)} = \{ s^k| 0 \leq k < \frac{n}{2} \}. \]
\end{enumerate}
The assignment
\[ p(\dot\Gamma_{\pm\eta(a)} \cdot \eta(a))=1 \]
provides a gauge $p : \Gamma\cdot\{\pm\eta(a)\} \rw \{\pm 1\}$ on the $\Gamma \times \{\pm 1\}$-orbit  of $\eta(a)$. We have
\begin{equation} r_{B,X,a}(w) = s_{B/p,a}(\sigma_w) \cdot r_{p,X,a}(w), \label{eq:chidiff6} \end{equation}
where the first term is constructed in \cite[\S 2.4]{LS87}, and the second term in \cite[\S 2.5]{LS87}.
Since the image of $W_E$ in $\Gamma_{K/F}$ is generated by $Q$, it will be enough to compute the value of $\chi_0(s_{B/p}(\cdot))$ at $Q$.

\begin{lem} \label{lem:sbpa} We have
\[ \chi_0(s_{B/p,a}(Q)) = \begin{cases}
-1&,\textrm{ if }\<Q\>\cdot a\textrm{ is symmetric and not singleton}\\
1&,\textrm{ else}
\end{cases} \]
\end{lem}
\begin{proof}
Recall that $s_{B/p,a}(\sigma)$ is equal to the product of $\lambda(-1)$, where $\lambda \in X_*(\hat S)$ belongs to the $\Gamma$-orbit of $\eta(a)$ and satisfies one of the two following conditions
\[ \{ \lambda>0, \sigma^{-1}\lambda<0, p(\lambda)=1, p(\sigma^{-1}\lambda)=1 \} \]
\[ \{ \lambda>0, \sigma^{-1}\lambda>0, p(\lambda)=-1, p(\sigma^{-1}\lambda)=1 \}. \]
The elements $\lambda>0$ in the $\Gamma$-orbit of $\eta(a)$ which pair non-trivially with $\chi_0$ are precisely the $\<Q\>$-orbit of $\eta(a)$. They never meet the first condition, because they fail $Q^{-1}\lambda<0$. We consider the second condition. If the orbit is symmetric and singleton, then $\lambda>0 \Leftrightarrow p(\lambda)=1$. If the orbit is asymmetric, then $\Gamma$ preserves $p$. In either case, we see that no $\lambda$ meets the second condition.

If the orbit of $a$ is symmetric and not a singleton, then an element of the form $\lambda = Q^t\eta(a)$ satisfies $p(\lambda)=-1$, $p(Q^{-1}\lambda)=1$ if and only if $t=m$. It follows that $\chi_0(s_{B/p}(Q))=\chi_0(Q^m\eta(a)(-1))=-1$.
\end{proof}

It remains to evaluate $\chi_0(r_{p,X,a}(w))$. This is where the $\chi$-data $X$ for the action of $\Gamma$ on $R(\hat S,\hat G)$, chosen according to Section \ref{sec:chispec}, enters. Let $F_+$ and $F_\pm$ be the fixed fields in $K$ of the groups $\Gamma_{\eta(a)}$ and $\Gamma_{\pm\eta(a)}$ respectively.

\begin{lem} \label{lem:rpa1} If $\<Q\>\cdot a$ is either asymmetric or symmetric and not a singleton, then
\[ r_{p,X,a}(w)=1. \]
\end{lem}
\begin{proof}
Let $\chi_a : F_+^\times \rw \C^\times$ be the character associated to $\eta(a)$ in Section \ref{sec:chispec}. If $\<Q\> \cdot a$ is asymmetric, this character is trivial and the statement follows immediately. In the second case, this is the unique non-trivial unramified quadratic character of $F_+^\times$. We enlarge the field $K$ if necessary to ensure it contains the unramified quadratic extension of $F_+$, which we call $F_2$.  Then $\chi_a$ kills the norm subgroup $N(F_2^\times)$, and under the Artin reciprocity map $W_{F_+}^\tx{ab} \rw F_+^\times$ it provides a character on $W_{F_+}$ that factors through the quotient $\Gamma_{F_2/F_+}$. We now use the formula for $r_{p,X,a}$ from \cite[\S2.5]{LS87}. In our situation all computation can be performed in $\Gamma_{K/F}$ instead of $W_{K/F}$. We have for $\sigma \in \Gamma_{K/F}$
\[ r_{p,X,a}(\sigma) = \prod_{t = 0}^{m-1}\prod_{k \in Z/n\Z} \chi_a( v_0(u_{k,t}(\sigma)))^{s^kQ^t\eta(a)} \]
where $u_{k,t}(\sigma)$ is the unique element of $\Gamma_{F_\pm/F}$ of the form
\[ Q^{-t}s^{-k}\sigma s^l Q^s,\qquad l \in \Z/n\Z,0 \leq s < m \]
and for $\tau \in \Gamma_{F_\pm/F}$ the element $v_0(\tau)$ is the unique one in $\Gamma_{F_+/F}$ of the form
\[ \tau (s^aQ^m)^\epsilon,\qquad 0 \leq \epsilon \leq 1. \]
It is enough to compute $r_{p,X,a}(\sigma)$ for $\sigma=s$ and $\sigma=Q$. We claim that in both these cases the value $\chi_a(v_0(u_{k,t}(\sigma)))$
is independent of $k$. In the first case one has $u_{k,t}(s) \in I_{F_\pm/F}$, hence $v_0(u_{k,t}(s))=u_{k,t}(s) \in I_{F_+/F}=I_{F_\pm/F}$, which belongs to the kernel of $\chi_a$. In the second case one has
\[
\begin{cases}
u_{k,t}(Q)=1&, t>0\\
u_{k,0}(Q)=s^aQ^m&,t=0
\end{cases}%
\Longrightarrow%
\begin{cases}
v_0(u_{k,t}(Q))=1&, t>0\\
v_0(u_{k,0}(Q))=Q^{2m}&,t=0
\end{cases}%
\]
This shows that indeed the value $\chi_a(v_0(u_{k,t}(\sigma)))$ for $\sigma=s,Q$ is independent of $k$. Then we obtain
\[ r_{p,X,a}(\sigma) = \prod_{t = 0}^{m-1}\chi_a( v_0(u_{0,t}(\sigma)))^{Q^t\sum_{k \in Z/n\Z}s^k\eta(a)} \]
Since $\eta(a)$ belongs to the coroot lattice of $\hat S$ and $s\in  \Omega(\hat S)$ is an elliptic element, we conclude that $r_{p,X,a}(\sigma)=1$.
\end{proof}

\begin{lem} \label{lem:rpa2} If $\<Q\>\cdot a$ is the unique symmetric singleton orbit, then for all $w \in W_E$ we have
\[ \chi_0(r_{p,X,a}(w)) = \chi_a(\tx{Art}(w)), \]
where $\chi_a$ is the character on $F_+^\times=E^\times$ associated to the root $\eta(a)$ in Section \ref{sec:chispec}.
\end{lem}

\begin{proof}
We are in the situation $2|n$ and $a=n/2$, thus $\eta(a)=(0,n/2)$. The only factor in the product defining $r_{p,X,a}(w)$ which pairs non-trivially with $\chi_0$ is the factor $\chi_a(v_0(u_0(w)))^{\eta(a)}$, and we have
\[ \chi_0( \chi_a(v_0(u_0(w)))^{\eta(a)} ) = \chi_a(v_0(u_0(w))). \]
Since $W_E=W_{F^+}$, we have $v_0(u_0(w))=w$.
\end{proof}

We are now ready to state the final formula for the character $\chi_S\cdot\xi^{-1}$ of $E^\times$. Let $v_E$ be the valuation on $E^\times$ which sends uniformizers to $1$. We will write $\tx{sgn}(q,\Z/n\Z)$ for the sign of the permutation which multiplication by $q$ induces on the set $\Z/n\Z$. If $2|n$, then $\eta(n/2)$ is a representative for the unique inertially symmetric orbit of $\Gamma$ in $R(\hat S,\hat G)$. In that case we have $F_+=E$. Let $\chi_{\eta(n/2)}$ be the character of $E^\times$ constructed in Section \ref{sec:chispec} corresponding to the root $\eta(n/2)$.

\begin{pro} \label{pro:chisxi} For $e \in E^\times$ we have
\[ \chi_S\cdot\xi^{-1}(e) = \begin{cases}
\tx{sgn}(q,\Z/n\Z)^{v_E(e)}&, \textrm{ if } 2 \nmid n\\
\tx{sgn}(q,\Z/n\Z)^{v_E(e)}\cdot \chi_{\eta(n/2)}(e)^{-1}&, \textrm{ if } 2|n
\end{cases}. \]
\end{pro}
\begin{proof} According to Equations \eqref{eq:chidiff4}, \eqref{eq:chidiff5}, \eqref{eq:chidiff6}, we have
\[ \chi_S \cdot \xi^{-1}(\tx{Art}(w)) = \prod_{a \in \Xi} \chi_0(s_{B/p,a}(\sigma_w)\cdot r_{p,X,a}(w))^{-1}. \]
We write $\Xi$ as the disjoint union
\[ \Xi = \Xi_\tx{a} \sqcup \Xi_\tx{sn} \sqcup \Xi_\tx{ss}, \]
according to whether an element $a \in \Xi$ represents a asymmetric, a symmetric and non-singleton, or a symmetric and singleton orbit. This is equivalent to $\eta(a)$ representing an asymmetric, symmetric but inertially asymmetric, or an inertially symmetric orbit of $\Gamma$ in $R(\hat S,\hat G)$. Note that $\Xi_\tx{ss}$ is empty when $2 \nmid n$ and contains exactly one element of $2|n$.

Then according to Lemmas \ref{lem:rpa1} and \ref{lem:rpa2} we have
\[ \chi_S \cdot \xi^{-1}(e) = \prod_{a \in \Xi_\tx{sn}} (-1)^{v_E(e)} \prod_{a \in \Xi_\tx{ss}} \chi_a(e)^{-1} \]
Since $\tx{sgn}(q,\Z/n\Z)$ is equal to the parity of the number of symmetric orbits of even size in $\Z/n\Z$, which is equal to the number of symmetric orbits of even size in $\Z/n\Z^\dagger$, we see that the first product is equal to $\tx{sgn}(q,\Z/n\Z)^{v_E(e)}$.

\end{proof}

\subsection{A comparison with the rectifying character} \label{sec:cmprect}

The purpose of this section is to prove the following:

\begin{thm} \label{thm:recti}
\[ \chi_S = \xi \cdot \mu_{\xi}. \]
\end{thm}

\begin{proof}

If $2 \nmid n$, then the theorem follows at once from Proposition \ref{pro:chisxi} above, Theorem 2.1(1) of \cite{BH05b}, and the well known observation of Zolotarev that $\tx{sgn}(q,\Z/n\Z)$ is equal to the Jacobi symbol $(q/n)$ if $n$ is odd.

For the remainder of the proof, we assume $2|n$. Recall that we have the admissible pair $(E/F,\xi)$ with the property that the $W_F$-representation $\tx{Ind}_{W_E}^{W_F}\xi$ is isomorphic to the one given by the epipelagic parameter $\phi : W_F \rw \hat G$. In particular, $E/F$ is totally tamely ramified. Working through section 1 of \cite{BH05b} we gather the following data:
\[ l=1,E_0=E,E_1=F,d_0=[E:F]=n,d_1=1,S(\xi)=\{1\},i^+=i_+=1,d^+=1. \]
Let $\psi_F : F \rw \C^\times$ be a character trivial on $\mf{p}_F$ but non-trivial on $O_F$.  For every field extension $K/F$ we set $\psi_K = \psi_F \circ \tx{Tr}_{K/F}$. Then, for each tower $L/K/F$ we have the Langlands constant
\[ \lambda_{L/K} = \frac{\epsilon(\tx{Ind}_{\Gamma_L}^{\Gamma_K} 1_L,\frac{1}{2},\psi_K)}{\epsilon(1_L,\frac{1}{2},\psi_L)}. \]

Since $\xi|_{U_E^2}=1$, we can choose an element $\alpha \in E^\times$ with valuation $\tx{val}_E(\alpha)=-1$ and the property that
\[ \xi|_{U_E^1}(x) = \psi_E(\alpha(x-1)). \]
Given any uniformizing element $\omega \in E^\times$ we put
\[ \zeta(\omega,\xi) = w\alpha \in k_E. \]
Since $\alpha$ is well defined up to multiplication by $U^1_E$, this element does not depend on the choice of $\alpha$. Theorem 2.1(2) of \cite{BH05b} then asserts that
\[ \mu_\xi(\omega) = \left(\frac{\zeta(\omega,\xi)}{q}\right)\cdot\lambda_{E/F}.\]
Note that both factors depend on the arbitrary character $\psi_F$, but their product is independent of $\psi_F$.

We want to show that this formula agrees with the formula for $\chi_S \cdot \xi^{-1}$ given in Proposition \ref{pro:chisxi}.

\begin{lem} Let $E_2$ be the subfield of $K$ fixed by the subgroup $\{1,s^\frac{n}{2}\} \times \<Q\>$ of $\Gamma_{K/F}$. Then
\[ \lambda_{E/F} = \tx{sgn}(q,\Z/n\Z) \cdot \lambda_{E/E_2}. \]
\end{lem}

\begin{proof}

The extension $E/E_2$ is quadratic and ramified, and we have (see e.g. \cite[\S1.5]{BH05b})
\[ \lambda_{E/F} = \lambda_{E/E_2} \cdot \lambda_{E_2/F}^2. \]
If $[E_2:F]$ is odd, then $\lambda_{E_2/F}^2$ is trivial. If $[E_2:F]$ is even, then $\Gamma_{K/F}$ contains the normal subgroup $\{1,s^\frac{n}{4},s^\frac{n}{2},s^\frac{3n}{4}\}$.
We claim that this group preserves $E_2$. Indeed, $E_2$ is the subfield of $K$ fixed by both $Q$ and $s^\frac{n}{2}$ and one sees right away that if an element $e$ has this property, then so does $s^\frac{n}{4}e$. Let $E_4$ be the subfield of $E_2$ fixed by $s^\frac{n}{4}$. Then we have
\[ \lambda_{E_2/F}^2 = \lambda_{E_2/E_4}^2 = \left(\frac{-1}{q}\right).\]
All in all, we obtain
\[ \lambda_{E/F} = \left(\frac{-1}{q}\right)^{\frac{[E:F]}{2}-1} \cdot \lambda_{E/E_2}. \]
The lemma now follows from \cite[Thm. 1]{DH05}.
\end{proof}

We conclude that
\[ \mu_\xi(\omega) = \tx{sgn}(q,\Z/n\Z)\cdot\left(\frac{\zeta(\omega,\xi)}{q}\right)\cdot\lambda_{E/E_2} \]
and focus on the product of the second and third factors. Recall that these factors depend on the arbitrary character $\psi_F$, with $\zeta(w,\xi)$ involving $\psi_E = \psi_F \circ \tx{Tr}_{E/F}$, and $\lambda_{E/E_2}$ involving furthermore $\psi_{E_2} = \psi_F \circ \tx{Tr}_{E_2/F}$. All of these characters enter into the above objects through their reduction to $k_E=k_{E_2}=k_F$, and for these reductions we have
\[ \psi_E(x) = \psi_F(nx), \qquad \psi_{E_2}(x) = \psi_F(\frac{n}{2}x),\qquad\qquad x \in k_F.  \]
We choose $\psi_F$ so that its reduction satisfies $\psi_F(nx) = \xi(\omega x+1)$ for all $x \in k_F$. Then by construction we have $\zeta(\omega,\xi)=1$. The following lemma completes the proof of the theorem:

\begin{lem} With the above choice of $\psi_F$, we have
\[ \lambda_{E/E_2} = \chi_{\eta(n/2)}(\omega)^{-1}. \]
\end{lem}

\begin{proof}
According to the construction of Section \ref{sec:chispec}, we have $\chi_{\eta(n/2)}(\omega)^{-1}=\lambda_{E/E_2}(\xi_{\alpha,\omega})$, where $\xi_{\alpha,\omega}$ is the character on $k_{E_2} = k_E$ obtained as the composition
\[ \xi_{\alpha,\omega} : \xymatrix{ k_{F_\alpha}\ar[rr]^-{x\mapsto \omega x+1}&&U_{F_\alpha}^1/U_{F_\alpha}^2\ar[r]^-{\xi_\alpha}&\C^\times}, \]
where $\xi_\alpha : U_E^1 \rw \C^\times$ is dual to the homomorphism
\[ \xymatrix{ P\ar@{^{(}->}[r]&W\ar[r]^\phi&\hat T\ar[r]^\alpha&\C^\times}, \]
and $\alpha = \eta(n/2)$. Our goal is to show that the characters $\xi_{\alpha,\omega}$ and $x \mapsto \xi(\omega 2^{-1}x+1)$ of $k_{E_2}$ are translates under $k_{E_2}^2$.

We claim that $\xi_\alpha : U_E^1 \rw \C^\times$ is the restriction of the character on $E^\times$ given by
\[ x \mapsto \xi(x \cdot x^{-\tau}) \]
where $\tau$ is the non-trivial element of $\Gamma_{E/E_2}$. Indeed, the homomorphisms $\phi$ and $\phi_\xi$ have the same restriction to $P$, so $\xi_\alpha$ is the restriction to $U_E^1$ of the character of $E^\times$ given by the composition
\[ \xymatrix{ E^\times\ar[r]^{\tx{Art}^{-1}}&W_E^\tx{ab}\ar[r]^{\phi_\xi}&\hat T\ar[r]^\alpha&\C^\times}, \]
while, according to the Shapiro isomorphism, $\xi$ is the character of $E^\times$ given by the composition
\[ \xymatrix{ E^\times\ar[r]^{\tx{Art}^{-1}}&W_E^\tx{ab}\ar[r]^{\phi_\xi}&\hat T\ar[r]^{\chi_0}&\C^\times}. \]
The claim follows from the fact that $\alpha=\chi_0 - s^\frac{n}{2}\cdot\chi_0$, and $s^\frac{n}{2}$ restricts to $E$ as the non-trivial automorphism preserving $E_2$.

The claim implies $\xi_{\alpha,\omega}(x)=\xi(\omega 2 x +1)$, and the proof is complete.
\end{proof}

\end{proof}

\subsection{Epipelagic representations for $\tx{GL}_n$} \label{sec:epigln}

Consider an embedding $j : S \rw G$ and a character $\chi : S(F) \rw \C^\times$ subject to Conditions \ref{cnd:char}. The following facts are easily checked.

\begin{fct}
We have $S(F) \cong E^\times$ for a totally ramified extension $E/F$ of degree $n$. Moreover, $\chi : E^\times \rw \C^\times$ is generic in the sense of Kutzko \cite[Def. 2.2.3]{My86}.
\end{fct}

\begin{fct} \label{fct:repcmp} The representation obtained from $(jS,j\chi_S)$ coincides with the representation obtained from the admissible pair $(E^\times,\chi)$ in \cite[\S2.3]{BH05a}.
\end{fct}

\begin{lem} The toral invariant $f_{jS,G}$ is trivial.
\end{lem}
\begin{proof}
We replacing $j$ by a $G(F)$-conjugate embedding we can find $g \in G$ such that $g^{-1}{^\sigma g}$ is a permutation matrix for any $\sigma \in \Gamma$. One then computes immediately that with respect to the standard splitting of $\tx{GL_n}$, the element $f_{jS,G}^\tx{coh}(X_\alpha)$ is trivial.
\end{proof}

\end{document}